\documentclass{amsart}
\usepackage{amsmath,amssymb,amsthm,a4wide,graphicx,mathtools}

\usepackage[usenames, dvipsnames]{color}

\usepackage{xypic} 
\usepackage{mathrsfs} 
\usepackage{xspace}
\usepackage[shortlabels]{enumitem}
\usepackage{hyperref}

\newtheorem{thm}{Theorem}[section] 

\newtheorem{obs}[thm]{Observation}
\newtheorem{prop}[thm]{Proposition}
\newtheorem{lem}[thm]{Lemma}
\newtheorem{cor}[thm]{Corollary}

\theoremstyle{definition}
\newtheorem{defn}[thm]{Definition}

\newtheorem{rmk}[thm]{Remark}
\newtheorem{claim}{Claim}
\newtheorem*{claim*}{Claim}
\newtheorem*{ack}{Acknowledgments}

\newtheorem{notation}[thm]{Notation}

\newcommand{\mc}[1]{\mathcal{#1}}
\newcommand{\mb}[1]{\mathbb{#1}}
\newcommand{\mf}[1]{\mathfrak{#1}}

\newcommand{\Zb}{\mathbb{Z}}
\newcommand{\Nb}{\mathbb{N}}
\newcommand{\Rb}{\mathbb{R}}
\newcommand{\Qb}{\mathbb{Q}}

\newcommand{\Norm}{\mathcal{N}}

\newcommand{\tdlc}{t.d.l.c.\@\xspace}

\newcommand{\defbold}{\textbf}

\newcommand{\acts}{\curvearrowright}

\newcommand{\normal}{\trianglelefteq}

\newcommand{\inv}{^{-1}}
\newcommand{\triv}{\{1\}}

\newcommand{\CC}{\mathrm{C}}
\newcommand{\N}{\mathrm{N}}
\newcommand{\Z}{\mathrm{Z}}
\newcommand{\LL}{\mathrm L}
\newcommand{\Tw}{\mc{T}_{\aleph_0}}

\newcommand{\pwr}{\mathcal{P}}

\newcommand{\rest}{\upharpoonright}

\newcommand{\injects}{\hookrightarrow}
\newcommand{\GH}{G\rtimes_{\psi}H}

\newcommand{\la}{\left\langle}
 \newcommand{\ra}{\right\rangle}

\newcommand{\cgrp}[1]{\overline{\langle #1 \rangle}}
\newcommand{\ngrp}[1]{\overline{\langle\langle #1\rangle \rangle}}
\newcommand{\ngrpd}[1]{\langle\langle #1\rangle \rangle}
\newcommand{\grp}[1]{\langle #1 \rangle}
\newcommand{\ol}[1]{\overline{#1}}

\newcommand{\Sym}{\mathop{\rm Sym}\nolimits}
\newcommand{\FSym}{\mathop{\rm FSym}\nolimits}
\newcommand{\Aut}{\mathop{\rm Aut}\nolimits}
\newcommand{\Inn}{\mathop{\rm Inn}\nolimits}

\newcommand{\wt}[1]{\widetilde{#1}}
\newcommand{\wh}[1]{\widehat{#1}}

 \newcommand{\abs}[1]{\left\lvert #1\right\rvert}
 \newcommand{\norm}[1]{\left\lVert #1\right\rVert}

\makeatletter
\let\@wraptoccontribs\wraptoccontribs
\makeatother

\makeindex
\begin{document}

\title[Chief factors in Polish groups]{Chief factors in Polish groups}
\author{Colin D. Reid}
\address{ University of Newcastle,
   School of Mathematical and Physical Sciences,
   University Drive,
   Callaghan NSW 2308, Australia}
\email{colin@reidit.net}
\thanks{The first named author was an ARC DECRA fellow.  Research supported in part by ARC Discovery Project DP120100996.}

\author{Phillip R. Wesolek}
\address{ }
\email{prwesolek@gmail.com}
\thanks{The second named author was supported by ERC grant \#278469.}

\dedicatory{\MakeUppercase{\textnormal{With an appendix by Fran\c{c}ois le Ma\^{i}tre}}}
\address{ Institut de Math\'{e}matiques de Jussieu-PRG,
  Universit\'{e} Paris Diderot, 
  Sorbonne Paris Cit\'{e}, 
  75205 Paris cedex 13, France}
 \email{francois.le-maitre@imj-prg.fr }

\begin{abstract}
In finite group theory, chief factors play an important and well-understood role in the structure theory. We here develop a theory of chief factors for Polish groups. In the development of this theory, we prove a version of the Schreier refinement theorem. We also prove a trichotomy for the structure of topologically characteristically simple Polish groups.

The development of the theory of chief factors requires two independently interesting lines of study. First we consider injective, continuous homomorphisms with dense normal image. We show such maps admit a canonical factorization via a semidirect product, and as a consequence, these maps preserve topological simplicity up to abelian error. We then define two generalizations of direct products and use these to isolate a notion of semisimplicity for Polish groups.  
\end{abstract}

\maketitle


\tableofcontents

\addtocontents{toc}{\protect\setcounter{tocdepth}{1}}

\section{Introduction}
A Polish group is a topological group so that the topology is separable and admits a complete, compatible metric. A chief factor of a Polish group $G$ is a factor $K/L$ where $L<K$ are distinct closed normal subgroups of $G$ that admit no closed $M\normal G$ with $L<M<K$. Such factors play an important and well-understood role in the setting of finite groups. They also arise naturally and are important in the study of compactly generated totally disconnected locally compact groups via work of P-E. Caprace and N. Monod \cite{CM11}; however, they are not well-understood in this setting. 

The present article arose out of a project to develop a theory of chief factors for locally compact Polish groups. It quickly became apparent that the basic components of the theory do not require local compactness but only that the groups are Polish. Moreover, this theory appears to be applicable to and useful for the study of many Polish groups. The work at hand therefore develops the theory of chief factors and the requisite tools in the setting of Polish groups. 

The theory of chief factors developed in this article relies on two independently interesting directions of study. The first is a study of normal compressions: these are continuous group homomorphisms $\psi:G\rightarrow H$ which are injective with dense normal image.  Such maps appear naturally in the study of closed normal subgroups of a Polish group because of the deficiencies of the second isomorphism theorem of group theory in the setting of topological groups, including Polish groups.  Specifically, if $N$ and $M$ are closed normal subgroups of the Polish group $G$, then $NM$ is not necessarily closed, and so the obvious homomorphism $\phi: N/N\cap M\rightarrow \ol{NM}/M$ is not necessarily surjective; however, $\phi$ is still a normal compression.  Normal compressions are also related to the study of Polishable subgroups, cf. \cite{FS06}: if $K\normal H$ is a Polishable dense normal subgroup, then there is a normal compression $\psi: G \rightarrow H$ of Polish groups with $\psi(G) = K$.


We additionally investigate two generalizations of direct products: generalized central products and quasi-products. The concept of a quasi-product (with finitely many factors) is introduced by Caprace and Monod in \cite{CM11}, where it is shown that quasi-products play a critical and unavoidable role in the structure theory of compactly generated totally disconnected locally compact (\tdlc) groups.  Specifically, the following situation arises, which is by no means limited to locally compact groups: Let $G$ be a topological group and suppose that the collection of minimal closed normal subgroups $\mc{M}$ is non-empty and consists of non-abelian groups. One often wishes to study the topological socle $S := \cgrp{\mc{M}}$. The elements of $\mc{M}$ centralize each other and have trivial intersection, so as an abstract group, $S$ contains the direct sum $\bigoplus_{M\in\mc{M}}M$.  The induced topology of $\grp{\mc{M}}$ as a subgroup of $G$ however may be quite different from the product topology, even when $|\mc{M}|$ is finite. Our notion of a generalized central product will include such groups $S$ along with many others.

\begin{rmk}
This article is the first in a series; the subsequent articles, which specialize to the case of locally compact Polish groups, are \cite{RW_LC} and \cite{RW_EC}.  In locally compact Polish groups, the compactly generated open subgroups play a special role in the structure theory; this is used to deduce that, whenever a locally compact Polish group $G$ has sufficient complexity as a topological group (in a sense that can be made precise using a rank function), that complexity must be manifested in chief factors.  The authors are not aware of any structure in general Polish groups that could play a role analogous to compactly generated open subgroups of locally compact groups; it seems likely that there are Polish groups with a very complex normal subgroup structure, but no chief factors.  However, in the case of non-Archimedean Polish groups, which occur as automorphism groups of countable structures, potentially there are properties of the countable structure that imply the existence of chief factors, and then the present article gives tools to study these chief factors.  Subsection~\ref{ex:stacking} below gives examples of chief factors arising from an action on a countable structure.  Chief blocks could also be useful in studying groups that are known to have many quotients for other reasons, for instance Polish groups that are surjectively universal for some class of Polish groups (or for all Polish groups; see \cite{D12}).

At the other extreme, the theory of chief factors and blocks is trivial when the group of interest is topologically simple, as is the case for several well-studied examples of Polish groups.  Even in this context though, many interesting examples are not simple as abstract groups, that is, they have proper dense normal subgroups.  In such a situation, it is natural and useful to consider normal compressions.  See the appendix for examples.
\end{rmk}

\subsection{Normal compressions}

\begin{defn}
Let $G$ and $H$ be topological groups. A continuous homomorphism $\psi:G\rightarrow H$ is a \defbold{normal compression}\index{normal compression} if it is injective with a dense, normal image.  We will refer to $H$ as a normal compression of $G$ in cases where the choice of $\psi$ is not important.
\end{defn}

Let $\psi: G \rightarrow H$ be a normal compression between topological groups $G$ and $H$ and let $H$ act on itself by conjugation.  As abstract groups, there is a unique action of $H$ on $G$ by automorphisms, which we call the \defbold{$\psi$-equivariant action}\index{$\psi$-equivariant action}, that makes $\psi$ an $H$-equivariant map, that is, such that $\psi(h.g) = h\psi(g)h\inv$ for all $g \in G$ and $h \in H$.  If $G$ and $H$ are Polish, this action is well-behaved, and we have good control over the structure of the associated semidirect product.

\begin{thm}[See Proposition~\ref{prop:compression_action_cont} and Theorem~\ref{thm:psi-compression_factor_rel}]\label{thmintro:psi-compression_factor} Let $G$ and $H$ be Polish groups, $\psi: G \rightarrow H$ be a normal compression, and $G \rtimes_{\psi} H$ be the associated semidirect product equipped with the product topology.
\begin{enumerate}[(1) ]
\item $G \rtimes_{\psi} H$ is a Polish topological group.
\item $\pi:G\rtimes_{\psi}H\rightarrow H$ via $(g,h)\mapsto \psi(g)h$ is a continuous surjective homomorphism, $\ker(\pi)=\{(g\inv,\psi(g))\mid g\in G\}$, and $\ker(\pi)\simeq G$ as topological groups.
\item $\psi = \pi \circ \iota$ where $\iota: G \rightarrow G\rtimes_{\psi}H$ is the usual inclusion.
\item $G\rtimes_{\psi}H =\ol{\iota(G)\ker (\pi)}$, and the subgroups $\iota(G)$ and $\ker(\pi)$ are closed normal subgroups of $G\rtimes_{\psi}H$ with trivial intersection.
\end{enumerate}
\end{thm}


This factorization imposes a considerable structural restriction on normal compressions in the category of Polish groups, which allows us to relate the normal subgroup structure and topological structure of $G$ to that of $H$ and vice versa.  The next theorem is a special case of a result of this kind.

\begin{thm}[See Theorem~\ref{thm:compression_simple}]\label{thmintro:compression_simple}
Suppose that $G$ and $H$ are centerless Polish groups, each with dense commutator subgroup, and that $H$ is a normal compression of $G$.  Then $G$ is topologically simple if and only if $H$ is topologically simple.
\end{thm}

\subsection{Generalizations of direct products and groups of semisimple type}
Suppose that $G$ is a topological group and that $J \subseteq \Norm(G)$, where $\Norm(G)$ is the collection of closed normal subgroups of $G$.  Set $G_J :=  \cgrp{N \mid N \in J}$.

\begin{defn}
We say that $\mc{S}\subseteq \Norm(G)$ is a \defbold{generalized central factorization}\index{generalized central factorization} of $G$, if $\triv \not\in \mc{S}$, $G_{\mc{S}}=G$ and $[N,M] = \triv$ for any $N\neq M$ in $\mc{S}$. In such a case, $G$ is said to be a \defbold{generalized central product}\index{generalized central product}.

We say $\mc{S}$ is a \defbold{quasi-direct factorization}\index{quasi-direct factorization} of $G$ if $\triv \not\in \mc{S}$, $G_{\mc{S}}=G$, and $\mc{S}$ has the following topological independence property:
\[
\forall X\subseteq \mathcal{P}(\mc{S}):\; \bigcap X=\emptyset \Rightarrow \bigcap_{A\in X}G_A=\{1\}. 
\]
In such a case, $G$ is said to be a \textbf{quasi-product}\index{quasi-product}. A subgroup $H$ of $G$ is a \defbold{quasi-factor}\index{quasi-factor} of $G$ if it is an element of some quasi-direct factorization. The factorizations $\mc{S}$ are \textbf{non-trivial}\index{non-trivial factorization} if $|\mc{S}|\ge 2$.
\end{defn}
\begin{rmk}This notion of quasi-product generalizes the one given in \cite{CM11}, which considered quasi-products with finitely many factors and trivial center.
\end{rmk}
Direct products are obviously quasi-products. Furthermore, we already see a general circumstance in which quasi-products that are not direct products occur: in Theorem~\ref{thmintro:psi-compression_factor}, the semidirect product $G \rtimes_{\psi} H$ is a quasi-product of the set of closed normal subgroups $\{\iota(G),\ker(\pi)\}$, but it is not a direct product of these quasi-factors unless $\psi$ is surjective.

The notions of generalized central factorization and quasi-direct factorization are equivalent for centerless groups (Proposition~\ref{prop:quasiproduct:centralizers}), but in general a generalized central product is a weaker kind of structure than a quasi-product.  Both kinds of generalized product are related to the direct product by the following theorem, which also serves to motivate the independence property in the definition of quasi-products.

\begin{thm}[See \S\ref{sec:quasiproduct1}]\label{thmintro:directproduct:quasiproduct}
Let $G$ be a topological group.
\begin{enumerate}[(1) ]
\item If $\mc{S}$ is a non-trivial generalized central factorization of $G$, then the diagonal map 
\[
d: G \rightarrow \prod_{N\in \mc{S}}G/G_{\mc{S}\setminus\{N\}}
\]
is a continuous homomorphism such that $d(G) \cap G/G_{\mc{S}\setminus \{N\}}$ is dense in $G/G_{\mc{S}\setminus \{N\}}$ for every $N\in \mc{S}$ and $\ker(d)$ is central in $G$. Furthermore, $\mc{S}$ is a quasi-direct factorization of $G$ if and only if $d$ is injective.

\item Let $\mc{K}$ be a set of topological groups and suppose that $G$ admits an injective, continuous homomorphism $\delta: G \rightarrow \prod_{K\in \mc{K}}K$ such that $\delta(G) \cap K$ is dense in each $K\in \mc{K}$.  For $\mc{S}: = \{\delta\inv(K) \mid K\in \mc{K}\}$, the group $\cgrp{\mc{S}}$ is a quasi-product of $\mc{S}$.
\end{enumerate}
\end{thm}

We next introduce the class of topological groups of semisimple type. Our approach here is motivated by similar notions in the theories of finite groups and of Lie groups.  Given a group $G$ and a subset $K$ of $G$, we write $\ngrpd{K}{G}$ for the normal subgroup $\grp{gkg\inv \mid k \in K, g \in G}$.  Then the topological closure $\ngrp{K}{G}$ of $\ngrpd{K}{G}$ is the smallest closed normal subgroup of $G$ that contains $K$.

\begin{defn}
Let $G$ be a topological group.  A \defbold{component}\index{component} of $G$ is a closed subgroup $M$ of $G$ such that the following conditions hold:
\begin{enumerate}[(a) ]
\item $M$ is normal in $\ngrp{M}{G}$.
\item $M/\Z(M)$ is non-abelian.
\item Whenever $K$ is a proper closed normal subgroup of $M$, then $K$ is central in $\ngrp{M}{G}$. 
\end{enumerate}
The \defbold{layer}\index{layer} $E(G)$ of $G$ is the closed subgroup generated by the components of $G$. 
\end{defn}

Any Polish group has at most countably many components (Lemma~\ref{lem:semisimple:countable}).  Each component $M$ of $G$ is also a component of any closed subgroup containing $M$, so the study of components naturally reduces to the case when $G =E(G)$.

\begin{defn}
A topological group $G$ is of \defbold{semisimple type}\index{semisimple type} if $G = E(G)$. We say $G$ is of \defbold{strict semisimple type}\index{strict semisimple type} if in addition $\Z(G) = \triv$.
\end{defn}

We have good control over closed normal subgroups and quotients of groups of semisimple type.  The case of quotients is one reason to consider groups of semisimple type, rather than only groups of strict semisimple type.

\begin{thm}[See Theorem~\ref{thm:norm_sgrps}]\label{thmintro:norm_sgrps}
Suppose that $G$ is a topological group of semisimple type and $K$ is a closed normal subgroup of $G$.  Then $G/K$ is of semisimple type, and $G/\Z(G)$ is of strict semisimple type.
\end{thm}

One naturally anticipates that a group of semisimple type is almost a product of topologically simple groups.  We make this precise using generalized central products and quasi-products.

\begin{thm}[See \S\ref{sec:semisimple_structure}]\label{thmintro:min_normal:simple}
Let $G$ be a non-trivial topological group and let $\mc{M}$ be the collection of components of $G$. 
\begin{enumerate}[(1) ]
\item If $G$ is of semisimple type, then $\mc{M}$ is a generalized central factorization of $G$ and each $M\in \mc{M}$ is non-abelian and central-by-topologically simple.
\item If $G$ is of strict semisimple type, then $\mc{M}$ is a quasi-direct factorization of $G$ and each $M\in \mc{M}$ is non-abelian and topologically simple. 
\end{enumerate}
\end{thm}

\subsection{Chief factors in Polish groups}

\begin{defn}
For a topological group $G$, a \defbold{normal factor}\index{normal factor} of $G$ is a quotient $K/L$ such that $K$ and $L$ are distinct closed normal subgroups of $G$ with $L < K$. We say $K/L$ is a (topological) \defbold{chief factor}\index{chief factor} if there are no closed normal subgroups of $G$ lying strictly between $L$ and $K$.
\end{defn}

The key to analyzing a group via its chief factors is the association relation.

\begin{defn}
For a topological group $G$, we say the normal factors $K_1/L_1$ and $K_2/L_2$ are \defbold{associated}\index{association relation} to one another if the following equations hold:
\[
\ol{K_1L_2} = \ol{K_2L_1}; \; K_1 \cap \ol{L_1L_2} = L_1; \; K_2 \cap \ol{L_1L_2} = L_2.
\]
\end{defn}

One motivation for this definition is to overcome the breakdown of the second isomorphism theorem in topological groups. Given closed normal subgroups $N,M$ in a group $G$, we wish to say the factor $\ol{NM}/N$ is ``the same" as the factor $M/N\cap M$. The relation of association allows us to say precisely this (Lemma~\ref{lem:associated:secondiso}).  The association relation is furthermore an equivalence relation when restricted to non-abelian chief factors (Proposition~\ref{associated:centralizers:chief}).

Considering non-abelian chief factors up to association, a version of the Schreier refinement theorem for chief series holds.

\begin{thm}[See Theorem~\ref{thm:unique_associate}]\label{thmintro:Schreier_refinement}
Let $G$ be a Polish group, $K/L$ be a non-abelian chief factor of $G$, and
\[
\triv = G_0 \le G_1 \le \dots \le G_n = G
\]
be a series of closed normal subgroups in $G$.  Then there is exactly one $i \in \{0,\dots,n-1\}$ such that there exist closed normal subgroups $G_i \le B \le A \le G_{i+1}$ of $G$ for which $A/B$ is a non-abelian chief factor associated to $K/L$.  
\end{thm}

The above theorem suggests that the \textit{association classes} are fundamental building blocks of Polish groups. This leads to the following definition:

\begin{defn}
A \defbold{chief block}\index{chief block} of a topological group $G$ is an association class of non-abelian chief factors of $G$. For a non-abelian chief factor $A/B$, we write $[A/B]$ for the chief block. We write $\mf{B}_G$ for the set of chief blocks of $G$. 
\end{defn}

There is additionally a partial order on $\mf{B}_G$, denoted by $\leq$, which keeps track of the order in which representatives of association classes can appear in normal series.

Given a normal subgroup $N$ of a Polish group $G$ and a chief block $\mf{a}$, we say $N$ \defbold{covers}\index{covers} $\mf{a}$ if there exist $B \le A \le N$ such that $A/B \in \mf{a}$. The set of closed normal subgroups that cover a chief block $\mf{a}$ form a filter in the lattice of closed normal subgroups of $G$ (Lemma~\ref{block:covering:intersection}).  We say that $\mf{a}$ is \defbold{minimally covered}\index{minimally covered} if this filter is principal\textemdash{}\textit{i.e.}\ if there is a unique smallest closed normal subgroup that covers $\mf{a}$. We write $\mf{B}_G^{\min}$ for the set of minimally covered chief blocks.  

\begin{rmk}
At the present level of generality, it is not clear which, if any, chief blocks will be minimally covered; in the context of locally compact Polish groups this is remedied in \cite{RW_LC}, where it is shown that ``large'' chief blocks are necessarily minimally covered.
\end{rmk}

Considering the structure of $\mf{B}_{G}^{\min}$ when $G$ is a topologically characteristically simple Polish group (for example, when $G$ is a chief factor of some ambient Polish group) leads to an interesting trichotomy.
\begin{defn}
Let $G$ be a topologically characteristically simple topological group. 
\begin{enumerate}[(1) ]
\item The group $G$ is of \defbold{weak type}\index{weak type} if $\mf{B}_G^{\min}=\emptyset$. 
\item The group $G$ is of \defbold{stacking type}\index{stacking type} if $\mf{B}_G^{\min}\neq \emptyset$ and for all $\mf{a}, \mf{b} \in \mf{B}_G^{\min}$, there exists $\psi \in \Aut(G)$ such that $\psi.\mf{a} < \mf{b}$.
\end{enumerate}
\end{defn}

\begin{thm}[See Theorem~\ref{thm:chief:block_structure}]
If $G$ is a topologically characteristically simple Polish group, then $G$ is of either weak type, semisimple type, or stacking type.  Moreover, the three types are mutually exclusive.
\end{thm}

We go on to define a space associated to a Polish group that captures the chief factor data. Using this space, we show that it makes sense to refer to a \textit{chief block} as being of weak, semisimple or stacking type.

\begin{thm}[See Proposition~\ref{prop:type_invariant}]\label{thm:prop:type_invariant}
Let $G$ be a Polish group and $\mf{a}\in \mf{B}_G$.  Every representative $A/B\in\mf{a}$ is of the same type: weak, semisimple, or stacking.
\end{thm}

This work concludes with examples. Notably, F. Le Ma\^{i}tre contributes a number of interesting examples of normal compressions in the non-locally compact non-abelian Polish setting; these may be found in the appendix.

\section{Preliminaries}
All groups are Hausdorff topological groups and are written multiplicatively. Topological group isomorphism is denoted by $\simeq$. For a topological group $G$, the connected component of the identity is denoted by $G^\circ$, and the center is denoted by $\Z(G)$.  The group of automorphisms of $G$ as a topological group is denoted by $\Aut(G)$.  For a subset $K\subseteq G$, $\CC_G(K)$ is the collection of elements of $G$ that centralize every element of $K$. We denote the collection of elements of $G$ that normalize $K$ by $\N_G(K)$. We write $\ngrpd{K}{G}$ for the group $\grp{gkg\inv \mid k \in K, g \in G}$. The topological closure of $K$ in $G$ is denoted by $\ol{K}$. For $a,b,c\in G$, we set $[a,b]:=aba\inv b\inv $ and $[a,b,c]:=[[a,b],c]$. For $A,B\subseteq G$, we put 
\[
\left[A,B\right]:=\grp{[a,b]\;|\;a\in A\text{ and }b\in B};
\]
thus $[G,G]$ denotes the (abstract) \defbold{commutator subgroup}\index{commutator subgroup} of $G$.

In this article, minimal and maximal subgroups of a group $G$ are always taken to exclude $\triv$ and $G$ respectively; however, we will sometimes note for emphasis that the relevant groups are non-trivial, respectively proper.

Given a group $G$ acting on a set $X$, a map $\psi: X \rightarrow Y$ is said to be \defbold{$G$-invariant}\index{$G$-equivariant} if $\psi(g.x) = \psi(x)$ for all $g \in G$ and $x \in X$; if $G$ also acts on $Y$, then $\psi$ is \defbold{$G$-equivariant}\index{$G$-equivariant} if $\psi(g.x) = g.\psi(x)$ for all $g \in G$ and $x \in X$.  A \defbold{$G$-invariant}\index{$G$-invariant} subset $Z$ of $X$ is one for which the indicator function is $G$-invariant, in other words, for all $g \in G$ and $x \in X$ we have $x \in Z \Leftrightarrow g.x \in Z$.  Unless otherwise specified, when $X$ is a group and $G$ is a subgroup, the action of $G$ is assumed to be by conjugation, that is, $g.x = gxg\inv$.

\subsection{Polish spaces and groups}

\begin{defn}A \textbf{Polish space}\index{Polish space} is a separable topological space that admits a complete, compatible metric. A \textbf{Polish group}\index{Polish group} is a topological group so that the topology is Polish.
\end{defn}
We recall a few classical facts concerning Polish spaces.

\begin{lem}[{\cite[\S6.1 Proposition 1]{B_top2_89}}]\label{lem:increasing_Polish}
\begin{enumerate}[(1) ]
\item Every closed subspace of a Polish space is Polish.
\item Given a countable family of Polish spaces, their product and disjoint union are also Polish spaces.
\end{enumerate}
\end{lem}

\begin{thm}[Lusin--Souslin, cf. {\cite[(15.1)]{K95}}]\label{thm:lusin-souslin} Let $X$ and $Y$ be Polish spaces and $f:X\rightarrow Y $  be continuous. If $A\subseteq X$ is Borel and $f\rest_A$ is injective, then $f(A)$ is Borel.
\end{thm}

\begin{thm}[{\cite[(8.38)]{K95}}]\label{thm:dense_cont} Suppose that $X$ and $Y$ are Polish spaces. If $f:X\rightarrow Y$ is Borel measurable, then there is a dense $G_{\delta}$ set $L\subseteq X$ so that $f\rest_L$ is continuous.
\end{thm}

Specializing to Polish groups, we note two closure properties of the class.
\begin{prop}[{\cite[3.1 Proposition 4]{B_top2_89}}]\label{prop:quotient}
If $G$ is a Polish group and $H\normal G$ is a closed normal subgroup, then $G/H$ is a Polish group.
\end{prop}
\begin{prop}[{\cite[2.10 Proposition 28]{B_top1_89}}]\label{prop:semidirect}
Suppose that $G$ and $H$ are Polish groups. If $H$ acts continuously by automorphisms on $G$ via $\alpha$, then the semidirect product $G\rtimes_{\alpha} H$ is a Polish group under the product topology. 
\end{prop}

Various automatic continuity properties additionally hold for Polish groups. 
\begin{thm}[{\cite[(9.10)]{K95}}]\label{thm:borel_measurable}
Let $G$ and $H$ be Polish groups. If $\psi:G\rightarrow H$ is a Borel measurable homomorphism, then $\psi$ is continuous. In particular, any Borel measurable automorphism of a Polish group is a homeomorphism. 	
\end{thm}

\begin{thm}[{\cite[(9.14)]{K95}}]\label{thm:sep_cont}
Let $G$ be a Polish group and $X$ a Polish space on which $G$ acts by homeomorphisms. If the action of $G$ on $X$ is separately continuous, then the action is continuous.
\end{thm}

\section{Normal compressions in Polish groups}\label{sec:compressions}
We here make a general study of normal compressions between Polish groups. In particular, Theorem~\ref{thmintro:psi-compression_factor} will be established. In Section~\ref{sec:associated}, we shall see that normal compressions play a fundamental role in the study of chief factors in Polish groups.

\begin{defn}
Let $G$ and $H$ be topological groups. A continuous homomorphism $\psi:G\rightarrow H$ is a \defbold{normal compression}\index{normal compression} if it is injective with a dense, normal image.
\end{defn}

\subsection{The equivariant action}\label{sec:equivariant}

For Polish groups $G$ and $H$ such that $H$ is a normal compression of $G$, there is a canonical continuous action of $H$ on $G$ which allows us to form the external semidirect product $G\rtimes H$, which is itself a Polish group with respect to the product topology. This external semidirect product gives a natural factorization of the normal compression.

\begin{lem}\label{lem:Polish:continuous_pullback}
Let $G$ and $H$ be Polish groups with $\psi: G \rightarrow H$ a continuous, injective homomorphism and let $\chi$ be a topological group automorphism of $H$ such that $\chi(\psi(G)) = \psi(G)$.  Then the map $\phi_{\chi}:G\rightarrow G$ defined by $\phi_{\chi}:=\psi\inv\circ\chi\circ\psi$ is a topological group automorphism of $G$.
\end{lem}

\begin{proof}
It is clear that $\phi_{\chi}$ is an abstract group automorphism of $G$. In view of Theorem~\ref{thm:borel_measurable}, it suffices to show that $\phi_{\chi}$ is Borel measurable to conclude the lemma. 

Fixing $O\subseteq G$ an open set, we see
\[
\phi_{\chi}^{-1}(O)=\{g\in G\mid \psi\inv\circ\chi\circ\psi(g)\in O\}=\{g\in G\mid g\in \psi^{-1}\circ\chi^{-1}\circ \psi(O)\}.
\]
Since $\psi$ is injective and continuous, $\psi(O)$ is a Borel set by Theorem~\ref{thm:lusin-souslin}, and thus, $\psi^{-1}\chi^{-1} \psi(O)$ is a Borel set. We conclude that $\phi_{\chi}^{-1}(O)$ is a Borel set, and it now follows the map $\phi_{\chi}$ is Borel measurable, verifying the lemma.
\end{proof}

We only use the following special case of Lemma~\ref{lem:Polish:continuous_pullback}.

\begin{cor}\label{cor:Polish:compression_pullback}
Let $G$ and $H$ be Polish groups and let $\psi: G \rightarrow H$ be a continuous, injective homomorphism such that $\psi(G)$ is normal in $H$.  Then for each $h\in H$, the map $\phi_h:G\rightarrow G$ defined by $\phi_h(g): = \psi^{-1}(h\psi(g)h\inv)$ is a topological group automorphism of $G$.
\end{cor}

Corollary~\ref{cor:Polish:compression_pullback} gives a canonical continuous action in the setting of normal compressions.  Given a normal compression $\psi:G\rightarrow H$, this action of $H$ on $G$ is the unique action that makes $\psi$ $H$-equivariant.
\begin{defn}
Suppose that $G$ and $H$ are Polish groups and $\psi:G\rightarrow H$ is a normal compression. 
We call the action of $H$ on $G$ given by $(h,g)\mapsto \phi_h(g)$ the \textbf{$\psi$-equivariant action}\index{$\psi$-equivariant action}; when clear from context, we suppress ``$\psi$". 
\end{defn}

Our next proposition is a special case of \cite[(9.16)]{K95}; we include a proof for completeness.
\begin{prop}\label{prop:compression_action_cont}
If $G$ and $H$ are Polish groups and $\psi:G\rightarrow H$ is a normal compression, then the $\psi$-equivariant action is continuous.
\end{prop}
\begin{proof}
In view of Theorem~\ref{thm:sep_cont}, it suffices to show the equivariant action is separately continuous. Fixing $h\in H$, Corollary~\ref{cor:Polish:compression_pullback} ensures that $\phi_h$ is a continuous automorphism of $G$. Hence, if $g_i\rightarrow g$, then $\phi_h(g_i)\rightarrow \phi_h(g)$; that is to say, the equivariant action is continuous in $G$ for fixed $h$. 

It remains to show that the equivariant action is continuous in $H$ for fixed $g$. Fix $g\in G$ and consider the map $\alpha:h\rightarrow \phi_h(g)=:h.g$. For $O\subseteq G$ open, 
\[
\alpha^{-1}(O)=\{h\in H\mid \psi^{-1}(h\psi(g)h^{-1})\in O\}=\{h\in H\mid h\psi(g)h^{-1}\in \psi(O)\}.
\] 
The map $c:H\rightarrow H$ via $k\mapsto k\psi(g)k^{-1}$ is plainly continuous. Additionally, since $\psi$ is injective and continuous, Theorem~\ref{thm:lusin-souslin} implies that $\psi(O)$ is a Borel set of $H$, hence the following set is Borel:
\[
c^{-1}(\psi(O))= \{h\in H\mid h\psi(g)h^{-1}\in \psi(O)\}=\alpha^{-1}(O).
\]
The map $\alpha$ is therefore a Borel measurable map.

Theorem~\ref{thm:dense_cont} now ensures $\alpha$ is continuous when restricted to a dense $G_{\delta}$ set $L\subseteq H$. We argue that $\alpha$ is indeed continuous at every point of $L$.  Fix $h\in L$ and suppose that $h_i\rightarrow h$ in $H$. Appealing to the Baire category theorem, $\wt{L}:=Lh^{-1}\cap \bigcap_{i\in \Nb}Lh_i^{-1}$ is dense and \textit{a fortiori} non-empty. Fixing $l\in \wt{L}$, we see that $\alpha(h_i)=l^{-1}.(lh_i.g)$. Moreover, the choice of $l$ ensures that $lh_i\in L$ for all $i$ and that $ lh\in L$. The continuity of $\alpha$ when restricted to $L$ thus implies $lh_i.g\rightarrow lh.g$. Since the action is continuous in $G$ for fixed $l$, it follows that $\alpha(h_i)\rightarrow \alpha(h)$, hence $\alpha$ is continuous at every point of $L$. 

Now take $h\in L$, fix some arbitrary $k\in H$, and let $k_i\rightarrow k$. We see that 
\[
\alpha(k_i)=(kh^{-1}).(hk^{-1}k_i.g),
\]
and moreover, $hk^{-1}k_i\rightarrow h$. That $\alpha$ is continuous at $h$ implies
\[
hk^{-1}k_i.g\rightarrow h.g.
\]
On the other hand, the action is continuous in $G$ for $kh^{-1}$, so 
\[
(kh^{-1}).(hk^{-1}k_i.g)\rightarrow (kh^{-1}).(h.g)=k.g.
\]
The map $\alpha$ is thus continuous at $k$, and the proposition follows.
\end{proof}

Theorem~\ref{thmintro:psi-compression_factor}(1) now follows immediately from Propositions~\ref{prop:compression_action_cont} and~\ref{prop:semidirect}. We now derive the remainder of Theorem~\ref{thmintro:psi-compression_factor}. In fact, any normal compression $\psi: G \rightarrow H$ of Polish groups factors through $G \rtimes_{\psi} O$ where $O$ is any open subgroup of $H$.  The significance of the additional generality will become apparent in the totally disconnected locally compact Polish setting, where $O$ can be chosen to be a \emph{compact} open subgroup.

\begin{thm}\label{thm:psi-compression_factor_rel}
Let $G$ and $H$ be Polish groups, let $\psi: G \rightarrow H$ be a normal compression, and let $O\leq H$ be an open subgroup. Then the following hold:
\begin{enumerate}[(1) ]
\item $\pi:G\rtimes_{\psi}O\rightarrow H$ via $(g,o)\mapsto \psi(g)o$ is a continuous surjective homomorphism with $\ker(\pi)=\{(g\inv,\psi(g))\mid g\in \psi^{-1}(O)\}$, and if $O = H$, then $\ker(\pi)\simeq G$ as topological groups.
\item $\psi = \pi \circ \iota$ where $\iota: G \rightarrow G\rtimes_{\psi}O$ is the usual inclusion.
\item $G\rtimes_{\psi}O =\ol{\iota(G)\ker (\pi)}$, and the subgroups $\iota(G)$ and $\ker(\pi)$ are closed normal subgroups of $G\rtimes_{\psi}O$ with trivial intersection.
\end{enumerate}
\end{thm}

\begin{proof}
The map $\pi$ is clearly continuous and surjective, and a calculation shows $\pi$ is also a homomorphism. It is also immediate that $\ker(\pi)=\{(g\inv,\psi(g))\mid g\in \psi^{-1}(O)\}$. It thus remains to show $G\simeq \ker(\pi)$ in the case that $O = H$. Define $G\rightarrow \ker(\pi)$ by $g\mapsto (g\inv,\psi(g))$. Plainly, this is a bijective continuous map. Checking this map is a homomorphism is again a calculation. We thus deduce $(1)$.

Claim $(2)$ is immediate.

By $(1)$, $\iota(G)=\{(g,1)\mid g\in G\}$ intersects $\ker(\pi)$ trivially, and both $\iota(G)$ and $\ker(\pi)$ are closed normal subgroups. The product $\iota(G)\ker(\pi)$ is dense, since it is a subgroup containing the set $\{(1,h)\mid h \in \psi(G) \cap O\}\cup \iota(G)$. We have thus verified $(3)$.
\end{proof}

\begin{cor}\label{cor:normal_invar_equi_action} 
Suppose that $G$ and $H$ are Polish groups with $\psi:G\rightarrow H$ a normal compression.  Then every closed normal subgroup of $G$ is invariant under the $\psi$-equivariant action of $H$ on $G$. 
\end{cor}

\begin{proof}
Let $K$ be a closed normal subgroup of $G$ and let $\pi$ be as in Theorem~\ref{thm:psi-compression_factor_rel} with $H = O$.  Then $\ker(\pi)\cap \iota(G)=\{1\}$, so $\ker\pi$ centralizes $\iota(G)$; in particular, $\ker\pi$ normalizes $\iota(K)$.  The normalizer $\N_{G\rtimes_{\psi}H}(K)$ therefore contains the dense subgroup $\iota(G)\ker(\pi)$; moreover, $\iota(K)$ is a closed subgroup of $G \rtimes_{\psi}H$, since $\iota$ is a closed embedding.  Since normalizers of closed subgroups are closed, we conclude that $\iota(K)\normal G\rtimes_{\psi}H$. Thus $K$ is invariant under the action of $H$, verifying the corollary. 
\end{proof}

\subsection{Properties invariant under normal compressions}\label{sec:normalcompression_invariants}
We now explore which properties pass between Polish groups $G$ and $H$ for which there is a normal compression $\psi:G\rightarrow H$. In view of the goals for this work and subsequent works, the properties we explore here are related to normal subgroup structure and connectedness. There are likely other interesting properties that pass between such $G$ and $H$.

\begin{prop}\label{prop:normal_compression}
Let $G$ and $H$ be Polish groups, let $\psi: G \rightarrow H$ be a normal compression, and let $K$ be a closed normal subgroup of $G$.
\begin{enumerate}[(1) ]
\item The image $\psi(K)$ is a normal subgroup of $H$.
\item If $\psi(K)$ is also dense in $H$, then $\ol{[G,G]} \le K$, and every closed normal subgroup of $K$ is normal in $G$.
\end{enumerate}
\end{prop}

\begin{proof}
Form the semidirect product $G\rtimes_{\psi} H$, let $\iota:G\rightarrow G\rtimes_{\psi}H$ be the usual inclusion, and let $\pi:\GH\rightarrow H$ be the map given in Theorem~\ref{thm:psi-compression_factor_rel}.

Claim $(1)$ is immediate from Corollary~\ref{cor:normal_invar_equi_action}.

For $(2)$, since $\pi$ is a quotient map and $\pi(\iota(K))$ is dense in $H$, it follows that $\iota(K)\ker(\pi)$ is dense in $\GH$.  Corollary~\ref{cor:normal_invar_equi_action} implies that $\iota(K)$ is a closed normal subgroup of $\GH$. The image of $\ker(\pi)$ is thus dense under the usual projection $\chi:\GH\rightarrow \GH/\iota(K)$. On the other hand,  Theorem~\ref{thm:psi-compression_factor_rel} ensures $\iota(G)$ and $\ker(\pi)$ commute, hence $\iota(G)/\iota(K)$ has dense centralizer in $\GH/\iota(K)$. The group $\iota(G)/\iota(K)$ is then central in $\GH/\iota(K)$, so in particular, $\iota(G)/\iota(K)$ is abelian.  It now follows that $G/K$ is abelian, so $\ol{[G,G]}\leq K $.

Finally, let $M$ be a closed normal subgroup of $K$.  Applying part $(1)$ to the compression map from $K$ to $H$, we see that $\psi(M)$ is normal in $H$, so in particular $\psi(M)$ is normal in $\psi(G)$.  Since $\psi$ is injective, we conclude that $M$ is normal in $G$.
\end{proof}

We now prove various simplicity notions pass between normal compressions. For $G$ a topological group and $A$ a group acting on $G$ by automorphisms, say that $G$ is \defbold{$A$-simple}\index{$A$-simple} if $G \neq \triv$ and $A$ leaves no proper non-trivial closed normal subgroup of $G$ invariant.  For example, $G$ is $\triv$-simple if and only if $G$ is topologically simple.

\begin{thm}\label{thm:compression_simple}Let $G$ and $H$ be non-abelian Polish groups and let $\psi: G \rightarrow H$ be a normal compression.  Suppose also that $G$ and $H$ admit actions by topological automorphisms of a (possibly trivial) group $A$ and that $\psi$ is $A$-equivariant.
\begin{enumerate}[(1) ] 
\item If $G$ is $A$-simple, then so is $H/\Z(H)$, and $\Z(H)$ is the unique largest proper closed $A$-invariant normal subgroup of $H$.
\item If $H$ is $A$-simple, then so is $\ol{[G,G]}$, and $\ol{[G,G]}$ is the unique smallest non-trivial closed $A$-invariant normal subgroup of $G$.
\end{enumerate}
\end{thm}

\begin{proof}
For $(1)$, suppose that $G$ is $A$-simple and let $L$ be a proper closed normal $A$-invariant subgroup of $H$.  Then $\psi(G) \not\le L$, so $\psi\inv(L)$ is a proper, hence trivial, closed normal $A$-invariant subgroup of $G$.  The subgroups $\psi(G)$ and $L$ are thus normal subgroups of $H$ with trivial intersection, so $\psi(G)$ and $L$ commute.  Since $\psi(G)$ is dense in $H$, it follows that $L \le \Z(H)$.  In particular, $H/\Z(H)$ does not have any proper non-trivial closed normal $A$-invariant subgroup, and $(1)$ follows.

\

For $(2)$, suppose that $H$ is $A$-simple.  The subgroup $L: = \ol{[G,G]}$ is topologically characteristic in $G$ and hence is normal and $A$-invariant; moreover $L \neq \triv$, since $G$ is not abelian. The image $\psi(L)$ is therefore a non-trivial, hence dense, $A$-invariant subgroup of $H$. Proposition~\ref{prop:normal_compression} now implies any closed $A$-invariant normal subgroup of $L$ is normal in $G$.

Letting $K$ be an arbitrary non-trivial closed $A$-invariant normal subgroup of $G$, Proposition~\ref{prop:normal_compression} ensures the group $\psi(K)$ is normal in $H$. Since $\psi$ is $A$-equivariant, $\psi(K)$ is indeed an $A$-invariant subgroup of $H$, so by the hypotheses on $H$, the subgroup $\psi(K)$ is dense in $H$. Applying Proposition~\ref{prop:normal_compression} again, we conclude that $K \ge L=\ol{[G,G]}$. The subgroup $L$ is thus the unique smallest non-trivial closed $A$-invariant normal subgroup of $G$, and $(2)$ follows.
\end{proof}

Similar to $A$-simplicity, connectedness is well-behaved up to abelian quotients. Demonstrating this requires a couple of lemmas. 

\begin{lem}[see for instance {\cite[Chapter III \S 9.2 Propositions 1 and 2]{Bo98}}]\label{lem:connected_commutator}
Let $G$ and $H$ be subgroups of a topological group $K$ and consider the commutator group $[G,H]$.
\begin{enumerate}[(1) ]
\item We have $[\ol{G},\ol{H}] \le \ol{[G,H]}$.
\item If $G$ is connected, then $[G,H]$ is connected.
\end{enumerate}
\end{lem}

\begin{lem}\label{lem:disconnected_to_connected}
Let $G$ and $H$ be Polish groups and let $\psi: G \rightarrow H$ be a normal compression.  Set $R := \psi\inv(H^\circ)$ and $N := \ol{\psi(R)}$.  Then the following inclusions hold:
\[
[H^\circ,\psi(G)] \le \psi(R); \;[H^\circ,H] \le N; [H^{\circ},\psi(G)]\leq \psi(G^{\circ}); \; [H^\circ, N] \le \ol{\psi(G^\circ)}.
\]
\end{lem}

\begin{proof}
Since $\psi(G)$ and $H^\circ$ are both normal subgroups of $H$, we have $[H^\circ,\psi(G)] \le \psi(G) \cap H^\circ$. From the definition of $R$, $\psi(G) \cap H^\circ = \psi(R)$, hence
\[
[H^\circ,\psi(G)] \le \psi(R).
\]
verifying the first inclusion.

By the first inclusion, $H^\circ/N$ is centralized by the dense subgroup $\psi(G)N/N$ of $G/N$.  The subgroup $H^\circ/N$ is therefore central in $H/N$ by Lemma~\ref{lem:connected_commutator}(1), hence
\[
[H^\circ,H] \le N
\]
establishing the second inclusion.

For the third and fourth inclusions, let $K$ be the semidirect product $G \rtimes_{\psi} H$. Consider the commutator group $[H^\circ,G]$ where $G$ and $H$ are regarded as subgroups of $K$.  The subgroup $[H^{\circ},G]$ is then connected by Lemma~\ref{lem:connected_commutator}(2) and is a subgroup of $G$, hence $[H^{\circ},G] \le G^\circ$. Taking the image under the map $\pi$ as given by Theorem~\ref{thm:psi-compression_factor_rel}, we see that 
\[
\pi([H^{\circ},G])=[H^{\circ},\psi(G)]\leq \pi(G^{\circ})=\psi(G^{\circ})
\]
giving the third inclusion. It now easily follows that
\[
[H^\circ, N] \le \ol{\psi(G^\circ)}.
\]
\end{proof}

\begin{thm}\label{thm:connected}
Let $G$ and $H$ be Polish groups and let $\psi: G \rightarrow H$ be a normal compression. Then the following hold:
\begin{enumerate}[(1) ]
\item The group $H^\circ/\ol{\psi(G^\circ)}$ is nilpotent of class at most two.
\item If $H$ is connected, then both $G/G^\circ$ and $H/\ol{\psi(G^\circ)}$ are abelian.
\item If $G$ is totally disconnected, then $H^\circ$ is central in $H$.
\end{enumerate}
\end{thm}

\begin{proof}
Let $R$ and $N$ be as in Lemma~\ref{lem:disconnected_to_connected}. The second inclusion relation of Lemma~\ref{lem:disconnected_to_connected} implies $[H^{\circ},H^{\circ}]\leq N$. Applying the fourth inclusion, we conclude 
\[
[H^{\circ},[H^{\circ},H^{\circ}]]\leq \ol{\psi(G^{\circ})}.
\]
Hence, $H^{\circ}/\ol{\psi(G^{\circ})}$ is nilpotent of class at most two.

For claim $(2)$, first suppose $H$ is connected.  Then the third inclusion of Lemma~\ref{lem:disconnected_to_connected} gives $[H,\psi(G)]\leq \psi(G^{\circ})$. Therefore, $[\psi(G),\psi(G)]\leq \psi(G^{\circ})$, and since $\psi$ is injective, we conclude $G/G^{\circ}$ is abelian. For the next claim of $(2)$, observe that $\psi(G)\ol{\psi(G^{\circ})}/\ol{\psi(G^{\circ})}$ is dense and abelian. Therefore, $H/\ol{\psi(G^{\circ})}$ is abelian. 

For the third claim, the fact that $G^\circ = \triv$ implies $[H^\circ,\psi(G)]=\triv$, by the third inclusion given by Lemma~\ref{lem:disconnected_to_connected}. Since $\psi(G)$ is dense in $H$, we deduce that $H^\circ$ is central in $H$.
\end{proof}

The following is an easy consequence of Theorem~\ref{thm:connected}; note particularly the case that $G$ and $H$ are non-abelian and topologically characteristically simple. Recall a group is \textbf{totally disconnected}\index{totally disconnected} if the only connected sets are singletons; \textit{i.e.} the connected component of the identity is trivial. This property is occasionally called \textbf{hereditarily disconnected}\index{hereditarily disconnected} in the topological group literature.

\begin{cor}\label{cor:con_association}Suppose that $G$ and $H$ are centerless Polish groups with dense commutator subgroups and that $H$ is a normal compression of $G$.  Then $G$ is connected if and only if $H$ is connected, and $G$ is totally disconnected if and only if $H$ is totally disconnected.
\end{cor}

\begin{rmk}
A hypothesis to rule out the abelian case is necessary; \textit{e.g.} $\Qb$ and $\Rb$ are topologically characteristically simple Polish groups, and $\Rb$ is a normal compression of $\Qb$.\end{rmk}

\section{Generalized direct products}\label{sec:quasiproduct}

We now define and study quasi-products along with generalized central products. The conditions for a generalized central product have the advantage of being stable under quotients, and they are easier to define and verify.  The distinction is similar to the classical one between direct products and central products.

\subsection{Definitions and first properties}\label{sec:quasiproduct1}
Suppose that $G$ is a topological group and that $J \subseteq \Norm(G)$, where $\Norm(G)$ is the collection of closed normal subgroups of $G$.  Set $G_J :=  \cgrp{N \mid N \in J}$.

\begin{defn}
We say that $\mc{S}\subseteq \Norm(G)$ is a \defbold{generalized central factorization}\index{generalized central factorization} of $G$, if $\triv \not\in \mc{S}$, $G_{\mc{S}}=G$ and $[N,M] = \triv$ for any $N\neq M$ in $\mc{S}$. In such a case, $G$ is said to be a \defbold{generalized central product}\index{generalized central product}.

We say $\mc{S}$ is a \defbold{quasi-direct factorization}\index{quasi-direct factorization} of $G$ if $\triv \not\in \mc{S}$, $G_{\mc{S}}=G$, and $\mc{S}$ has the following topological independence property:
\[
\forall X\subseteq \mathcal{P}(\mc{S}):\; \bigcap X=\emptyset \Rightarrow \bigcap_{A\in X}G_A=\{1\}. 
\]
In such a case, $G$ is said to be a \textbf{quasi-product}\index{quasi-product}. A subgroup $H$ of $G$ is a \defbold{quasi-factor}\index{quasi-factor} of $G$ if it is an element of some quasi-direct factorization. A factorization $\mc{S}$ is \textbf{non-trivial}\index{non-trivial factorization} if $|\mc{S}|\ge 2$.
\end{defn}

The largest possible subgroup of the form $\bigcap_{A\in X}G_A$ where $X\subseteq \mathcal{P}(\mc{S})$ and $\bigcap X=\emptyset$ occurs when $X = \{ \mc{S} \setminus \{A\} \mid A \in \mc{S}\}$.  We thus observe an equivalent form of the independence property.

\begin{obs}\label{obs:quasi-product_char}
	Let $G$ be a topological group and suppose that $\mc{S}$ is a set of closed normal subgroups of $G$ such that $|\mc{S}|\ge 2$. Then $\mc{S}$ is a quasi-direct factorization if and only if $\bigcap_{N\in \mc{S}}G_{\mc{S}\setminus\{N\}}=\{1\}$.
\end{obs}

The independence property also easily implies a non-redundancy property in generating the group topologically.

\begin{lem}\label{lem:quasiproduct:nonredundant}
Let $G$ be a topological group and suppose that $\mc{S}$ is a quasi-direct factorization of $G$.  If $\mc{S}_1$ and $\mc{S}_2$ are subsets of $\mc{S}$, then $G_{\mc{S}_1} = G_{\mc{S}_2}$ if and only if $\mc{S}_1 = \mc{S}_2$.
\end{lem}

\begin{proof}
If $\mc{S}_1 = \mc{S}_2$, then clearly $G_{\mc{S}_1} = G_{\mc{S}_2}$.  Let us now suppose that $\mc{S}_1 \neq \mc{S}_2$; without loss of generality, there exists $N \in \mc{S}_1 \setminus \mc{S}_2$.  We now deduce that $N \le G_{\mc{S}_1}$, but $N \cap G_{\mc{S}_2} = \triv$ by the independence property, since $\{N\} \cap \mc{S}_2 = \emptyset$.  The groups $G_{\mc{S}_1}$ and $G_{\mc{S}_2}$ are thus distinct.
\end{proof}

The difficulties associated with the independence property all arise from central subgroups. 

\begin{prop}\label{prop:quasiproduct:centralizers}
Let $G$ be a centerless topological group and let $\mc{S}\subseteq \Norm(G)$ have at least two members.  Then $\mc{S}$ is a quasi-direct factorization of $G$ if and only if $\mc{S}$ is a generalized central factorization of $G$.
\end{prop}

\begin{proof}
The forward implication is immediate. Conversely, let $X\subseteq \pwr(\mc{S})$ be a set of subsets with empty intersection and take $x \in \bigcap_{A\in X}G_{A}$. For each $N\in \mc{S}$, there exists $A\in X$ such that $N \not\in A$, so $G_{A}$ centralizes $N$.  The element $x$ therefore centralizes $N$ for all $N\in \mc{S}$.  We deduce that $\CC_G(x)$ is dense in $G$, so in fact $x$ is central in $G$. Therefore, $x = 1$, proving that $\mc{S}$ is a quasi-direct factorization.
\end{proof}

\begin{rmk}
To see the role of the independence property in the case of Polish groups with non-trivial center, consider the additive group of the Hilbert space $H = \ell^2(\Rb)$ equipped with the norm topology.  One can decompose $H$ into closed subspaces that are linearly independent in the sense of abstract vector spaces, but fail to be independent in a topological sense.  Any decomposition of $H$ as an orthogonal sum of closed subspaces however is a quasi-direct factorization. 
\end{rmk}

We now prove Theorem~\ref{thmintro:directproduct:quasiproduct}, which elucidates the relationship between quasi-products and direct products.

\begin{proof}[Proof of Theorem~\ref{thmintro:directproduct:quasiproduct}.]
$(1)$
Set $H_N:=G/G_{\mc{S}\setminus \{N\}}$ and define $d:G\rightarrow \prod_{N\in \mc{S}}H_N$ to be the diagonal map. That is to say, $g\mapsto (gG_{\mc{S}\setminus \{N\}})_{N\in \mc{S}}$. The map $d$ is  a continuous group homomorphism, and for a fixed $N\in \mc{S}$, the image $d(N)$ lies in $H_N$. Since $NG_{\mc{S}\setminus\{N\}}$ is dense in $G$, the image of $N$ is dense in $H_N$. We conclude that $d(G)\cap H_N$ is dense in $H_N$ for all $N\in \mc{S}$.

For each $N\in \mc{S}$, the subgroup $G_{\mc{S}\setminus \{N\}}$ centralizes $N$. On the other hand,
\[
\ker(d) = \bigcap_{N \in \mc{S}}G_{\mc{S}\setminus \{N\}},
\]
so $\ker(d)$ centralizes every $N\in \mc{S}$. Since $\mc{S}$ generates a dense subgroup of $G$, we deduce that $\ker(d)$ is central in $G$.  We also conclude via Observation~\ref{obs:quasi-product_char} that $\mc{S}$ is a quasi-direct factorization if and only if $d$ is injective.

\medskip

$(2)$
Let $K \in \mc{K}$.  Since $K$ is closed and normal in $\prod_{K \in \mc{K}}K$, it is clear that $\delta\inv(K)$ is closed and normal in $G$.  To show that ${\mc{S}:=\{\delta\inv(K)\mid K\in \mc{K}\}}$ is a quasi-direct factorization of $H:=\cgrp{\mc{S}}$, it remains to show that $\mc{S}$ satisfies the topological independence property.

For $L\in \mc{K}$, let $\pi_L: \prod_{K\in \mc{K}}K \rightarrow L$ be the projection of the direct product onto $L$ and set $\tilde{L}:=\delta^{-1}(L)$. We now consider $X\subseteq\pwr(\mc{S})$ with empty intersection. For each $\tilde{L}\in \mc{S}$, there exists $A\in X$ such that $\tilde{L}\notin A$.  For each $\tilde{K}\in A$, it now follows that $\tilde{K} \le \ker(\pi_L\circ\delta)$. Therefore, $H_A \le \ker(\pi_L\circ \delta)$, and we conclude
\[
\bigcap_{A\in X}H_{A} \le \bigcap_{L \in \mc{K}}\ker(\pi_L\circ \delta) = \ker(\delta) = \triv.
\]
This completes the proof that $\mc{S}$ is a quasi-direct factorization of $H$.
\end{proof}

\begin{rmk}
In the class of Polish groups, normal compressions play a role in the structure of quasi-products and vice versa.  In one direction, the quotient map $G \rightarrow G/G_{\mc{S}\setminus\{N\}}$ restricts to a normal compression of the quasi-factor $N$ of $G$.  In the other direction, given any normal compression $\psi: L \rightarrow H$, Theorem~\ref{thmintro:psi-compression_factor} ensures that the semidirect product $L\rtimes_{\psi} H$ is a quasi-product of two copies of $L$: the first is embedded in the obvious way, and the second occurs as the closed normal subgroup $\{(l\inv,\psi(l)) \mid l \in L\}$.
\end{rmk}

\subsection{Further observations}
We here note a general method for producing quasi-products from a quotient of a generalized central product.

\begin{prop}\label{prop:quasiproduct:quotient}
	Let $G$ be a topological group with $\mc{S}$ a non-trivial generalized central factorization of $G$ and let $N$ be a proper closed normal subgroup of $G$.  Set
	\[
	M := \bigcap_{S \in \mc{S}}\ol{G_{\mc{S} \setminus \{S\}}N}.
	\]
	Then $M/N$ is central in $G/N$, and if $M$ is a proper subgroup,
	\[
	\mc{S}/M := \{\ol{SM}/M \mid S\in\mc{S}\text{ and } S \not\le M\}
	\]
	is a quasi-direct factorization of $G/M$.
\end{prop}

\begin{proof}
For each $S \in \mc{S}$, the group $G_{\mc{S} \setminus \{S\}}$ centralizes $S$, and hence $[S,\ol{G_{\mc{S} \setminus \{S\}}N}] \le N$.  Therefore, $M/N$ commutes with $SN/N$ for all $S \in \mc{S}$.  It now follows that $M/N$ is central in $G/N$.
	
For the second claim, we observe first that 
\[
M = \bigcap_{S \in \mc{S}}\ol{G_{\mc{S}\setminus \{S\}}M}.
\]
 Indeed, fix $r \in \bigcap_{S \in \mc{S}}\ol{G_{\mc{S}\setminus \{S\}}M}$. For each $S\in\mc{S}$, we have
\[
r \in \ol{G_{\mc{S} \setminus \{S\}}M} \le \ol{G_{\mc{S} \setminus \{S\}}(G_{\mc{S} \setminus \{S\}}N)} = \ol{G_{\mc{S} \setminus \{S\}}N}.
\]
Thus, $r \in \ol{G_{\mc{S} \setminus \{S\}}N}$ for all $S \in \mc{S}$, so  $r \in M$. The reverse inclusion is obvious.

It is easy to verify 
\[
(G/M)_{\mc{S}/M\setminus \{\ol{SM}/M\}}=\ol{G_{\mc{S}\setminus \{S\}}M}/M.
\]
In view of the previous paragraph, we deduce
\[
\bigcap_{L \in \mc{S}/M}(G/M)_{\mc{S}/M\setminus \{L\}} = \bigcap_{S \in \mc{S}}\ol{G_{\mc{S}\setminus \{S\}}M}/M = 1.
\]
Appealing to Observation~\ref{obs:quasi-product_char}, the set $\mc{S}/M$ is a quasi-direct factorization of $G/M$, verifying the second claim.
\end{proof}

If $G$ is a Polish group, there is an unsurprising restriction on the number of quasi-factors.

\begin{prop}\label{prop:quasiproduct:countable:Polish}
If $G$ is a second countable topological group and $\mc{S}$ is a quasi-direct factorization of $G$, then $|\mc{S}|$ is countable.
\end{prop}

\begin{proof}
Let $(U_i)_{i\in \Nb}$ be a basis at $1$ for $G$. Since $G$ is second countable, every open cover of $G$ has a countable subcover, so $G$ is covered by countably many left $G$-translates of $U_i$ for each $i$.  The subgroup $S = \langle N \mid N\in \mc{S} \rangle$ is dense, so $G = SU_i$ for any $i$.  It follows there is a countable subset $X$ of $S$ such that $G = XU_i$ for any $i$.  We may thus choose $\mc{S}'\subseteq \mc{S}$ countable so that $\grp{N\mid N\in \mc{S}'}$ is dense.  We conclude that $\mc{S}' = \mc{S}$ by Lemma~\ref{lem:quasiproduct:nonredundant}.
\end{proof} 

\subsection{Local direct products and powers}

We conclude this section by defining a natural group-forming operation, the \defbold{local direct product} (also known as the \defbold{restricted product}), that gives topological groups that are often quasi-products but not direct products or sums. (The idea is well-known; see for example \cite[Chapter 7 \S 1]{La13}.)  For a set $I$, we denote the collection of finite subsets of $I$ by $P_f(I)$. 

\begin{defn}
Suppose that $(G_i)_{i\in I}$ is a family of topological groups and that there is a distinguished open subgroup $O_i\leq G_i$ for each $i\in I$. For each $F\in P_f(I)$, define
\[
S_{F}:=\prod_{i\in F}G_i\times \prod_{i\in I\setminus F}O_i
\]
with the product topology; products over the empty set are taken to be the trivial group.

The \textbf{local direct product of $(G_i)_{i\in I}$ over $(O_i)_{i\in I}$}\index{local direct product} is defined to be 
\[
\bigoplus_{i\in I}\left(G_i,O_i\right):=\bigcup_{F\in P_f(I)}S_F
\]
with the inductive limit topology\textemdash{}that is $A\subseteq \bigoplus_{i\in I}\left(G_i,O_i\right)$ is open if and only if $A\cap S_F$ is open in $S_F$ for all $F\in P_f(I)$.
\end{defn}

The group $\bigoplus_{i\in I}\left(G_i,O_i\right)$ is again a topological group. When $I$ is countable and each $G_i$ is Polish, the local direct product is again a Polish group by Lemma~\ref{lem:increasing_Polish}, since it can be decomposed as a countable disjoint union of clopen subspaces, each of which is a countable product of Polish spaces:
\[
\bigoplus_{i\in I}\left(G_i,O_i\right) = \bigsqcup_{F\in P_f(I)} \left( \prod_{i \in F} (G_i \setminus O_i) \times \prod_{i \in I \setminus F}O_i \right).
\]
Seeing $G_i$ as a subgroup of the local direct product in the obvious way, it is a closed normal subgroup, and $\{G_i\mid i\in I\}$ is a quasi-direct factorization of $\bigoplus_{i\in I}\left(G_i,O_i\right)$. In the case that $O_i$ is a proper non-trivial subgroup of $G_i$ for infinitely many $i$, the local direct product  $\bigoplus_{i\in I}\left(G_i,O_i\right)$ is neither a direct product nor a direct sum.

There is an alternative definition of the local direct product which is easier to work with in practice.

\begin{defn} Suppose that $(G_i)_{i\in I}$ is a family of topological groups and that there is a distinguished open subgroup $O_i\leq G_i$ for each $i\in I$. The \textbf{local direct product}\index{local direct product} of $(G_i)_{i\in I}$ over $(O_i)_{i\in I}$ is defined to be 
\[
\bigoplus_{i\in I}\left(G_i,O_i\right):=\left\{ (x_i)_{i \in I} \in \prod_{i \in I}G_i \mid x_i \in O_i \text{  for all but finitely many }i\in I\right\}
\]
with the group topology such that the natural embedding of $\prod_{i\in I}O_i$ with the product topology is continuous and open.
\end{defn}

A special case of a local direct product is when the factors are all isomorphic to some given topological group.  We write $A^{I}$ for the direct product of copies of $A$ indexed by $I$ and $A^{<I}$ for the direct sum of copies of $A$ indexed by $I$.

\begin{defn}
Suppose that $G$ is a topological group and that $(U_i)_{i \in I}$ is a family of open subgroups of $G$.  The \defbold{local direct power}\index{local direct power} of $G$ with respect to $(U_i)_{i\in I}$ is the local direct product
\[
G^{(U_i)}: = \bigoplus_{i \in I}(G,U_i).
\]
If $I$ is countable, we say $G^{(U_i)}$ is a \textbf{countable local direct power}.
\end{defn}

If the index set $I$ is finite, the local direct power $G^{(U_i)}$ is the direct product $G^{I}$.  However, as soon as $I$ is infinite, there will in general be many possible local direct powers of a given group, depending on the choice of $U_i$.  For instance, up to isomorphism the cyclic group $C_p$ of order $p$ has three countable local direct powers, according to whether $\triv$, $C_p$ itself, or both occur infinitely often in the tuple $(U_i)_{i\in I}$.  These three local direct powers are $C^{<\Nb}_p$, $C^{\Nb}_p$, and ${C^{\Nb}_p \times C^{<\Nb}_p}$.

Local direct powers are a general method of constructing new topologically characteristically simple groups starting from a given topologically characteristically simple group. Proving this shall require an easy observation; we leave the proof to the reader. For a set $I$, we denote the finitely supported permutations by $\FSym(I)$.

\begin{obs}
Let $G$ be a topological group and $(U_i)_{i\in I}$ be a family of open subgroups. Then, the local direct power $G^{(U_i)}$ admits a faithful action of $\FSym(I)$, acting by permuting the factors $G_i$, as well as a faithful action by the direct sum $\Aut(G)^{<I}$, each coordinate acting on the individual factors.
\end{obs}

\begin{prop}\label{prop:localdirect_charsimple} Let $G$ be a non-abelian topological group and $(U_i)_{i\in I}$ be a family of open subgroups.
\begin{enumerate}[(1) ]
\item If $G$ is topologically characteristically simple, then $G^{(U_i)}$ is topologically characteristically simple.
\item If $I$ is countable and $G$ is Polish and $A$-simple where $A$ is a countable group of automorphisms, then $G^{(U_i)}$ is a $B$-simple Polish group where $B$ is a countable group of automorphisms.
\end{enumerate}
\end{prop}

\begin{proof}
As the proofs are similar, we prove $(2)$. Set $P:=G^{(U_i)}$ and let $G_i\normal P$ be the copy of $G$ supported at $i\in I$. Define additionally $B: = \langle \FSym(I), A^{<I} \rangle\leq \Aut(P)$. Since $A$ is countable, the group $B$ is also countable.

We see that $P$ has a Polish open subgroup of countable index, namely $\prod_{i \in I}U_i$; thus $P$ is a Polish group.  We now argue that $P$ is $B$-simple. Suppose that $K$ is a non-trivial normal $B$-invariant subgroup of $P$; the subgroup $K$ is not central since $G$ has trivial center. There exists $i \in I$ such that $K$ does not centralize $G_i$, and thus, since $G_i$ and $K$ normalize each other and $G_i$ is non-abelian, $K$ must intersect $G_i$ non-trivially. That $G_i$ is $A$-simple ensures $K$ contains a dense subgroup of $G_i$.  Since $B$ acts transitively on the factors $G_i$, we deduce that $K$ contains a dense subgroup of $G_i$ for all $i \in I$, hence $K$ is dense in $P$.  This demonstrates that $P$ has no proper non-trivial closed normal $B$-invariant subgroup.
\end{proof}

\section{Groups of semisimple type}\label{sec:semisimple}
Topological groups that resemble products of topologically simple groups appear frequently in the study of Polish groups. The most well-known examples are the semisimple connected Lie groups. The local direct products mentioned above give additional examples. Here we adapt some definitions from finite group theory to the topological group setting, in order to isolate a class of groups that appears to capture the groups that are ``close to being products of simple groups." In Subsection~\ref{ssec:chief_block_structure}, we shall see that this class of groups plays a fundamental role in the structure of chief factors.

\begin{defn}
For a topological group $G$, a \defbold{component}\index{component} of $G$ is a closed subgroup $M$ of $G$ such that the following conditions hold:
\begin{enumerate}[(a) ]
\item $M$ is normal in $\ngrp{M}{G}$.
\item $M/\Z(M)$ is non-abelian.
\item Whenever $K$ is a proper closed normal subgroup of $M$, then $K$ is central in $\ngrp{M}{G}$. 
\end{enumerate}
The \defbold{layer}\index{layer} $E(G)$ of $G$ is the closed subgroup generated by the components of $G$. 
\end{defn}

Note that the conditions defining a component $M$ of $G$ are unchanged if we replace $M$ with $\theta(M)$ for $\theta$ a topological group automorphism of $G$, so the set of components is preserved by conjugation in $G$.  In particular, the layer is a closed normal (indeed, topologically characteristic) subgroup of $G$. 

\begin{defn}
A topological group $G$ is of \defbold{semisimple type}\index{semisimple type} if $G = E(G)$. We say $G$ is of \defbold{strict semisimple type}\index{strict semisimple type} if in addition $\Z(G) = \triv$.

A topological group $G$ is \defbold{topologically perfect}\index{topologically perfect} if $[G,G]$ is dense in $G$; in other words, whenever $\phi: G \rightarrow H$ is a continuous homomorphism to an abelian topological group $H$, then $\phi(G) = \triv$.
\end{defn}

Let us make some easy observations on these definitions.

\begin{lem}\label{lem:component}
Let $G$ be a topological group with $M$ a component of $G$. Then the following hold:
\begin{enumerate}[(1) ]
\item $\Z(M)$ is the unique largest proper closed normal subgroup of $M$.
\item $M/\Z(M)$ is topologically simple.
\item Both $M$ and $E(G)$ are topologically perfect.
\item $M$ is not contained in any other component of $G$.
\item Given a closed subgroup $H$ of $G$ such that $M \le H$, then $M$ is a component of $H$.  In particular, $E(G)$ is of semisimple type.
\item If $G$ is a centerless quasi-product of topologically simple groups, then $G$ is of strict semisimple type.
\end{enumerate}
\end{lem}

\begin{proof}
Condition (b) ensures that $\Z(M)$ is a proper normal subgroup of $M$; Lemma~\ref{lem:connected_commutator}(1) shows that $[M,\ol{\Z(M)}]$ is trivial, so $\ol{\Z(M)}$ is central in $M$, in other words, $\Z(M)$ is closed in $M$.  Let $N$ be a proper closed normal subgroup of $M$.  Then $N$ is central in $\ngrp{M}{G}$, and hence $N \le \Z(M)$.  Thus $\Z(M)$ is the unique largest proper closed normal subgroup of $M$, proving $(1)$.  Part $(2)$ follows immediately.

Since $M/\Z(M)$ is non-abelian, we have $\ol{[M,M]} \nleq \Z(M)$ and hence by $(1)$, $\ol{[M,M]} = M$, that is, $M$ is topologically perfect.  In particular, we have $M \le \ol{[E(G),E(G)]}$.  Since this inequality holds for every component and $E(G)$ is the closed subgroup generated by the components, it follows that $E(G) \le \ol{[E(G),E(G)]}$, that is, $E(G)$ is topologically perfect, proving $(3)$.

Suppose $M$ is contained in some component $M'$.  Since $M/\Z(M)$, and hence $M$, is non-abelian, we have $M \nleq \Z(M')$; hence by $(1)$ we have $M' = \ngrp{M}{M'}$, so $M'$ is a subgroup of $\ngrp{M}{G}$.  But then $M$ is normal in $M'$, so by $(1)$ again $M = M'$, proving $(4)$.

In the situation of $(5)$, consider $M$ as a subgroup of $H$ in comparison to $M$ as a subgroup of $G$.  We see that condition (b) is unchanged, and conditions (a) and (c) are also inherited because $\ngrp{M}{H} \le \ngrp{M}{G}$.  Thus $M$ is a component of $H$ and $(5)$ follows.

For $(6)$, suppose that $G$ is a centerless quasi-product of a set $\{M_i \mid i \in I\}$ of topologically simple groups, and consider $M_i$ for $i \in I$.  Then $M_i$ is closed and normal in $G$.  Since a quasi-product is in particular a generalized central product, we see that $\Z(M_i) \le \Z(G)$, so in fact $\Z(M_i) = \triv$, showing that $M_i$ is non-abelian.  Any proper closed normal subgroup of $M_i$ is trivial and hence central in $M_i$.  Thus $M_i$ is a component of $G$.  Since $G = \cgrp{M_i \mid i \in I}$ it follows that $G$ is of semisimple type; since $\Z(G) = \triv$, in fact $G$ is of strict semisimple type.  Thus $(6)$ holds.
\end{proof}

We will later see that the converse to Lemma~\ref{lem:component}(6) also holds, as stated in Theorem~\ref{thmintro:min_normal:simple}.  We will also show that the class of topological groups of semisimple type is closed under taking closed normal subgroups, Hausdorff quotients, and normal compressions.

\subsection{Properties of components}
We begin with a number of basic properties of groups of semisimple type.

\begin{lem}\label{lem:perfect_commutators}
Let $A$ and $B$ be closed subgroups of a topological group $G$.  If $[[A,B],B] = \triv$, then $[A,[B,B]] = \triv$. If $B$ is also topologically perfect, then $[A,B]=\triv$.
\end{lem}

\begin{proof}
The equation $[A,B] = [B,A]$ holds since $[a,b]\inv = [b,a]$ for any $a,b \in G$, so
\[
[[A,B],B] = [[B,A],B] = \triv.
\]
Appealing to the Three Subgroups Lemma (see for instance \cite[Lemma~8.27]{Is09}), it follows that $[[B,B],A] = \triv$, hence $[B,B]$ commutes with $A$, verifying the first claim. 

If $B$ is also topologically perfect, $\ol{[B,B]}=B$ commutes with $A$, and thus, $[A,B] = \triv$.
\end{proof}

\begin{lem}\label{lem:components:equiv}
Given a topological group $G$ and a component $M$ of $G$, then $[gMg\inv,M]=\triv$ for all $g \in G$ such that $M \neq gMg\inv$.
\end{lem}

\begin{proof}
Let $g \in G$ such that $M \neq gMg\inv$; we claim that $[gMg\inv,M] = \triv$.  By symmetry we may assume $M \nleq gMg\inv$.  Then $M$ and $gMg\inv$ are closed normal subgroups of $L$, so $gMg\inv \cap M$ is a closed normal subgroup of $L$ that contains $[gMg\inv,M]$ and is properly contained in $M$.  Thus
\[
[gMg\inv,M] \le gMg\inv \cap M \le \Z(M).
\]
We deduce that $[[gMg\inv,M],M] = \triv$, hence $[gMg\inv,M] = \triv$ by Lemma~\ref{lem:component}(3) and Lemma~\ref{lem:perfect_commutators}.
\end{proof}

\begin{prop}\label{prop:components}
Let $G$ be a topological group with $M$ a component of $G$.
\begin{enumerate}[(1) ]
\item Given a closed subgroup $K$ of $G$ that is normalized by $M$, either $M \le K$ or $K$ centralizes $M$. 
\item $M$ commutes with the group $R$ generated by all components of $G$ other than $M$ itself.  In particular, $M$ is normal in $E(G)$.
\item If $\mc{R}$ is a set of components of $G$ that generates $E(G)$ topologically, then $\mc{R}$ is the set of all components of $G$ and forms a generalized central factorization of $E(G)$. 
\end{enumerate}
\end{prop}

\begin{proof} Set $L:=\ngrp{M}{G}$ for the duration of this proof and note that $M \normal L$.

For $(1)$, let $K$ be a closed subgroup of $G$ normalized by $M$ such that $M \nleq K$.  By Lemma~\ref{lem:component}(3) we have $M = \ol{[M,M]}$, so there exist $m,n \in M$ such that $[m,n] \not\in K$.  Suppose that there is $k \in K$ such that $kMk\inv \neq M$.  The group $kMk\inv$ then commutes with $M$ by Lemma~\ref{lem:components:equiv}, so $[m,n] = [[m,k],n]$.  On the other hand, $[[m,k],n] \in K$ since $K$ is normalized by $M$, so $[m,n] \in K$, a contradiction.  So in fact $K$ normalizes $M$. The group $[K,M]$ is thus a proper normal subgroup of $M$, so $[K,M] \le \Z(L)$. We conclude $[[K,M],M]=\triv$, and applying Lemma~\ref{lem:perfect_commutators}, $[K,M]=\triv$. Claim $(1)$ is now verified.

For $(2)$, let $N$ be a component of $G$ different from $M$ and set $S: = \ngrp{N}{G}$. If $M \nleq S$, then $S$ centralizes $M$ by part $(1)$, so $N$ centralizes $M$.  If instead $M \leq S$, then $M$ normalizes $N$ but is not contained in $N$, so again part $(1)$ implies that $N$ centralizes $M$.  The component $M$ therefore commutes with every component other than itself.

Part $(3)$ follows from $(2)$: Distinct elements of $\mc{R}$ centralize each other and hence are normal in $E(G)$. Any component not in $\mc{R}$ commutes with $E(G)$, which contradicts the fact that components are non-abelian.
\end{proof}

Immediately from the previous proposition, we deduce the following:

\begin{cor}\label{cor:components_normal}
If $G$ is a topological group of semisimple type and $M$ is a component of $G$, then $G = \ol{M\CC_G(M)}$; in particular $M$ is normal in $G$.
\end{cor}

It is useful to note a couple of subquotients arising from components which are topologically simple.

\begin{prop}\label{prop:top_simple}
Suppose that $G$ is a topological group with $M$ a component of $G$ and set $R := \ol{M\CC_G(M)}$.  Then $\ol{M\Z(R)}/\Z(R)$ and $R/\CC_G(M)$ are non-abelian and topologically simple.
\end{prop}
\begin{proof}
We may assume without loss of generality that $G = R$. Set $C:=\CC_G(M)$ and $N:=\ol{M\Z(G)}$.

For a closed normal subgroup $L$ of $G$ such that $L > C$, the group $L$ does not centralize $M$, so $L \ge M$ by Proposition~\ref{prop:components}(1). Therefore, $L \ge \ol{MC} = G$, verifying that $G/C$ is topologically simple.

For $S$ a proper closed normal subgroup of $N$ such that $S \ge \Z(G)$, we have that $M \nleq S$, so $S$ centralizes $M$ by Proposition~\ref{prop:components}(1).  On the other hand, $N \le \CC_G(C)$, since $C$ commutes with both $M$ and $\Z(G)$.  The group $S$ thus centralizes both $M$ and $C$, so $S$ centralizes $\ol{MC} = G$. We deduce that $S = \Z(G)$, and thus, $N/\Z(G)$ is topologically simple.

Finally, all the quotients under consideration contain a non-trivial image of $M$ and are thus non-abelian.
\end{proof}

A duality between components and simple quotients of groups of semisimple type now follows.
\begin{cor}\label{cor:components_simple}
Suppose that $G$ is a topological group of semisimple type, $\mc{M}$ is the set of components of $G$, and $\mc{N}$ lists all closed $N\normal G$ so that $G/N$ is topologically simple. Then
\[
\mc{N} = \{\CC_G(M) \mid M \in \mc{M}\}. 
\]
\end{cor}

\begin{proof}
By Corollary~\ref{cor:components_normal} and Proposition~\ref{prop:top_simple}, every $M \in \mc{M}$ is such that the quotient $G/\CC_G(M)$ is non-abelian and topologically simple, hence $\CC_G(M) \in \mc{N}$. On the other hand, suppose that $N\in \mc{N}$. The components of $G$ generate $G$ topologically, so there is a component $M\in \mc{M}$ for which $M\nleq N$. Appealing to Proposition~\ref{prop:components}(1), $N\leq \CC_G(M)$. The quotient $G/N$ is topologically simple, so either $\CC_G(M)=N$ or $\CC_G(M)=G$. The latter case is impossible since $M$ is non-abelian. We thus deduce that $N=\CC_G(M)$.
\end{proof}

With groups of strict semisimple type, we obtain better control over the components.

\begin{lem}\label{lem:abelian_normal}
If $G$ is a topological group of semisimple type, then $G/\Z(G)$ has no non-trivial abelian normal subgroups.  In particular, every abelian normal subgroup of $G$ is central.
\end{lem}
\begin{proof}
By Corollary~\ref{cor:components_normal} and Proposition~\ref{prop:top_simple} the factor $\ol{M\Z(G)}/\Z(G)$ is non-abelian for every component $M$ of $G$.  In particular, if $A$ is a closed normal subgroup of $G$ such that $\Z(G)\leq A$ and $A/\Z(G)$ is abelian, then $A$ cannot contain any component of $G$.  It now follows by Proposition~\ref{prop:components}(1) that $A$ centralizes every component of $G$.  Since $G$ is topologically generated by components, we deduce that $A \le \Z(G)$, and thus, $A/\Z(G)$ is trivial. 

The second claim follows from the first since $\ol{B\Z(G)}/\Z(G)$ is abelian and normal in $G/\Z(G)$ for any abelian normal subgroup $B$ of $G$.
\end{proof}

\begin{prop}\label{prop:strict_normal}
Let $G$ be a topological group of strict semisimple type with $M$ a closed subgroup of $G$.  Then $M$ is a component of $G$ if and only if $M$ is a minimal closed normal subgroup of $G$.  Additionally, if $M$ is a component, then $M$ is topologically simple.
\end{prop}

\begin{proof}
If $M$ is a component of $G$, then Corollary~\ref{cor:components_normal} ensures $M$ is normal in $G$, so the center of $M$ is normal in $G$. Applying Lemma~\ref{lem:abelian_normal}, the center is indeed trivial.  The component $M$ is thus a non-abelian topologically simple group, and \textit{a fortiori} it is a minimal closed normal subgroup of $G$.

Conversely, if $M$ is a minimal closed normal subgroup of $G$, then $M$ is non-abelian by Lemma~\ref{lem:abelian_normal}, so by Proposition~\ref{prop:components}(1), $M$ must contain at least one component of $G$.  Since components of $G$ are normal, minimality of $M$ ensures that in fact $M$ is a component of $G$, completing the proof.
\end{proof}

\subsection{Normal subgroups of groups of semisimple type}\label{sec:semisimple_structure}

We have good control over the normal subgroups of a group of semisimple type.  In particular, semisimplicity is stable under quotients, and there is a natural connection between components of $G$ and components of a quotient of $G$.

\begin{thm}\label{thm:norm_sgrps}
Suppose that $G$ is a topological group of semisimple type, let $\mc{M}$ be the set of components of $G$, and let $K$ be a closed normal subgroup of $G$.
\begin{enumerate}[(1) ]
\item The subgroup $\ol{[K,G]}$ is topologically generated by components of $G$.  In particular, it is of semisimple type.
\item The quotient $G/K$ is of semisimple type.  Indeed, the set 
\[
\{\ol{MK}/K \mid M \in \mc{M}\text{ and } [M,K] = \triv\}
\]
is exactly the set of components of $G/K$ and generates $G/K$ topologically.
\item The quotient $G/\Z(G)$ is of strict semisimple type and admits the quasi-direct factorization $\mc{M}':= \{\ol{M\Z(G)}/\Z(G) \mid M \in \mc{M}\}$.
\end{enumerate}
\end{thm}
\begin{proof}
Let $\mc{M}$ be the collection of components of $G$ and define 
\[
\mc{L}:=\{M\in \mc{M}\mid M\leq K\}\text{ and }\mc{C}:=\{M\in \mc{M}\mid [M,K]=\triv\}.
\]
The set $\mc{M}$ is union of $\mc{L}$ and $\mc{C}$ by Proposition~\ref{prop:components}(1).

Letting $R:= \cgrp{\mc{L}}$, we see that $R \le K$ and that $R$ is topologically generated by components of $G$. In the quotient $G/R$, the group $K/R$ centralizes $\ol{MR}/R$ for all $M \in \mc{M}$ since $\mc{M}=\mc{L}\cup \mc{C}$. Observing that $G/R$ is topologically generated by the groups $\ol{MR}/R$ as $M$ ranges over $\mc{M}$, we deduce that $K/R$ is central in $G/R$, and thus, $\ol{[K,G]}\le R$.

On the other hand, since it is topologically generated by components, $R$ is itself of semisimple type.  In particular, $R$ is topologically perfect, ensuring that $R = \ol{[K,G]}$.  This completes the proof of $(1)$.

We claim that the set $\mc{C}/K := \{\ol{MK}/K \mid M \in \mc{C}\}$, which clearly generates a dense subgroup of $G/K$, consists of components of $G/K$. Taking $M \in \mc{C}$, Corollary~\ref{cor:components_normal} ensures $\ol{MK}/K$ is a normal subgroup of $G/K$.  The component $M$ is topologically perfect, so $\ol{MK}/K$ is also topologically perfect.  Since $K$ centralizes $M$, the quotient $\ol{MK}/K$ is non-trivial, and $\ol{MK}/K$ modulo its center is non-abelian. Finally, given a proper closed normal subgroup $S/K$ of $\ol{MK}/K$, the group $S$ does not contain $M$, so $S$ centralizes $M$ by Proposition~\ref{prop:components}(1). Therefore, $S/K$ is central in $G/K$.  We conclude $\ol{MK}/K$ is a component of $G/K$, and thus, $G/K$ is of semisimple type.  The set $\mc{C}/K$ accounts for all components of $G/K$ by Proposition~\ref{prop:components}, completing the proof of $(2)$.

By part $(2)$, the group $G/\Z(G)$ is of semisimple type, and $\mc{M}'$ is a generalized central factorization of $G/\Z(G)$. Appealing to Lemma~\ref{lem:abelian_normal}, we conclude $G/\Z(G)$ is centerless. Hence $G/\Z(G)$ is of strict semisimple type, and Proposition~\ref{prop:quasiproduct:centralizers} implies $\mc{M}'$ is a quasi-direct factorization of $G$, completing the proof of $(3)$.
\end{proof}

We can now derive our description of groups of (strict) semisimple type given in the introduction.

\begin{proof}[Proof of Theorem~\ref{thmintro:min_normal:simple}.]
Suppose that $G$ is of semisimple type. Proposition~\ref{prop:components} ensures that $\mc{M}$ is a generalized central factorization of $G$. It is clear that each $M\in \mc{M}$ is central-by-topologically simple.  If $G$ is of strict semisimple type, then $\mc{M}$ is a quasi-direct factorization of $G$ by Theorem~\ref{thm:norm_sgrps}(3), and Proposition~\ref{prop:top_simple} ensures that each $M \in \mc{M}$ is topologically simple.
\end{proof}

Let us note a characterization of characteristically simple groups of semisimple type that will be used later.

\begin{prop}\label{prop:semisimple:charsimple}
If $G$ is a non-abelian topologically characteristically simple group, then the following are equivalent:
\begin{enumerate}[(1) ]
\item $G$ is of semisimple type;
\item $G$ has a component;
\item $G$ has a minimal closed normal subgroup.
\end{enumerate}
\end{prop}

\begin{proof}
Suppose that $G$ is of semisimple type.  Since $\Z(G)$ is characteristic, $G$ is of strict semisimple type, which ensures, via Proposition~\ref{prop:strict_normal}, that every component is a minimal closed normal subgroup of $G$.  Hence, $(1)$ implies $(2)$ and $(3)$.

If $G$ has a component, then $E(G)$ is a non-trivial closed characteristic subgroup of $G$, so $E(G) = G$, which implies $G$ is of semisimple type. We thus deduce $(2)$ implies $(1)$.

Suppose finally there is a minimal closed normal subgroup $M$ of $G$.  The group 
\[
\cgrp{\psi(M)\mid \psi \in \Aut(G)}
\]
is then a non-trivial closed characteristic subgroup of $G$, and thus, it equals $G$. Since $M$ is a minimal normal subgroup of $G$, the intersection $M\cap \psi(M)=\{1\}$ for all $\psi(M)\neq M$, so $[M,\psi(M)]=\triv$ for all $\psi(M)\neq M$. We thus deduce that $\{\psi(M)\mid \psi \in \Aut(G)\}$ is a generalized central factorization of $G$.

The group $M$ is non-abelian from the fact that $G$ is non-abelian, and the minimality of $M$ ensures $\Z(M)= \triv$.  Since all the other $\psi(M)$ commute with $M$, every normal subgroup of $M$ is indeed normal in $G$. On the other hand, $M$ is a minimal closed normal subgroup of $G$, so any proper $G$-invariant closed subgroup of $M$ is trivial. We conclude that $M$ is a component of $G$.  Thus, $(3)$ implies $(2)$, completing the cycle of implications.
\end{proof}

\subsection{Polish groups of semisimple type}
Let us finally consider the class of Polish groups of semisimple type. As expected, the number of components is bounded.

\begin{lem}\label{lem:semisimple:countable}
If $G$ is a Polish group, then $G$ has at most countably many components.
\end{lem}

\begin{proof}Let $\mc{M}$ be the set of components of $G$.  By replacing $G$ with the closed subgroup $E(G)$, we may assume that $G$ is of semisimple type.

Applying Theorem~\ref{thm:norm_sgrps}(3), the group $G/\Z(G)$ is a quasi-product of 
\[
\mc{M}': = \{\ol{M\Z(G)}/\Z(G) \mid M \in \mc{M}\},
\]
so Proposition~\ref{prop:quasiproduct:countable:Polish} implies $|\mc{M}'|\leq \aleph_0$. It follows that $|\mc{M}'|=|\mc{M}|$, hence $|\mc{M}| \le \aleph_0$.
\end{proof}

Semisimplicity is preserved up to abelian error under normal compressions between Polish groups.

\begin{prop}\label{prop:semisimple:compressions}
Let $G$ and $H$ be Polish groups, let $\mc{M}_G$ and $ \mc{M}_H$ list the components of $G$ and $H$, and suppose that $\psi: G \rightarrow H$ is a normal compression.
\begin{enumerate}[(1) ]
\item If $G$ is of semisimple type, then $H$ is of semisimple type, and $\mc{M}_H=\{\ol{\psi(K)}\mid K\in \mc{M}_G\}$.
\item If $H$ is of semisimple type, then $\ol{[G,G]}$ is of semisimple type, $\psi(\ol{[G,G]})$ is dense and normal in $H$, and $\mc{M}_G=\{\ol{[\psi\inv(L),\psi\inv(L)]}\mid L\in \mc{M}_H\}$.
\end{enumerate}
\end{prop}

\begin{proof}
For $(1)$, let $K$ be a component of $G$. Since $K\normal G$, the subgroup $\psi(K)$ is normal in $H$ by Proposition~\ref{prop:normal_compression}.  In particular $\ol{\psi(K)}$ is normal in $H$, and that $K$ is a component of $G$ ensures $\ol{\psi(K)}$ is topologically perfect and $\ol{\psi(K)}/\Z(\ol{\psi(K)})$ is non-abelian.  Let $L$ be a proper closed normal subgroup of $\ol{\psi(K)}$.  Then $[L,\psi(K)]\leq L\cap \psi(K) < \psi(K)$, so $[L,\psi(K)]$ is central in $\psi(K)$; Lemma~\ref{lem:perfect_commutators} then implies $L \le \Z(\ol{\psi(K)})$.  Thus $\ol{\psi(K)}$ is a component of $H$.

It is now clear that
\[
H = \cgrp{\ol{\psi(K)} \mid K \in \mc{M}_G},
\]
so $H$ is of semisimple type.  We deduce via Proposition~\ref{prop:components} that the set $\{\ol{\psi(K)} \mid K \in \mc{M}_G\}$ contains every component of $H$, so $\mc{M}_H = \{\ol{\psi(K)} \mid K \in \mc{M}_G\}$, completing the proof of $(1)$.

\medskip

For $(2)$, let $L$ be a component of $H$.  The group $L$ is not central in $H$, so Lemma~\ref{lem:perfect_commutators} ensures $[L,\psi(G)]$ is also not central in $L$. We deduce that $L\cap \psi(G)$ is not central in $L$. As $L$ is a component, $L \cap \psi(G)$ must be dense in $L$. We thus obtain a normal compression from $K := \psi\inv(L)$ to $L$. 

Setting $D: = \ol{[K,K]}$, we see that $\psi(D)$ is dense in $L$, since $L$ is topologically perfect. We claim $D$ is a component of $G$; the only non-trivial condition to check is the third. Let $E$ be a proper closed normal subgroup of $D$.  Applying Proposition~\ref{prop:normal_compression} to the normal compression $K \rightarrow L$, we see that $\psi(E)$ is normal but not dense in $L$, so $\psi(E) \le \Z(L)$ and hence $E \le \Z(D)$.  Hence, $D$ is a component of $G$.

For each $L \in \mc{M}_G$, let $D_L: = \ol{[\psi\inv(L),\psi\inv(L)]}$ and set $S := \cgrp{D_L\mid L\in \mc{M}_G}$. Each $D_L$ is a component of $G$ by the previous paragraph, and $\psi(D_L)$ is dense in $L$.  We deduce that $\psi(S)$ is dense in $H$, and so $\ol{[G,G]} \le S$ by Proposition~\ref{prop:normal_compression}. On the other hand, $D_L\leq \ol{[G,G]}$ for each $L\in \mc{M}_G$, and thus, we indeed have $S=\ol{[G,G]}$. Each $D_L$ for $L\in \mc{M}_G$ is a component of $S$, so $S$ is of semisimple type. We have thus verified the first two claims of $(2)$. The same argument as in $(1)$ implies $\{D_L \mid L \in \mc{M}_G\}$ is precisely the set of components of $S$, finishing the proof.
\end{proof}

Given a normal compression $\psi: G \rightarrow H$ of Polish groups, Proposition~\ref{prop:semisimple:compressions} ensures that $\ol{\psi(E(G))} = E(H)$ and establishes a correspondence between the components of $G$ and those of $H$.  We will later establish a much more general correspondence of this nature between non-abelian chief factors (modulo a suitable notion of equivalence) of $G$ and those of $H$; see Theorem~\ref{thm:compression:chief}.

\section{Relationships between normal factors}\label{sec:associated}

Our general discussion of normal compressions and generalized direct products is now complete. At this point, we begin our study of normal factors of a topological group, with an emphasis on chief factors. 
 
 \begin{defn}
 	For a topological group $G$, a \defbold{normal factor}\index{normal factor} of $G$ is a quotient $K/L$ such that $K$ and $L$ are distinct closed normal subgroups of $G$ with $L < K$. We say $K/L$ is a (topological) \defbold{chief factor}\index{chief factor} if whenever $M$ is a closed normal subgroup of $G$ such that $L \le M \le K$, then either $M = L$ or $M = K$. 
 \end{defn}

\begin{rmk}
We consider a normal factor $K/L$ to be the group of $L$-cosets with representatives in $K$. As a consequence, normal factors $K/L$ and $N/M$ are equal if and only if $K=N$ and $L=M$.
\end{rmk}

\subsection{The association relation}

Suppose that $G$ is a Polish group and that $K, L$ are closed normal subgroups of $G$. As abstract groups, the second isomorphism theorem states that $KL/L\simeq K/K\cap L$. Unfortunately, this statement \textit{does not} hold in a topological sense in the setting of Polish groups; since $KL$ is not in general closed, $KL/L$ fails to be a Polish group. A motivation for the relation of association is to ``fix" the second isomorphism theorem for Polish groups. Instead of relating $K/K \cap L$ to $KL/L$, we relate it to $\ol{KL}/L$. This way we stay within the category of Polish groups.

We begin by studying the relationship between $\ol{KL}/L$ and $K/K\cap L$. 

\begin{lem}\label{lem:associated:secondiso}
Let $K$ and $L$ be closed normal subgroups of the topological group $G$. Then the map $\phi: K/(K \cap L) \rightarrow \ol{KL}/L$ via $\; k(K \cap L) \mapsto kL$ is a $G$-equivariant normal compression map.
\end{lem}

\begin{proof}
Let $\pi: G/(K \cap L) \rightarrow G/L$ by $g(K \cap L) \mapsto gL$. Equipping the groups $G/(K \cap L)$ and $G/L$ with the respective quotient topologies, the map $\pi$ is a quotient map, so in particular it is continuous. The restriction $\psi$ of $\pi$ to $K/K\cap L$ is then a continuous, injective group homomorphism. Restricting the range of $\psi$ to the closure of the image produces a normal compression map $\phi:K/K\cap L\rightarrow \ol{KL}/L$, which is obviously $G$-equivariant.
\end{proof}

Let $G$ be a topological group with normal factors $K_1/L_1$ and $K_2/L_2$. We say $K_2/L_2$ is an \defbold{internal compression}\index{internal compression} of $K_1/L_1$ if $K_2 = \ol{K_1L_2}$ and $L_1 = K_1 \cap L_2$.  The \defbold{internal compression map}\index{internal compression map} is defined to be $\phi: K_1/L_1 \mapsto K_2/L_2$ via $\phi(kL_1): = kL_2$.

We now introduce a symmetric relation on normal factors, which in particular relates $K_1/L_1$ and $K_2/L_2$ whenever $K_2/L_2$ is an internal compression of $K_1/L_1$.

\begin{defn}
Given a topological group $G$, we say the closed normal factors $K_1/L_1$ and $K_2/L_2$ are \defbold{associated}\index{association relation} to one another if the following equations hold:
\[
\ol{K_1L_2} = \ol{K_2L_1}; \; K_1 \cap \ol{L_1L_2} = L_1; \; K_2 \cap \ol{L_1L_2} = L_2.
\]
\end{defn}

As designed, the association relation captures internal compressions:
\begin{obs}\label{obs:compression:associated}
Let $G$ be a topological group and let $K_1/L_1$ and $K_2/L_2$ be normal factors of $G$. If $K_2/L_2$ is an internal compression of $K_1/L_1$, then $K_2/L_2$ is associated to $K_1/L_1$.
\end{obs}

If two normal factors are associated, it does not follow that one is an internal compression of the other.  However, the factors do have an internal compression in common.

\begin{lem}\label{lem:associated:common_compression}
Let $G$ be a topological group and let $K_1/L_1$ and $K_2/L_2$ be associated normal factors of $G$.  Setting $K: = \ol{K_1K_2}$ and $L := \ol{L_1L_2}$, the normal factor $K/L$ is an internal compression of both $K_1/L_1$ and $K_2/L_2$.
\end{lem}

\begin{proof}
Since $\ol{K_1L_2} = \ol{K_2L_1}$, we observe that
\[
K = \ol{K_1K_2} \subseteq \ol{K_1L_2K_2L_1} = \ol{K_1L_2} \subseteq K,
\]
hence $K = \ol{K_1L_2}$.

The definition of association gives that $L_1 = K_1 \cap \ol{L_1L_2} = K_1 \cap L$ and $K = \ol{K_1L_2} = \ol{K_1L}$. Therefore, $K/L$ is an internal compression of $K_1/L_1$. The same argument for $K_2/L_2$ shows $K/L$ is an internal compression of $K_2/L_2$.
\end{proof}

To distinguish normal factors in a way that is stable under internal compressions, we consider centralizers in the ambient group.  For a normal factor $K/L$ of $G$, the \defbold{centralizer}\index{normal factor centralizer} of $K/L$ in $G$ is
\[
\CC_G(K/L):=\{g\in G\mid \forall k\in K\; [g,k]\in L\}.
\]
The set $\CC_G(K/L)$ is a closed normal subgroup $G$, and $g \in \CC_G(K/L)$ if and only if $[g,k] \in L$ as $k$ ranges over a dense subset of $K$.  Given a subgroup $H$ of $G$, we put $\CC_H(K/L): = \CC_G(K/L) \cap H$.

\begin{lem}\label{lem:association_centralizer}
Let $K_1/L_1$ and $K_2/L_2$ be normal factors of $G$. If $K_1/L_1$ and $K_2/L_2$ are associated, then $\CC_G(K_1/L_1) = \CC_G(K_2/L_2)$.
\end{lem}
\begin{proof}
In view of Lemma~\ref{lem:associated:common_compression}, it suffices to consider the case that $K_2/L_2$ is an internal compression of $K_1/L_1$. Suppose this is so and take $h \in \CC_G(K_1/L_1)$. For all $k \in K_1$, we have $[h,k] \in L_1$, so $[h,k] \in L_2$ since $L_1 \le L_2$. We deduce further that $[h,kl] \in L_2$ for all $k \in K_1$ and $l \in L_2$. Since $K_2 = \ol{K_1L_2}$, it follows that $h$ centralizes a dense subgroup of $K_2/L_2$, so $h$ centralizes $K_2/L_2$. Thus, $h \in \CC_G(K_2/L_2)$.

Conversely, take $h \in \CC_G(K_2/L_2)$. Since $K_1 \le K_2$, we have $[h,k] \in L_2$ and $[h,k] \in K_1$ for all $k \in K_1$. By hypothesis, $L_1 = K_1 \cap L_2$, so we indeed have $[h,k] \in L_1$ for all $k \in K_1$. Thus, $h \in \CC_G(K_1/L_1)$.
\end{proof}

Lemma~\ref{lem:association_centralizer} admits a converse in the case that $K_1/L_1$ and $K_2/L_2$ are \textit{non-abelian chief factors}.

\begin{prop}\label{associated:centralizers:chief}
Let $G$ be a topological group and let $K_1/L_1$ and $K_2/L_2$ be non-abelian chief factors of $G$.  Then the following are equivalent:
\begin{enumerate}[(1) ]
\item $K_1/L_1$ is associated to $K_2/L_2$.
\item $\CC_G(K_1/L_1) = \CC_G(K_2/L_2)$.
\item $\ol{K_1L_2} = \ol{K_2L_1}$ and $\ol{K_1K_2} > \ol{L_1L_2}$.

\end{enumerate}
\end{prop}

\begin{proof}
$(1)\Rightarrow (2)$. Lemma~\ref{lem:association_centralizer} establishes this implication.

\

$(2)\Rightarrow (3)$. Suppose that $(2)$ holds. Since $K_1/L_1$ and $K_2/L_2$ are centerless, $\CC_{K_1}(K_1/L_1) = L_1$ and $\CC_{K_2}(K_2/L_2)=L_2$. We conclude that $\ol{L_1L_2}\leq \CC_G(K_1/L_1)$, and so
\[
\ol{L_1L_2} \cap K_1 = \CC_{K_1}(K_1/L_1) = L_1<K_1.
\]
Therefore, $\ol{L_1L_2}<\ol{K_1K_2}$.

We now suppose toward a contradiction that $\ol{K_1L_2}\neq\ol{K_2L_1}$; without loss of generality $ \ol{K_1L_2}\not\le\ol{K_2L_1}$.  The subgroup $K_1 \cap \ol{K_2L_1}$ is then a proper closed $G$-invariant subgroup of $K_1$ that contains $L_1$, so $K_1 \cap \ol{K_2L_1} = L_1$. In particular, $[K_1,K_2] \le L_1$, so 
\[ 
\ol{K_2L_1} \le \CC_G(K_1/L_1) = \CC_G(K_2/L_2).
\]
However, this is absurd as $K_2/L_2$ is non-abelian. We conclude $\ol{K_1L_2}=\ol{K_2L_1}$, verifying $(3)$.

\

$(3)\Rightarrow (1)$. Suppose that $(3)$ holds.  We see that $\ol{K_1K_2} = \ol{K_1L_2} = \ol{K_2L_1}$, hence $\ol{K_1L_2} > \ol{L_1L_2}$. Since $L_2 \le \ol{L_1L_2}$, it follows that $K_1 \not\le \ol{L_1L_2}$, so $K_1 \cap \ol{L_1L_2}$ is a proper subgroup of $K_1$.  Since $K_1 \cap \ol{L_1L_2}$ is closed and normal in $G$ and $K_1/L_1$ is a chief factor, in fact $K_1 \cap \ol{L_1L_2}= L_1$.  The same argument shows that $K_2 \cap \ol{L_1L_2} = L_2$.  Hence, $K_1/L_1$ and $K_2/L_2$ are associated.
\end{proof}

\begin{cor} 
For a topological group $G$, the relation of association is an equivalence relation on non-abelian chief factors.
\end{cor}

The association relation between \textit{abelian} chief factors is not an equivalence relation; see \S\ref{ex:kleinfour}.

\begin{rmk}
The association relation and normal compressions are closely related. Say that $\psi:G\rightarrow H$ is a normal compression of Polish groups and form the semidirect product $G\rtimes_{\psi}H$. One verifies that the obvious copy of $G$ and $(G\rtimes_{\psi}H)/\ker\pi\simeq H$ in $G\rtimes_{\psi}H$ are associated. Polish groups $G$ and $H$ which admit a normal compression map between them are thus associated in some larger Polish group. On the other hand, Lemma~\ref{lem:associated:common_compression} shows associated factors admit normal compressions to a common third factor.
\end{rmk}

\subsection{The role of chief factors in normal series}

We now show how a chief factor ``appears" in any normal series in a unique way, as stated in Theorem~\ref{thmintro:Schreier_refinement}.  In fact, we can prove a result for topological groups in general, although the conclusion is somewhat weaker compared to the case of Polish groups.

\begin{thm}\label{thm:unique_associate}
Let $G$ be a topological group, $K/L$ be a non-abelian chief factor of $G$, and
\[
\triv = G_0 \le G_1 \le \dots \le G_n = G
\]
be a series of closed normal subgroups in $G$.  Then there is exactly one $i \in \{0,\dots,n-1\}$ such that there exist closed normal subgroups $G_i \le B \le A \le G_{i+1}$ of $G$ for which $A/B$ is a normal factor associated to $K/L$.  Specifically, this occurs for the least $i \in \{0,\dots,n-1\}$ such that $G_{i+1} \not\le \CC_G(K/L)$.

If $G$ is Polish, then $A$ and $B$ can be chosen so that $A/B$ is a non-abelian chief factor.
\end{thm}

For the uniqueness part of the proof, we need to show that two distinct factors of a normal series cannot be associated to the same non-abelian chief factor.  In fact, this is a more general phenomenon, not limited to chief factors.

\begin{lem}\label{chief:associated:ordered}
Let $G$ be a topological group and let $K_1/L_1$ and $K_2/L_2$ be non-trivial normal factors of $G$ such that $K_1 \le L_2$.  Then there does not exist a normal factor $K_3/L_3$ of $G$ such that $K_3/L_3$ is associated to both $K_1/L_1$ and $K_2/L_2$. In particular, $K_1/L_1$ and $K_2/L_2$ are not associated to each other.

\end{lem}

\begin{proof}
Suppose that $K_3/L_3$ is a factor of $G$ associated to $K_1/L_1$. Since $K_1 \le L_2$, we see that 
\[
\ol{K_3L_1} = \ol{K_1L_3} \le \ol{L_2L_3},
\]
and hence $K_3 \le \ol{L_2L_3}$.  In particular,
\[
K_3 \cap \ol{L_2L_3} = K_3 > L_3,
\]
so $K_3/L_3$ cannot be associated to $K_2/L_2$.
\end{proof}

\begin{proof}[Proof of Theorem~\ref{thm:unique_associate}.]
The uniqueness of $i$ is guaranteed by Lemma~\ref{chief:associated:ordered}.\par

We now show existence. Let $\alpha: G \rightarrow \Aut(K/L)$ be the homomorphism induced by the conjugation action of $G$ on $K/L$. Since $K/L$ is centerless, the normal subgroup $\Inn(K/L)$ of $\Aut(K/L)$ is isomorphic (as an abstract group) to $K/L$ and the action of $\Aut(K/L)$ on $\Inn(K/L)$ by conjugation faithfully replicates the action of $\Aut(K/L)$ on $K/L$ by automorphisms; in particular, $\Inn(K/L)$ has trivial centralizer in $\Aut(K/L)$.  Given a non-trivial subgroup $N$ of $\Aut(G)$ normalized by $\Inn(K/L)$, then $[N,\Inn(K/L)]$ is non-trivial, and since $N$ and $\Inn(K/L)$ normalize each other we have $[N,\Inn(K/L)] \le N \cap \Inn(K/L)$.  Thus $N$ has non-trivial intersection with $\Inn(K/L)$.

Take $i$ minimal such that $G_{i+1} \not\le \CC_G(K/L)$. The group $\alpha(G_{i+1})$ is then non-trivial and normalized by $\Inn(K/L)$, so $\Inn(K/L) \cap \alpha(G_{i+1})$ is non-trivial.  Set 
\[
B := \CC_{G_{i+1}}(K/L), \; R := \alpha\inv \left(\Inn(K/L)\right) \cap G_{i+1}, \; \text{and }A :=\ol{[R,K]B}.
\]
The groups $A$ and $B$ are then closed normal subgroups of $G$ such that $G_i \le B \le A \le G_{i+1}$.

Since $\Inn(K/L) \cap \alpha(G_{i+1})$ is non-trivial, we see that there are non-trivial inner automorphisms of $K/L$ induced by the action of $R$, so $[R,K] \not\le L$.  Since $K/L$ is a chief factor of $G$, it follows that $K = \ol{[R,K]L}$, hence $\ol{AL} = \ol{KB}$, verifying one condition of association. For the other two conditions, observe that $\ol{LB}$ centralizes $K/L$ since $\CC_G(K/L)$ is closed and contains both $L$ and $B$.  Additionally, $\CC_K(K/L) = L$ since $K/L$ is centerless, and $\CC_A(K/L) = B$ by the definition of $B$.  From these we deduce that $A \cap \ol{LB} = B$ and $K \cap \ol{LB} = L$.  Thus, $A/B$ is associated to $K/L$.

Finally, suppose that $G$ is a Polish group.  Letting $C: = \CC_G(K/L)$ and $M := \ol{KC}$, we see that $K \cap C = L$ and $A \cap C = B$. Furthermore,
\[
M = \ol{KC} = \ol{KBC} = \ol{ALC} = \ol{AC}.
\]
Lemma~\ref{lem:associated:secondiso} now supplies $G$-invariant normal compression maps $\psi_1: K/L \rightarrow M/C$ and $\psi_2: A/B \rightarrow M/C$ such that $\psi_1(kL) = kC$ for all $k \in K$ and $\psi_2(aB) = aC$ for all $a \in A$.

By the fact that $\psi_1$ is a $G$-equivariant map with dense image, we have
\[
\CC_G(M/C) = \CC_G(\psi_1(K/L)) = \CC_G(K/L) = C,
\] 
so $M/C$ is centerless.  The factor $K/L$ is chief, so it has no proper $G$-invariant closed normal subgroups. Theorem~\ref{thm:compression_simple} now implies $M/C$ has no $G$-invariant closed normal subgroups; that is to say, $M/C$ is a chief factor of $G$.  Applying Theorem~\ref{thm:compression_simple} to $\psi_2$, we deduce that $D/B$ is the unique smallest non-trivial closed $G$-invariant subgroup of $A/B$ where $D: = \ol{[A,A]B}$.  In particular, $D/B$ is a chief factor of $G$.  The image $\psi_2(D/B)$ is then a non-trivial $G$-invariant subgroup of $M/C$; since $M/C$ is a chief factor, in fact $\psi_2(D/B)$ must be dense in $M/C$.  In particular, 
\[
\CC_G(D/B) = \CC_G(\psi_2(D/B)) = \CC_G(M/C).
\]
We conclude $\CC_G(D/B)=\CC_G(K/L)$, and hence $D/B$ is associated to $K/L$ by Proposition~\ref{associated:centralizers:chief}.  Replacing $A$ with $D$, we therefore obtain a factor of $G_{i+1}/G_i$ that is associated to $K/L$ and additionally is a chief factor of $G$.
\end{proof}

\begin{cor}
Let $G$ be a Polish group and suppose that $K/L$ is a non-abelian chief factor of $G$. If $M/N$ is a normal factor associated to $K/L$, then there are $M\leq B<A\leq N$ with $A,B$ closed normal subgroups of $G$ so that $A/B$ is a non-abelian chief factor associated to $K/L$.
\end{cor}

\begin{proof}
Plainly, $\triv \leq N<M\leq G$ is a normal series, so we may apply Theorem~\ref{thm:unique_associate}. We may thus insert normal subgroups $B<A$ into the series so that $A/B$ is a non-abelian chief factor associated to $K/L$.  Since $M/N$ is associated to $K/L$, Lemma~\ref{chief:associated:ordered} implies that $B$ and $A$ must be inserted between $N$ and $M$, verifying the corollary.
\end{proof}

\section{Chief blocks}\label{sec:chief_blocks}
In view of the theory of finite groups, a chief factor, when present, should be a ``basic building block" of a Polish group. Theorem~\ref{thm:unique_associate} suggests that it is indeed the \textit{association class} of a chief factor that is the ``basic building block" of a Polish group. We thus arrive at the following definitions:
 
\begin{defn}
For a topological group $G$, a \defbold{chief block}\index{chief block} of $G$ is an association class of non-abelian chief factors of $G$. We denote by $[K/L]$ the association class of a non-abelian chief factor $K/L$ of $G$.  The (possibly empty) collection of chief blocks of $G$ is denoted by $\mf{B}_G$.
\end{defn}

In this section, we study the chief blocks of a group $G$. In particular, we show $\mf{B}_G$ comes with a canonical partial order, find canonical representatives for chief block, and isolate an important set of chief blocks for later applications, the \emph{minimally covered} blocks.

\subsection{First definitions and properties}
By Proposition~\ref{associated:centralizers:chief}, the centralizer of a chief factor is an invariant of association. This suggests a definition on the level of chief blocks.

\begin{defn}
For a chief block $\mf{a}\in \mf{B}_G$ in a topological group $G$, the \defbold{centralizer}\index{chief block centralizer} of $\mf{a}$ is
\[
\CC_G(\mf{a}):=\CC_G(K/L)
\]
for some (any) representative $K/L$ of $\mf{a}$.  Given a subgroup $H$ of $G$, define $\CC_H(\mf{a}) := \CC_G(\mf{a}) \cap H$.
\end{defn}

We make a trivial but useful observation.

\begin{obs}\label{obs:cover_centralizer}
Suppose that $G$ is a topological group. If $A/B$ is a representative of $\mf{a}\in \mf{B}_G$, then $\CC_A(\mf{a}) = B$.
\end{obs}

A topological group is \textbf{monolithic}\index{monolithic} if there is a unique minimal non-trivial closed normal subgroup. The \textbf{socle}\index{socle} of a topological group is the subgroup generated by all minimal non-trivial closed normal subgroups.  (If there are no minimal non-trivial closed normal subgroups, then the socle is defined to be trivial.)  We now show that the quotient of $G$ by the centralizer of a chief block is monolithic, and the socle of this quotient furthermore provides a canonical representative of the block.

\begin{prop}\label{prop:upper-rep}
Let $G$ be a Polish group and let $\mf{a}\in \mf{B}_G$.
\begin{enumerate}[(1) ]
\item $G/\CC_G(\mf{a})$ is monolithic, and the socle $M/\CC_G(\mf{a})$ of $G/\CC_G(\mf{a})$ is a representative of $\mf{a}$.
\item  If $R/S \in \mf{a}$, then $M/\CC_G(\mf{a})$ is an internal compression of $R/S$.
\end{enumerate}
\end{prop}

\begin{proof}
Set $C:=\CC_G(\mf{a})$ and let $A/B$ be a representative of $\mf{a}$. We see that $A \not\le C$, and by Observation~\ref{obs:cover_centralizer}, $A \cap C = B$.  There is thus an internal compression map from $A/B$ to $\ol{AC}/C$, so $\ol{AC}/C$ is associated to $A/B$ by Observation~\ref{obs:compression:associated}. Lemma~\ref{lem:association_centralizer} ensures $C=\CC_G(\ol{AC}/C)$, so $\ol{AC}/C$ is a normal subgroup of $G/C$ with trivial centralizer. As internal compressions are $G$-equivariant, we deduce from Theorem~\ref{thm:compression_simple} that $\ol{AC}/C$ is indeed a non-abelian chief factor. That is to say, $\ol{AC}/C \in \mf{a}$. 

Given a closed normal subgroup $H$ of $G$ such that $H > C$, the group $H/C$ does not centralize $\ol{AC}/C$, so $\ol{AC}/C \cap H/C$ is non-trivial. That $\ol{AC}/C$ is a chief factor implies $\ol{AC}/C \le H/C$. We thus deduce that $G/C$ is monolithic with socle $M/C := \ol{AC}/C$ which is a representative of $\mf{a}$, proving (1).

Suppose that $R/S\in \mf{a}$. We may run the same argument as before to conclude that $\ol{RC}/C$ is the socle of $G/C$ and is an internal compression of $R/S$. The socle is unique, so indeed $M/C$ is an internal compression of $R/S$, verifying $(2)$.  \end{proof}

In view of the previous proposition, the following definition is sensible.

\begin{defn} Let $G$ be a topological group and $\mf{a}\in \mf{B}_G$. The \textbf{uppermost representative}\index{uppermost representative} of $\mf{a}$ is the socle of $G/\CC_G(\mf{a})$. Letting $G^{\mf{a}}$ be the preimage of the socle of $G/\CC_G(\mf{a})$ under the usual projection, we denote the uppermost representative by $G^{\mf{a}}/\CC_{G}(\mf{a})$.
\end{defn}

In addition to giving canonical representatives, centralizers induce a partial order on $\mf{B}_G$. 

\begin{defn}
For $\mf{a},\mf{b}\in \mf{B}_G$, we define $\mf{a} \le \mf{b}$ if $\CC_G(\mf{a}) \le \CC_G(\mf{b})$.
\end{defn}

The partial order on blocks tells us how representatives can appear in normal series relative to each other. Making this precise requires several definitions.

\begin{defn}
Let $\mf{a}$ be a chief block and let $N/M$ be a normal factor of a topological group $G$.  We say $N/M$ \defbold{covers}\index{covers} $\mf{a}$ if $\CC_G(\mf{a})$ contains $M$ but not $N$.
If $N/M$ does not cover $\mf{a}$, we say that $N/M$ \defbold{avoids}\index{avoids} $\mf{a}$. A block can be avoided in one of two ways: The factor $N/M$ is \defbold{below} $\mf{a}$ if $N \le \CC_G(\mf{a})$; equivalently, $G/N$ covers $\mf{a}$. The factor $N/M$ is \defbold{above} $\mf{a}$ if $M \not\le \CC_G(\mf{a})$; equivalently, $M/\triv$ covers $\mf{a}$. For a closed normal subgroup $N\normal G$, we will say that $N$ covers or avoids a chief block by seeing $N$ as the normal factor $N/\triv$.
\end{defn}

To clarify this terminology, we restate Theorem~\ref{thm:unique_associate} in terms of chief blocks.

\begin{prop}\label{prop:cover_realize}
Let $G$ be a topological group, $\mf{a}\in \mf{B}_G$, and
	\[
	\triv = G_0 \le G_1 \le \dots \le G_n = G
	\]
be a series of closed normal subgroups in $G$.  Then there is exactly one $i \in \{0,\dots,n-1\}$ such that $G_{i+1}/G_i$ covers $\mf{a}$. If $G$ is also Polish, then there exist closed normal subgroups $G_i \le B < A \le G_{i+1}$ of $G$ so that $A/B\in\mf{a}$.
\end{prop}

With these definitions in hand, we give a characterization of the partial ordering.

\begin{prop}\label{prop:block:ordering}
Let $G$ be a topological group with $\mf{a},\mf{b}\in\mf{B}_G$.
\begin{enumerate}[(1) ]
\item We have $\mf{a} \le \mf{b}$ if and only if every closed normal subgroup of $G$ that covers $\mf{b}$ also covers $\mf{a}$.
\item If $G$ is Polish, we have $\mf{a} < \mf{b}$ if and only if for every representative $A/B$ of $\mf{b}$, the subgroup $B$ covers $\mf{a}$.
\end{enumerate}
\end{prop}

\begin{proof}
Suppose that $\mf{a} \le \mf{b}$. Letting $K$ be a closed normal subgroup of $G$ that covers $\mf{b}$, we have $K \not\le \CC_G(\mf{b})$, so $K \not\le \CC_G(\mf{a})$. Hence, $K$ covers $\mf{a}$. Conversely, suppose that $\mf{a} \nleq \mf{b}$.  The group $\CC_G(\mf{a})$ is then a closed normal subgroup that covers $\mf{b}$ (since $\CC_G(\mf{a}) \not\le \CC_G(\mf{b})$) but does not cover $\mf{a}$.  This completes the proof of $(1)$.

For $(2)$, first suppose that $\mf{a}<\mf{b}$ and $A/B\in \mf{b}$. The subgroup $A$ covers $\mf{a}$ by part $(1)$, but since $\mf{b} \neq \mf{a}$, the factor $A/B$ avoids $\mf{a}$. It now follows $B$ covers $\mf{a}$, via Proposition~\ref{prop:cover_realize}.

Conversely, suppose that $\mf{a} \not< \mf{b}$.  If $\mf{a} = \mf{b}$, then $A/B$ covers $\mf{a}$, so $B$ does not cover $\mf{a}$.  Otherwise $\mf{a} \nleq \mf{b}$, and as we saw in the proof of part $(1)$, the group $\CC_G(\mf{a})$ covers $\mf{b}$ but not $\mf{a}$.  Applying Proposition~\ref{prop:cover_realize}, there is a representative $R/S$ of $\mf{b}$ such that $R \le \CC_G(\mf{a})$, and since $S \le \CC_G(\mf{a})$, we deduce that $S$ avoids $\mf{a}$, completing the proof of $(2)$.
\end{proof}

\subsection{Minimally covered chief blocks}\label{sec:minimially}

We here isolate an important subset of $\mf{B}_G$, which, when non-empty, gives more powerful tools with which to study the normal subgroup structure of a Polish group. Describing this subset requires an observation about chief blocks.

\begin{lem}\label{block:covering:intersection}
Suppose that $G$ is a topological group with $\mf{a}\in \mf{B}_G$ and let $\mc{K}_{\mf{a}}$ be the set of closed normal subgroups of $G$ that cover $\mf{a}$.  Then $\mc{K}_{\mf{a}}$ is a filter of closed normal subgroups.\end{lem}

\begin{proof}
Since $N$ covers $\mf{a}$ if and only if $N\nleq \CC_G(\mf{a})$, it is clear that $\mc{K}$ is upward-closed. It thus suffices to show that $\mc{K}_{\mf{a}}$ is closed under finite intersections. 

Let $K,L \in \mc{K}_{\mf{a}}$ and fix a representative $A/B$ of $\mf{a}$.  The subgroup $K$ does not centralize $A/B$, so $[A,K] \not\le B$. Since $A/B$ is a chief factor of $G$, we have $\ol{[A,K]B} = A$.  The subgroup $L$ also does not centralize $A/B$, so $L$ does not centralize the dense subgroup $[A,K]B/B$ of $A/B$. We deduce that $[[A,K],L] \not\le B$, so in fact $\ol{[[A,K],L]B} = A$. The group $\ol{[[A,K],L]}$ then covers $\mf{a}$, and since $\ol{[[A,K],L]} \le K \cap L$, we conclude that $K \cap L$ covers $\mf{a}$.
\end{proof}

In general, $\mc{K}_{\mf{a}}$ is \textit{not} a principal filter. For example, the free group on two generators $F_2$ is residually finite, so any chief block $\mf{a}$ represented by an infinite quotient of $F_2$ is such that $\mc{K}_{\mf{a}}$ is not principal; this follows from Theorem~\ref{thm:unique_associate} as an easy exercise. However, in the case that $\mc{K}_{\mf{a}}$ is a principal filter, much more can be said about the group and its normal subgroup structure.

\begin{defn} 
Given a topological group $G$ and $\mf{a}\in\mf{B}_{\mf{a}}$, define $G_{\mf{a}} := \bigcap_{K \in \mc{K}_{\mf{a}}}K$.  We say $\mf{a}$ is \defbold{minimally covered}\index{minimally covered} if $G_{\mf{a}}$ covers $\mf{a}$, \textit{i.e.}\ $\mc{K}_{\mf{a}}$ is a principal filter.  Write $\mf{B}^{\min}_G$ for the set of minimally covered chief blocks of $G$.
\end{defn}

\begin{rmk} 
The presence of minimally covered chief blocks is a finiteness condition. In \cite{RW_LC}, it is shown that minimally covered chief blocks arise in all locally compact Polish groups that have sufficiently non-trivial large-scale topological structure.
\end{rmk}

For a minimally covered chief block $\mf{a}$, we argue the subgroup $G_{\mf{a}}$ provides a canonical lowermost representative of $\mf{a}$. 

\begin{prop}\label{prop:minimalcovering}
Let $G$ be a topological group with $\mf{a} \in \mf{B}^{\min}_G$.
\begin{enumerate}[(1) ]
\item $G_{\mf{a}}/\CC_{G_{\mf{a}}}(\mf{a})$ is a representative of $\mf{a}$.
\item If $A/B\in \mf{a}$, then $A/B$ is an internal compression of $G_{\mf{a}}/\CC_{G_{\mf{a}}}(\mf{a})$.
\end{enumerate}
\end{prop}

\begin{proof}Let $K: = G_\mf{a}$.  By construction, $K$ is the unique smallest closed normal subgroup of $G$ that covers $\mf{a}$.  At the same time, $\CC_K(\mf{a})$ is the unique largest closed $G$-invariant subgroup of $K$ that avoids $\mf{a}$.  By Proposition~\ref{prop:cover_realize}, it follows that $K/\CC_K(\mf{a})$ covers $\mf{a}$, and the minimality and maximality of $K$ and $\CC_K(\mf{a})$ respectively ensure that $K/\CC_K(\mf{a})$ is a chief factor of $G$, so in fact $K/\CC_K(\mf{a}) \in \mf{a}$. We have thus verified $(1)$. 

Consider $A/B$ an arbitrary representative of $\mf{a}$.  The subgroup $A$ covers $\mf{a}$, so $A \ge K$. On the other hand, $B = \CC_A(\mf{a})$ by Observation~\ref{obs:cover_centralizer}, so $\CC_K(\mf{a}) = K \cap B$.  The natural homomorphism $\phi: K/\CC_K(\mf{a}) \rightarrow A/B $ by $k\CC_K(\mf{a}) \mapsto kB$ is consequently injective. As both $K/\CC_K(\mf{a})$ and $A/B$ are chief factors and $\phi$ is non-trivial, $\phi$ has dense image. The map $\phi$ is therefore an internal compression map, verifying $(2)$.
\end{proof}

\begin{defn}Let $G$ be a topological group and $\mf{a}\in \mf{B}^{\min}_G$. The \textbf{lowermost representative}\index{lowermost representative} of $\mf{a}$ is defined to be $G_{\mf{a}}/\CC_{G_{\mf{a}}}(\mf{a})$.
\end{defn}

We note an easy consequence of Proposition~\ref{prop:block:ordering} for minimally covered blocks.

\begin{cor}\label{cor:ordering:min_cover}
Let $G$ be a topological group with $\mf{a} \in\mf{B}_G$ and $\mf{b} \in \mf{B}^{\min}_G$.
\begin{enumerate}[(1) ]
\item We have $\mf{a} \le \mf{b}$ if and only if $G_{\mf{b}}$ covers $\mf{a}$.
\item We have $\mf{b} \le \mf{a}$ if and only if $G_{\mf{b}} \le G_{\mf{a}}$.
\end{enumerate}
\end{cor}

\begin{proof}
If $\mf{a} \le \mf{b}$, then $G_{\mf{b}}$ covers $\mf{a}$ by Proposition~\ref{prop:block:ordering}.  If $\mf{a} \nleq \mf{b}$, then there exists $K \normal G$ covering $\mf{b}$ but not $\mf{a}$; since $G_{\mf{b}} \le K$, it follows that $G_{\mf{b}}$ avoids $\mf{a}$, proving $(1)$.

If $\mf{b} \le \mf{a}$, then for every closed normal subgroup $K$ covering $\mf{a}$, the group $K$ covers $\mf{b}$, hence $K \ge G_{\mf{b}}$.  It now follows from the definition of $G_{\mf{a}}$ that $G_{\mf{a}} \ge G_{\mf{b}}$.  If $\mf{b} \nleq \mf{a}$, then there exists $K \normal G$ covering $\mf{a}$ but not $\mf{b}$, so $G_{\mf{b}} \nleq K$. Since $G_{\mf{a}} \le K$, we conclude that $G_{\mf{b}}\nleq G_{\mf{a}}$, proving $(2)$.
\end{proof}

We close this subsection by proving a version of Theorem~\ref{thm:unique_associate} for infinite series. 

\begin{prop}\label{prop:unique_associate:minimal}
Let $G$ be a Polish group, $\mf{a}$ be a minimally covered chief block of $G$, and $\mc{C}$ be a chain of closed normal subgroups of $G$ ordered by inclusion.  Then there is a downward-closed subset $\mc{D}$ of $\mc{C}$ and $A/B \in \mf{a}$ such that $R \le B$ for all $R \in \mc{D}$ and $S \ge A$ for all $S \in \mc{C} \setminus \mc{D}$.  Moreover, $\mc{D}$ is unique; specifically, $\mc{D} = \{R \in \mc{C} \mid R \le \CC_G(\mf{a})\}$.
\end{prop}

\begin{proof}
Put $\mc{D}: = \{R \in \mc{C} \mid R \le \CC_G(\mf{a})\}$ and $L: = \ol{\bigcup_{R \in \mc{D}}R}$.  Clearly $\mc{D}$ is downward-closed in $\mc{C}$.  We see that $L \le \CC_G(\mf{a})$, so $L$ avoids $\mf{a}$. On the other hand, for all $S \in \mc{C} \setminus \mc{D}$, it is the case that $S \nleq \CC_G(\mf{a})$, so $S$ covers $\mf{a}$.  Since $\mf{a}$ is minimally covered, we deduce that $K := \bigcap_{S\in \mc{C} \setminus \mc{D}}S$ covers $\mf{a}$. Proposition~\ref{prop:cover_realize} now ensures $K/L$ covers $\mf{a}$, so there exists $L \le B < A \le K$ such that $A/B \in \mf{a}$.  By the construction of $K$ and $L$, we have $R \le B$ for all $R \in \mc{D}$ and $S \ge A$ for all $S \in \mc{C} \setminus \mc{D}$.

The uniqueness of $\mc{D}$ follows from Theorem~\ref{thm:unique_associate}, noting that $R \in \mc{C}$ avoids $\mf{a}$ if and only if $R \in \mc{D}$.
\end{proof}

By starting with the trivial chain $\{1\}\leq G$ and repeatedly inserting representatives of chief factors into a chain of normal subgroups, one obtains the following corollary.

\begin{cor}
For $G$ a Polish group, there exists a chain $\mc{C}$ of closed normal subgroups of $G$ ordered by inclusion with the following property: there is an injective map ${\psi: \mf{B}_G^{\min} \rightarrow \mc{C}}$ such that for each $\mf{a} \in \mf{B}_G^{\min}$, the set $\{S \in \mc{C} \mid S > \psi(\mf{a})\}$ has a least element $\psi(\mf{a})^+$, and $\psi(\mf{a})^+/\psi(\mf{a})$ is a representative of $\mf{a}$.
\end{cor}

\subsection{Chief blocks of groups of semisimple type}

In general, the presence of minimally covered blocks is hard to ascertain.  However, they appear naturally in the context of components.

\begin{lem}\label{lem:component:min_cover}
Let $G$ be a topological group, let $\mc{C}$ be the set of components of $G$ and let $\mc{N}: = \{ \ngrp{M}{G} \mid M \in \mc{C}\}$.  Then there is an injective map $\Psi:\mc{N}\rightarrow \mf{B}^{\min}_G$ given by $K\mapsto[K/\Z(K)]$.  In particular, each $K \in \mc{N}$ covers exactly one chief block.
\end{lem}

\begin{proof}
Take $K\in \mc{N}$ and suppose that $K = \ngrp{M}{G}$ for some component $M$ of $G$. Any proper closed $G$-invariant subgroup $L$ of $K$ is so that $gMg\inv \nleq L$ for some $g \in G$, hence $gMg\inv \nleq L$ for all $g \in G$. Proposition~\ref{prop:components} now implies $L \le \Z(K)$. In particular, $\Z(K)$ is the unique largest proper $G$-invariant subgroup of $K$, so $K/\Z(K)$ is a chief factor.  Moreover, $K/\Z(K)$ is non-abelian, since it contains an injective image of $M/\Z(M)$. The factor $K/\Z(K)$ is thus a non-abelian chief factor of $G$, and $K$ cannot cover any other chief block, since $\Z(K)$ is abelian.

Setting $\mf{a} := [K/\Z(K)]$, the closed normal subgroups of $G$ covering $\mf{a}$ form a filter. On the other hand, $K$ admits no $G$-invariant subgroup which is non-central, so any normal subgroup covering $\mf{a}$ contains $K$. That is to say, $\mf{a}$ is minimally covered. The map $\Psi$ is thus well-defined with image in $\mf{B}_G^{\min}$.  For $K\neq J\in \mc{M}$, we see $[K,J]=\triv$ by Proposition~\ref{prop:components}, so $K\leq \CC_G([J/\Z(J)])$. We deduce that $[K/\Z(K)]\neq [J/\Z(J)]$ and $\Psi$ is injective.
\end{proof}

\begin{prop}\label{prop:semisimple_type:blocks}
Let $G$ be a topological group of semsimple type with $\mc{M}$ the set of components of $G$. 
\begin{enumerate}[(1) ]
\item There is a bijection $\Psi:\mc{M}\rightarrow \mf{B}_G$ by $K\mapsto[K/Z(K)]$.
\item $\mf{B}_G=\mf{B}_G^{\min}$, and $\mf{B}_G$ is an antichain.
\end{enumerate}
\end{prop}

\begin{proof}
Take $K\in \mc{M}$.  By Corollary~\ref{cor:components_normal}, $K$ is normal in $G$, and Lemma~\ref{lem:component:min_cover} ensures $\Psi:K\mapsto [K/Z(K)]$ is well-defined and injective with image in $\mf{B}_G^{\min}$. 

We now argue $\Psi$ is surjective. Let $\mf{a}\in \mf{B}_G$ and $A/B \in \mf{a}$. By Theorem~\ref{thm:norm_sgrps}, the group $L = \ol{[A,G]}$ is topologically generated by components of $G$.  Since $A/B$ is non-abelian, $L \nleq B$.  There thus exists a component $K$ of $G$ such that $K \le L \le A$, but $K \nleq B$.  Since $A/B$ is a chief factor of $G$, it follows that $A = \ol{KB}$, and $K$ covers $[A/B]$.  Since $K$ only covers one chief block by Lemma~\ref{lem:component:min_cover}, we conclude that $[A/B]=[K/\Z(K)]$, so $\Psi$ is surjective, verifying $(1)$. 

As the image of $\Psi$ consists of minimally covered blocks, we deduce that $\mf{B}_G^{\min}=\mf{B}_G$.
Following from Proposition~\ref{prop:block:ordering}, the block $[K/\Z(K)]$ is a minimal element of $\mf{B}_G$, hence $\mf{B}_G$ is an antichain. We have thus demonstrated $(2)$.
\end{proof}

Minimally covered chief blocks give a second characterization of Polish groups of (strict) semisimple type.

\begin{thm}\label{thm:semisimple_quot}
Let $G$ be a Polish group and $\mc{N}$ list all closed $N\normal G$ so that $G/N$ is non-abelian and topologically simple, and $[G/N]$ is minimally covered.
\begin{enumerate}[(1) ]
\item The group $G$ is of semisimple type if and only if $\bigcap\mc{N} =\Z(G)$ and every proper closed normal subgroup of $G$ is contained in some $N \in \mc{N}$.
\item The group $G$ is of strict semisimple type if and only if  $\bigcap\mc{N} = \triv$ and every proper closed normal subgroup of $G$ is contained in some $N \in \mc{N}$.
\end{enumerate}
\end{thm}

\begin{proof}
Part $(2)$ is clear given part $(1)$, so it suffices to prove $(1)$.

Let $\mc{M}$ be the set of components of $G$ and set $R:=\bigcap \mc{N}$.  Note that $\Z(G) \le R$ since $G/N$ is centerless for all $N \in \mc{N}$.

Suppose first that $G$ is of semisimple type. Proposition~\ref{prop:semisimple_type:blocks} ensures all chief blocks are minimally covered, so Corollary~\ref{cor:components_simple} implies $\CC_G(M) \in \mc{N}$ for all $M \in \mc{M}$.  As $G$ is topologically generated by $\mc{M}$, we deduce that $\bigcap_{M \in \mc{M}}\CC_G(M) = \Z(G)$. The group $R$ is thus central in $G$, and it follows $R=\Z(G)$.

For the second claim, take $S$ a proper closed normal subgroup of $G$.  There is some component $M$ of $G$ such that $M \nleq S$.  By Proposition~\ref{prop:components}, $S$ centralizes $M$, and so $S \le \CC_G(M) \in \mc{N}$.

\

Conversely, suppose that $\bigcap\mc{N} =\Z(G)$ and every proper closed normal subgroup of $G$ is contained in some $N \in \mc{N}$. If $\mc{N} = \{N\}$, we have $N = \Z(G)$, and $G$ itself is a component. Plainly then, $G$ is of semisimple type. We thus suppose that $|\mc{N}| \ge 2$.

Fix $N\in \mc{N}$, set $\mf{a}:=[G/N]$, and put $K:=\bigcap\mc{N}\setminus\{N\}$. Since $\mf{a}$ is minimally covered and each $N'\in \mc{N}\setminus\{N\}$ covers $\mf{a}$, the group $K$ covers $\mf{a}$. Letting $G\rightarrow G/N$ be the usual projection, the restriction $K/K\cap N\rightarrow G/N$ is a normal compression. Given the definition of $K$, we further have that $L:=K\cap N\leq \Z(G)$. Theorem~\ref{thm:compression_simple} now implies that $\ol{[K,K]L}/L$ is a topologically simple group, and it follows that $\ol{[K,K]L}$ is a component of $G$.  In particular, note that $\ol{[K,K]} \nleq L$, implying that $\ol{[K,K]}L \nleq N$, and hence $\ol{[K,K]}L$ covers $\mf{a}$.

The previous paragraph ensures that the closed normal subgroup $E(G)$ generated by all the components covers all chief blocks $[G/N]$ for $N\in \mc{N}$. Thus $E(G)$ is not contained in $N$ for any $N\in \mc{N}$. Given our hypothesis that every proper closed normal subgroup of $G$ is contained in some $N \in \mc{N}$, we conclude that $E(G)=G$. That is to say, $G$ is of semisimple type, completing the proof.
\end{proof}

\section{The chief block space}\label{sec:morphisms}

We now define and study a space that captures more of the chief factor data of a Polish group. In the previous section, we saw that $\mf{B}_G$ comes with a partial ordering which controls how chief factors can be ordered in a normal series. The group theoretic structure of the chief factors, however, is not recorded by $\mf{B}_G$. We thus step back and consider the set of non-abelian chief factors as a preorder, where the ordering is given by comparing centralizers.

\begin{defn}
For $G$ a Polish group, the \textbf{chief block space}\index{chief block space} of $G$ is the collection $X_G$ of non-abelian chief factors of $G$, equipped with the preorder given by setting $K_0/L_0\leq K_1/L_1$ if $\CC_G(K_0/L_0)\leq \CC_G(K_1/L_1)$.
\end{defn}

Taking the quotient of $(X_G,\leq)$ to make a partial ordering recovers the partially ordered collection of chief blocks $\mf{B}_G$. We stress that the chief block space is more than a preorder, since we also want to keep track of the chief factors as topological groups. This datum is captured by the notion of morphism.

\begin{defn}\label{defn:block_morphism}
	For $G$ and $H$ Polish groups, a \textbf{strong/covariant/contravariant block morphism}\index{block morphism} $\psi:(X_G,\leq)\rightarrow (X_H,\leq)$ is an homomorphism of preorders such that for each $A/B\in X_G$ the following holds, respectively:
	\begin{enumerate}[$\bullet$]
		\item $\psi(A/B)\simeq A/B$ as topological groups. (strong)
		\item There is a normal compression $A/B\rightarrow\psi(A/B)$. (covariant)
		\item There is a normal compression $\psi(A/B)\rightarrow A/B$. (contravariant)
	\end{enumerate}
When we say \textbf{block morphism}\index{monomorphism}, we mean any one of the three types.
\end{defn}

If $\psi:(X_G,\leq)\rightarrow (X_H,\leq)$ is a strong/covariant/contravariant block morphism, then $\psi$ induces a map of partially ordered sets $\wt{\psi}:\mf{B}_G\rightarrow \mf{B}_H$. If the induced map $\wt{\psi}:\mf{B}_G\rightarrow \mf{B}_H$ is injective, we say that $\psi$ is a \textbf{strong/covariant/contravariant block monomorphism}. For the notion of a block isomorphism, we place additional restrictions. 

\begin{defn}
Let $G$ and $H$ be Polish groups. A (covariant) contravariant block morphism $\psi:(X_G,\leq)\rightarrow (X_H,\leq)$ is a \textbf{(covariant) contravariant block isomorphism}\index{chief block space isomorphism} if there is a (contravariant) covariant block morphism $\chi:(X_H,\leq)\rightarrow (X_G,\leq)$ such that $\CC_G(\chi\circ \psi(A/B))=\CC_G(A/B)$ for $A/B\in X_G$ and $\CC_H(\psi\circ\chi(A/B))=\CC_H(A/B)$ for $A/B\in X_H$.  We call the block morphism $\chi$ an \textbf{inverse block morphism} of $\psi$.
\end{defn}

The condition on centralizers in the previous definition is equivalent to the induced maps $\wt{\chi}:\mf{B}_H\rightarrow \mf{B}_G$ and $\wt{\psi}:\mf{B}_G\rightarrow \mf{B}_H$ being such that $\wt{\chi}\circ\wt{\psi}=\mathrm{id}_{\mf{B}_G}$ and $\wt{\psi}\circ\wt{\chi}=\mathrm{id}_{\mf{B}_H}$. As a consequence, the partially ordered sets $\mf{B}_G$ and $\mf{B}_H$ are isomorphic as partial orders. For $\psi:(X_G,\leq)\rightarrow (X_H,\leq)$ a block isomorphism, we stress that $\psi$ need not even be a bijection. Additionally, an inverse block morphism need not be unique.

\subsection{Block space morphisms via group homomorphisms}
Various group homomorphisms induce block space morphisms. The first, easiest example of this phenomenon is for surjective maps.

\begin{prop}\label{prop:quot_block_morphism} 
Suppose that $G$ and $H$ are Polish groups. If $\phi:G\rightarrow H$ is a continuous, surjective homomorphism, then there is a strong block monomorphism $\psi:(X_H,\leq)\rightarrow (X_G,\leq)$ defined by $\psi(K/L):=\phi^{-1}(K)/\phi^{-1}(L)$.
\end{prop}

\begin{proof}
For $K/L$ a chief factor of $H$, the factor $\phi^{-1}(K)/\phi^{-1}(L)$ is a chief factor of $G$, so the map $\psi:X_H\rightarrow X_G$ is well-defined on the level of sets. Furthermore, $\psi(K/L)$ is isomorphic to $K/L$ for all $K/L\in X_H$.

Taking a non-abelian chief factor $K/L$ of $H$, we observe that $\phi\inv(\CC_H(K/L)) = \CC_G(\psi(K/L))$, hence for a second non-abelian chief factor $K'/L'$ of $H$,
\[
\CC_H(K/L) \le \CC_H(K'/L') \Leftrightarrow \CC_G(\psi(K/L)) \le \CC_G(\psi(K'/L')).
\]
It now follows that $\psi$ is a preorder homomorphism and furthermore that $\psi$ induces an embedding of $\mf{B}_H$ into $\mf{B}_G$.
\end{proof}

\begin{rmk} 
For each Polish group $G$ and $N\normal G$ a closed normal subgroup, Proposition~\ref{prop:quot_block_morphism} supplies an embedding of partial orders $\mf{B}_{G/N}\rightarrow \mf{B}_G$. On the other hand, $\mb{F}_{\omega}$, the free group on countably many generators, surjects onto every countable group. The partial order on $\mf{B}_{\mb{F}_{\omega}}$ must therefore be extremely rich. One naturally asks if this partial order has some universality properties. Perhaps naively, does every countable partial order embed into $\mf{B}_{\mb{F}_{\omega}}$? One can ask similar questions for a surjectively universal Polish group, \cite{D12}.
\end{rmk}

We now consider the more difficult case of normal compressions. Let us begin by observing a general lemma. Suppose that $G$ and $H$ are Polish groups with $\psi:G\rightarrow H$ a normal compression. Corollary~\ref{cor:normal_invar_equi_action} ensures every closed normal subgroup of $G$ is invariant under the $\psi$-equivariant action of $H$. For a normal factor $A/B$ of $G$, the group $H$ thus has an action on $A/B$ induced by the $\psi$-equivariant action: $h.aB:=(h.a)B$. This action has a useful feature:

\begin{lem}\label{lem:factor_kernel_centralizer}
Suppose that $G$ and $H$ are Polish groups with $\psi:G\rightarrow H$ a normal compression. If $A/B$ is a normal factor of $G$, then under the $\psi$-equivariant action of $H$ on $A/B$, the following holds: $\psi(g)\in \ker(H\acts A/B)$ if and only if $g\in \CC_G(A/B)$.
\end{lem}
\begin{proof}
For all $g\in G$ and $a\in A$, observe that
\[
[g,a]=\psi\inv([\psi(g),\psi(a)])= (\psi(g).a)a\inv
\]
where $\psi(g).a$ denotes the $\psi$-equivariant action. We infer that $\psi(g)\in \ker(H\acts A/B)$ if and only if $g\in \CC_G(A/B)$.
\end{proof}

We now state and prove the desired theorem.
\begin{thm}\label{thm:compression:chief}
Suppose that $G$ and $H$ are Polish groups. If $\psi: G \rightarrow H$ is a normal compression, then there is a contravariant block isomorphism $\phi:(X_H,\leq)\rightarrow (X_G,\leq)$ defined by
\[
\phi:K/L\mapsto \frac{\ol{[\psi^{-1}(K),\psi^{-1}(K)]\psi^{-1}(L)}}{\psi^{-1}(L)}
\]
with inverse $\chi:(X_G,\leq)\rightarrow (X_H,\leq)$ defined by
\[
\chi:A/B\mapsto \frac{\ol{\psi(A)}}{\CC_{\ol{\psi(A)}}\left(\ol{\psi(A)}/\ol{\psi(B)}\right)}.
\]
\end{thm}

To prove Theorem~\ref{thm:compression:chief}, we shall prove a series of lemmas and propositions. For the remainder of this subsection, it is assumed that $G$ and $H$ are Polish groups and that $\psi: G \rightarrow H$ is a normal compression. Some notation must also be established.

\begin{notation}\label{notation:compressions} For a chief factor $K/L$ of $H$, we set
\[
D:= \ol{[\psi^{-1}(K),\psi^{-1}(K)]\psi^{-1}(L)}.
\]
For $A\subseteq G$, we write $\widehat{A}$ for $\ol{\psi(A)}$. For a chief factor $A/B$ of $G$, we set
\[
R:=\CC_{\wh{A}}(\wh{A}/\wh{B}).
\]
\end{notation}
We begin our proof by verifying that $\phi$ is indeed a contravariant block morphism.

\begin{lem}\label{lem:phi_map}
For $K/L\in X_H$, the following hold:

\begin{enumerate}[(1) ]
\item The factor $D/\psi\inv(L)$ is a non-abelian chief factor of $G$.\label{lem:compression_1}
\item The restriction of $\psi$ gives an $H$-equivariant normal compression map\label{lem:compression_2} $D/\psi\inv(L)\rightarrow K/L$.
\item  Under the $\psi$-equivariant action of $H$ on $G$,\label{lem:compression_3}
\[
\ker(H\acts D/\psi\inv(L)) = \CC_H(K/L).
\] 
\end{enumerate}
\end{lem}

\begin{proof}
Suppose first for contradiction that $(\psi(G)L) \cap K = L$.  The groups $\psi(G)L/L$ and $K/L$ are then normal subgroups of $H/L$ with trivial intersection, so they commute. Since $\psi(G)L/L$ is dense in $H/L$, it follows that $K/L$ is central in $H/L$.  This is absurd as we have assumed that $K/L$ is non-abelian. We deduce that $(\psi(G)L) \cap K > L$. Since $K/L$ is a chief factor of $H$, it follows that $(\psi(G)L) \cap K = (\psi(G) \cap K)L$ is dense in $K$.

The factor $K/L$ is centerless and $H$-simple, and the map
\[
\psi':\psi\inv(K)/\psi\inv(L)\rightarrow K/L \text{ via } k\psi\inv(L) \mapsto \psi(k)L
\]
is an $H$-equivariant map where $H\acts \psi\inv(K)/\psi\inv(L)$ by the $\psi$-equivariant action. The previous paragraph ensures that $\psi'$ is also a normal compression. Applying Theorem~\ref{thm:compression_simple}, we conclude that $D/\psi\inv(L)$ is $H$-simple, and it follows that $D/\psi\inv(L)$ is a non-abelian chief factor of $G$. We have thus established claims $(1)$ and $(2)$.

Taking $h \in \CC_H(K/L)$, we have that $[h,k] \in L$ for all $k \in K$, so $[h,\psi(d)] \in L \cap \psi(G)$ for all $d \in D$.  Therefore, $h \in \CC_H(\psi(D)/(\psi(G) \cap L)) = \ker(H\acts D/\psi\inv(L))$.  On the other hand, if $h \in \ker(H\acts D/\psi\inv(L))$, then $h \in \CC_H(\psi(D)L/L)$, so $h$ centralizes a dense subgroup of $K/L$. We infer that $h \in \CC_H(K/L)$, proving that $\CC_H(K/L) = \ker(H\acts D/\psi\inv(L))$.
\end{proof}

\begin{prop}\label{prop:covariant:block}
The map $\phi:(X_H,\leq)\rightarrow (X_G,\leq)$ is a contravariant block morphism.
\end{prop}
\begin{proof}
Lemma~\ref{lem:phi_map} implies that $\phi:X_H\rightarrow X_G$ is well-defined and that there  is a normal compression $\phi(K/L)\rightarrow K/L$ for each $K/L\in X_H$. It remains to show that $\phi$ respects the preorder.

Say that $K/L$ and $K'/L'$ in $X_H$ are such that $\CC_H(K/L)\leq \CC_H(K'/L')$. Lemma~\ref{lem:phi_map} ensures that
\[
\ker(H\acts \phi(K/L))\leq \ker(H\acts \phi(K/L)),
\]
and from Lemma~\ref{lem:factor_kernel_centralizer}, we deduce that $\CC_G(\phi(K/L))\leq \CC_G(\phi(K'/L'))$. That is to say, $\phi$ is a preorder homomorphism.\end{proof}

We now argue $\chi$ is a covariant block morphism. Recalling Notation~\ref{notation:compressions}, let us begin with an analogue of Lemma~\ref{lem:phi_map} for $\chi$.

\begin{lem}\label{lem:chi_map}
For $A/B\in X_G$, the following hold:
\begin{enumerate}[(1) ]
\item The factor $\wh{A}/ R$ is a non-abelian chief factor of $H$.\label{lem:inv_compression_1}
\item The restriction of $\psi$ gives an $H$-equivariant normal compression map\label{lem:inv_compression_2} $\psi:A/B\rightarrow\widehat{A}/ R$.
\item  Under the $\psi$-equivariant action of $H$ on $G$,\label{lem:inv_compression_3}
\[
\CC_H(\widehat{A}/R) = \CC_H(\widehat{A}/\widehat{B}) = \ker(H\acts A/B).
\] 
\end{enumerate}
\end{lem}

\begin{proof}
Consider the normal subgroups $\wh{A}$ and $\wh{B}$ of $H$. If $\widehat{A} = \wh{B}$, then Proposition~\ref{prop:normal_compression} implies $ B\ge \ol{[A,A]}$ as $\psi:A\rightarrow \wh{A}$ is a normal compression, and this is absurd since $A/B$ is non-abelian. We conclude that $\widehat{A} > \widehat{B}$, so $\widehat{B}$ does not contain $\psi(A)$.  The subgroup $C: = \psi\inv(\wh{B}) \cap A$ is thus a proper $G$-invariant closed subgroup of $A$, and seeing as $A/B$ is a chief factor of $G$, in fact $C = B$.  The image of $C$ under $\psi$ is then $\widehat{B} \cap \psi(A)$, and thus, we deduce that $\widehat{B} \cap \psi(A) = \psi(B)$.

We now have a non-trivial normal factor $\wh{A}/\wh{B}$ of $H$, and the induced map $\psi:A/B\rightarrow \wh{A}/\wh{B}$ is an $H$-equivariant normal compression. Passing to $\wh{A}/R$ for $R:=\CC_{\wh{A}}(\wh{A}/\wh{B})$, the induced map $\psi:\wh{A}/\wh{B}\rightarrow \wh{A}/R$ is again an $H$-equivariant normal compression, verifying $(2)$. The subgroup $R/\wh{B}$ is the center of $\wh{A}/\wh{B}$, and since $A/B$ is $H$-simple, $\wh{A}/R$ is also $H$-simple by Theorem~\ref{thm:compression_simple}. The factor $\wh{A}/R$ is thus a non-abelian chief factor of $H$, establishing $(1)$. 

Letting $E: = \CC_H(\widehat{A}/R)$, we have $[E/\wh{B},\widehat{A}/\wh{B}] \le R/\wh{B}$, so $[[E/\wh{B},\wh{A}/\wh{B}],\wh{A}/\wh{B}] =\{1\}$. Since $\wh{A}/\wh{B}$ is topologically perfect, Lemma~\ref{lem:perfect_commutators} implies that indeed $[E,\wh{A}] \le \wh{B}$, so $E \le \CC_H(\wh{A}/\wh{B})$.  On the other hand, $\widehat{A}/R$ is a quotient of $\widehat{A}/\widehat{B}$, so any element of $H$ that centralizes $\widehat{A}/\widehat{B}$ also centralizes $\widehat{A}/R$.  Hence, $\CC_H(\widehat{A}/R) = \CC_H(\widehat{A}/\widehat{B})$.  

Let $h \in H$.  Under the $\psi$-equivariant action, $\psi((h.a)a\inv)=[h,\psi(a)]$ for any $a\in A$.  In particular, we see that $[h,\psi(a)] \in \psi(B)$ for all $a \in A$ if and only if $(h.a)a\inv \in B$ for all $a \in A$, in other words $\CC_H(\psi(A)/\psi(B))= \ker(H\acts A/B)$.  If $h \in \CC_H(\psi(A)/\psi(B))$, then $h$ centralizes a dense subgroup of $\wh{A}/\wh{B}$, so $h \in \CC_H(\wh{A}/\wh{B})$.  Conversely if $h \in \CC_H(\wh{A}/\wh{B})$, then $[h,\psi(a)] \in \wh{B} \cap \psi(A) = \psi(B)$ for all $a \in A$, so $h \in \CC_H(\psi(A)/\psi(B))$. The equality $\CC_H(\wh{A}/\wh{B}) = \ker(H\acts A/B)$ thus holds, verifying $(3)$.
\end{proof}

\begin{prop} 
The map $\chi:(X_G,\leq)\rightarrow (X_H,\leq)$ is a covariant block morphism.
\end{prop}
\begin{proof}
Lemma~\ref{lem:chi_map} implies that $\chi:X_G\rightarrow X_H$ is well-defined and that there is a normal compression $A/B\rightarrow \chi(A/B)$ for each $A/B\in X_G$. It remains to show that $\chi$ respects the preorder.

Say that $A/B$ and $A'/B'$ in $X_G$ are such that $\CC_G(A/B)\leq \CC_G(A'/B')$ and let 
\[
S := \ker(H\acts A/B)=\CC_H(\widehat{A}/R)=\CC_H(\chi(A/B)),
\]
where the middle equality is given by Lemma~\ref{lem:chi_map}. Lemma~\ref{lem:factor_kernel_centralizer} ensures that $S \cap \psi(G) \le \psi(\CC_G(A/B))$.  The image $\psi(A')$ is normal in $H$ by Proposition~\ref{prop:normal_compression}, so  $[S,\psi(A')] \le S \cap \psi(A')$; in particular, $[S,\psi(A')]\leq \psi(\CC_G(A/B))$. 

Since $\CC_G(A/B) \le \CC_G(A'/B')$, a second application of Lemma~\ref{lem:factor_kernel_centralizer} now gives 
\[
[S,\psi(A')] \le \ker(H\acts A'/B').
\]
Thus, $[[S,\psi(A')],\psi(A')] \le \psi(B')$.  Taking closures in $H$, we indeed have $[[S,\wh{A'}],\wh{A'}] \le \wh{B'}$.  The factor $\wh{A'}/\wh{B'}$ is topologically perfect, so Lemma~\ref{lem:perfect_commutators} implies $[S,\wh{A'}] \le \wh{B'}$. We thus deduce that
\[ 
S \le \CC_H(\widehat{A'}/\widehat{B'})=\CC_H(\wh{A'}/R')=\CC_H(\chi(A'/B')),
\]
where the middle equality follows from Lemma~\ref{lem:chi_map}. We conclude that $\chi$ respects the preordering.
\end{proof}

Concluding that $\chi$ is the desired inverse is now easy.

\begin{lem}\label{lem:compression:chief:order_isomorphism}
The block morphism $\chi$ is an inverse of $\phi$.
\end{lem}

\begin{proof}
Let us first consider $\chi\circ\phi$. Taking $K/L\in X_H$, the map $\phi$ sends $K/L$ to $D/\psi\inv(L)$. Appealing to Lemma~\ref{lem:phi_map}, we have that $\CC_H(K/L)=\ker(H\acts D/\psi\inv(L))$. The map $\chi$ sends $D/\psi\inv(L)$ to $\wh{D}/R$ where $R = \CC_{\wh{D}}(\wh{D}/\wh{\psi\inv(L)})$, and Lemma~\ref{lem:chi_map} implies that $\CC_H(\wh{D}/R)=\ker(H\acts D/\psi\inv(L))$. We conclude that 
\[
\CC_H(K/L)=\CC_H(\chi\circ\phi(K/L)).
\]

We now consider $\phi\circ \chi$. A similar argument as in the previous paragraph yields that
\[
\ker(H\acts A/B)=\ker(H\acts \phi\circ\chi (A/B))
\]
for any $A/B\in X_G$. From Lemma~\ref{lem:factor_kernel_centralizer}, we infer that $\CC_G(A/B)=\CC_G(\phi\circ\chi(A/B))$.
\end{proof}

The proof of Theorem~\ref{thm:compression:chief} is now complete. 

\subsection{Minimally covered blocks under block space morphisms}

We conclude the section by showing that minimally covered blocks are preserved by normal compressions; this will be used later to prove Theorem~\ref{thm:prop:type_invariant}. For a minimally covered block $\mf{a}\in \mf{B}_G$, recall that $G_{\mf{a}}$ is the unique smallest closed normal subgroup covering $\mf{a}$.
 
\begin{prop}\label{prop:compression_minimally_covered}
Suppose that $G$ and $H$ are Polish groups with a normal compression $\psi: G \rightarrow H$. Let $\phi: (X_H,\leq)\rightarrow (X_G,\leq)$ and $\chi:(X_G,\leq)\rightarrow (X_H,\leq)$ be the canonical block isomorphism and inverse and let $\wt{\phi}:\mf{B}_H\rightarrow \mf{B}_G$ and $\wt{\chi}:\mf{B}_G\rightarrow\mf{B}_H$ be the induced partial order isomorphisms.
\begin{enumerate}[(1) ]
\item If $\mf{a}\in \mf{B}_H$ is minimally covered, then $\wt{\phi}(\mf{a})$ is minimally covered in $G$, and 
\[
G_{\wt{\phi}(\mf{a})} = \ol{[\psi\inv(H_{\mf{a}}),\psi\inv(H_{\mf{a}})]}.
\]
\item  If $\mf{b}\in \mf{B}_G$ is minimally covered, then $\wt{\chi}(\mf{b})$ is minimally covered in $H$, and 
\[
H_{\wt{\chi}(\mf{b})} = \ol{\psi(G_{\mf{b}})}.
\]
\end{enumerate}
In particular, $\mf{a}\in \mf{B}_H$ is minimally covered if and only if $\wt{\phi}(\mf{a})\in \mf{B}_G$ is minimally covered.
\end{prop}

\begin{proof}
Suppose that $\mf{a}\in \mf{B}_H$ is minimally covered. Since the lowermost representative $H_{\mf{a}}/\CC_{H_{\mf{a}}}(\mf{a})$ is topologically perfect, the subgroup $[H_{\mf{a}},H]\CC_{H_{\mf{a}}}(\mf{a})$ is dense in $H_{\mf{a}}$, and \textit{a fortiori} $\ol{[H_{\mf{a}},H]}\nleq \CC_{H_{\mf{a}}}(\mf{a})$. We conclude that $\ol{[H_{\mf{a}},H]}$ covers $\mf{a}$, but since $H_{\mf{a}}$ is minimal, it is indeed the case that $\ol{[H_{\mf{a}},H]}=H_{\mf{a}}$.  The subgroup $[H_{\mf{a}},\psi(G)]$ is then dense in $H_{\mf{a}}$, so $H_{\mf{a}} \cap \psi(G)$ is dense in $H_{\mf{a}}$. The map $\psi$ thus restricts to a normal compression $\psi:\psi\inv(H_{\mf{a}})\rightarrow H_{\mf{a}}$.

Let $K: = \ol{[\psi\inv(H_{\mf{a}}),\psi\inv(H_{\mf{a}})]}$; from the definition of $\phi$, we see that $K$ covers $\wt{\phi}(\mf{a})$.  If $N\normal G$ is some closed subgroup of $\psi\inv(H_{\mf{a}})$ that covers $\wt{\phi}(\mf{a})$, then  $N \not\le \CC_G(\wt{\phi}(\mf{a}))$. From Lemmas~\ref{lem:phi_map} and \ref{lem:factor_kernel_centralizer}, we infer that $\psi(N) \not\le \CC_{H_{\mf{a}}}(\mf{a})$, and as $H_{\mf{a}}$ is the smallest closed normal subgroup of $H$ not contained in $\CC_{H}(\mf{a})$, the subgroup $\psi(N)$ is dense in $H_{\mf{a}}$. Proposition~\ref{prop:normal_compression} now ensures that $K\leq N$. We conclude that $\wt{\phi}(\mf{a})$ is minimally covered and that $\ol{[\psi\inv(H_{\mf{a}}),\psi\inv(H_{\mf{a}})]}$ is the smallest closed normal subgroup that covers $\wt{\phi}(\mf{a})$, verifying $(1)$.

For part $(2)$, suppose that $\mf{b}\in \mf{B}_G$ is minimally covered. The factor $G_{\mf{b}}/\CC_{G_{\mf{b}}}(\mf{b})$ is a representative of $\mf{b}$, so $\wh{G_{\mf{b}}}:=\ol{\psi(G_{\mf{b}})}$ covers $\wt{\chi}(\mf{b})$. Consider $M\normal H$ a closed subgroup of $\wh{G_{\mf{b}}}$ that also covers $\wt{\chi}(\mf{b})$; equivalently, $M \not\le \CC_{\wh{G_{\mf{b}}}}(\wt{\chi}(\mf{b}))$. By Lemma~\ref{lem:chi_map}, $M\nleq \ker(H\acts G_{\mf{b}}/\CC_{G_{\mf{b}}}(\mf{b}))$, hence $[M,\psi(G_{\mf{b}})] \not\le \psi(\CC_{G_{\mf{b}}}(\mf{b}))$.  In particular, $M \cap \psi(G_{\mf{b}}) \not\le \psi(\CC_{G_{\mf{b}}}(\mf{b}))$.  The subgroup $\psi\inv(M) \cap G_{\mf{b}}$ is then a closed $G$-invariant subgroup of $G_{\mf{b}}$ that is not contained in $\CC_{G_{\mf{b}}}(\mf{b})$, so it covers $\mf{b}$. In view of the minimality of $G_{\mf{b}}$, it is indeed the case that $\psi\inv(M) \cap G_{\mf{b}}= G_{\mf{b}}$.  We conclude that $M \ge \psi(G_{\mf{b}})$, and so $M = \widehat{G_{\mf{b}}}$. The block $\wt{\chi}(\mf{b})$ is thus minimally covered, and $\wh{G_{\mf{b}}}=H_{\wt{\chi}(\mf{b})}$, verifying $(2)$.

The final assertion follows since $\wt{\phi}\circ\wt{\chi}=\mathrm{id}_{\mf{B}_G}$ and $\wt{\chi}\circ \wt{\phi}=\mathrm{id}_{\mf{B}_H}$.
\end{proof}

\section{Extension of chief blocks}
We here relate chief blocks of groups to chief blocks of subgroups.  In particular, we consider the chief block structure of a normal subgroup to obtain a division of chief factors into three basic types. The primary tool for these explorations is a notion of extension for chief blocks.

\subsection{First properties of extensions}
\begin{defn}
Let $G$ be a topological group with $H$ a closed subgroup of $G$.  For $\mf{a} \in \mf{B}_H$ and $\mf{b} \in \mf{B}_G$, we say $\mf{b}$ is an \defbold{extension} of $\mf{a}$ in $G$ if for every closed normal subgroup $K$ of $G$, the subgroup $K$ covers $\mf{b}$ if and only if $K \cap H$ covers $\mf{a}$.
\end{defn}

It is not in general clear when extensions exist.  However, if an extension exists, then it is unique.

\begin{lem}\label{lem:extension_uniqueness}
Let $G$ be a topological group with $H$ a closed subgroup of $G$ and $\mf{a} \in \mf{B}_H$. If $\mf{b},\mf{c} \in \mf{B}_G$ are both extensions of $\mf{a}$, then $\mf{b} = \mf{c}$.
\end{lem}

\begin{proof}
Take $K/L \in \mf{b}$.  Since $K$ covers $\mf{b}$, the intersection $K \cap H$ covers $\mf{a}$, and thus, $K$ covers $\mf{c}$. On the other hand, $L$ avoids $\mf{b}$, and the same argument ensures $L$ avoids $\mf{c}$.  Applying Theorem~\ref{thm:unique_associate}, there exists $L \le B < A \le K$ such that $A/B$ is associated to a representative of $\mf{c}$.  Since $K/L$ is a chief factor of $G$, we indeed have $L = B$ and $A = K$, hence  $\mf{b} = \mf{c}$.
\end{proof}

\begin{defn}
Let $G$ be a topological group with $H$ a closed subgroup of $G$.  If $\mf{a} \in \mf{B}_H$ has an extension in $G$, we say $\mf{a}$ is \defbold{extendable to $G$} and write $\mf{a}^G$ for the extension of $\mf{a}$ in $G$.
\end{defn}


We extract the following useful observation implicit in the proof of Lemma~\ref{lem:extension_uniqueness}:
\begin{obs}\label{obs:extensions}
	Let $G$ be a topological group with $H$ a closed subgroup and suppose that $\mf{a}\in H$ is extendable to $G$. Then, a normal factor $K/L$ of $G$ covers $\mf{a}^G$ if and only if $K\cap H/L\cap H$ covers $\mf{a}$.
\end{obs}

Extensions are also transitive when they exist.

\begin{lem}\label{lem:extensions_transitive}
Let $A \le B \le G$ be closed subgroups of the topological group $G$ and suppose that $\mf{a} \in \mf{B}_A$ is extendable to $B$.  Then $\mf{a}$ is extendable to $G$ if and only if $\mf{a}^B$ is extendable to $G$. If $\mf{a}$ is extendable to $G$, then $\mf{a}^G = (\mf{a}^B)^G$.
\end{lem}

\begin{proof}
Suppose that $\mf{a}$ is extendable to $G$. For a closed normal subgroup $K$ of $G$, $K\cap A$ covers $\mf{a}$ if and only if $K$ covers $\mf{a}^G$. At the same time, $K \cap B$ is a closed normal subgroup of $B$, so $K\cap A$ covers $\mf{a}$ if and only if $K\cap B$ covers $\mf{a}^B$. We thus deduce that $K$ covers $\mf{a}^G$ if and only if $K\cap B$ covers $\mf{a}^B$, hence $\mf{a}^G = (\mf{a}^B)^G$.

The argument for the converse is similar.
\end{proof}

Let us note an obvious restriction on centralizers.

\begin{lem}\label{lem:centralizers} 
	Suppose that $G$ is a topological group with $H$ a closed subgroup. If $K/L$ is a normal factor of $G$ so that $K\cap H/L\cap H$ covers $\mf{a}\in \mf{B}_H$, then $\CC_H(K/L)\leq \CC_H(\mf{a})$.
\end{lem}
\begin{proof}
	Any $h\in H$ centralizing $K/L$ must centralize $K\cap H/L\cap H$. Since $\mf{a}$ is covered by $K\cap H/L\cap H$, the element $h$ centralizes $\mf{a}$, hence $\CC_H(K/L)\leq \CC_H(\mf{a})$. 
\end{proof}

Observation~\ref{obs:extensions} can be reworked into a criterion for the existence of extensions in the case of minimally covered blocks.

\begin{lem}\label{lem:extension:min_covered}
Let $G$ be a topological group with $H$ a closed subgroup of $G$ and $\mf{a} \in \mf{B}^{\min}_H$.  Then $\mf{a}$ is extendable to $G$ if and only if there exists $\mf{b} \in \mf{B}^{\min}_G$ such that $(G_{\mf{b}} \cap H)/\CC_{(G_{\mf{b}} \cap H)}(\mf{b})$ covers $\mf{a}$.  If $\mf{a}$ is extendable to $G$, then $\mf{a}^G = \mf{b}$ and $G_{\mf{b}} = \ngrp{H_{\mf{a}}}{G}$.
\end{lem}

\begin{proof}
Suppose that $\mf{a}$ is extendable to $G$. Let $\mf{b} := \mf{a}^G$ and take $\mc{K}$ to be the set of closed normal subgroups of $G$ which cover $\mf{a}^G$.  For each $K \in \mc{K}$, the intersection $K \cap H$ covers $\mf{a}$, and since $\mf{a}$ is minimally covered,
\[
\bigcap_{K \in \mc{K}}(K \cap H) = \left(\bigcap_{K \in \mc{K}}K\right) \cap H
\]
covers $\mf{a}$. We conclude $\bigcap_{K \in \mc{K}}K$ covers $\mf{b}$, and thus, $\mf{b}$ is minimally covered.  It is now clear that $G_{\mf{b}} = \ngrp{H_{\mf{a}}}{G}$ and that $(G_{\mf{b}} \cap H)/\CC_{(G_{\mf{b}} \cap H)}(\mf{b})$ covers $\mf{a}$.

Conversely, suppose that there exists $\mf{b} \in \mf{B}^{\min}_G$ with lowermost representative $G_{\mf{b}}/L$ such that $(G_{\mf{b}} \cap H)/(L \cap H)$ covers $\mf{a}$. Let $H_{\mf{a}}$ be the lowermost representative of $\mf{a}$ and put $M:=\ngrp{H_{\mf{a}}}{G}$. Consider the normal series $\triv < M \le G$. In view of Theorem~\ref{thm:unique_associate}, exactly one of the factors $M/\triv$ and $G/M$ covers $\mf{b}$.  If $G/M$ covers $\mf{b}$, we have $M \le \CC_G(\mf{b})$, and hence
\[
H_{\mf{a}} \le M \cap H \le \CC_H(\mf{b})\leq \CC_H(\mf{a}),
\]
where Lemma~\ref{lem:centralizers} ensures that $\CC_H(\mf{b})\leq \CC_H(\mf{a})$. This is absurd as $H_{\mf{a}}$ covers $\mf{a}$. We conclude that the subgroup $M$ covers $\mf{b}$. 

Since $G_{\mf{b}}/L$ is the lowermost representative of $\mf{b}$, it is the case that $M \ge G_{\mf{b}}$.  On the other hand, $G_{\mf{b}} \cap H$ covers $\mf{a}$, so $G_{\mf{b}} \ge H_{\mf{a}}$. It then follows that $M = G_{\mf{b}}$.  Given a closed normal subgroup $S$ of $G$, we now have the following equivalences:
\[
S \text{ covers } \mf{b} \Leftrightarrow S \ge G_{\mf{b}} \Leftrightarrow S \cap H \ge H_{\mf{a}} \Leftrightarrow S \cap H \text{ covers } \mf{a}.
\]
Therefore, $\mf{b} = \mf{a}^G$ as required.
\end{proof}

In particular, note that if $\mf{a} \in \mf{B}^{\min}_H$ is extendable to $G$, then $\mf{a}^G$ is also minimally covered.

\subsection{Chief factors extended from normal subgroups}
Suppose that $G$ is a topological group with $H\normal G$ a closed normal subgroup. The group $G$ has an action on $\mf{B}_H$ by $g.[K/L]:=[gKg^{-1}/gLg^{-1}]$. One easily verifies from the definition that this action is well-defined and preserves the subset $\mf{B}^{\min}_H$.

\begin{prop}\label{prop:induced_block}
Let $G$ be a topological group with $H$ a closed normal subgroup of $G$ and $\mf{a} \in \mf{B}^{\min}_H$.  Then $\mf{a}$ is extendable to $G$.  The lowermost representative of $\mf{a}^G$ is a closed normal factor $M/N$ of $H$, where
\[
M := \ngrp{H_{\mf{a}}}{G}\quad \text{ and } \quad N := \bigcap_{g \in G}g\CC_G(\mf{a})g\inv \cap M.
\]
\end{prop}

\begin{proof}
It is easy to see that $N < M$. Additionally, the group $H_{\mf{a}}/(H_{\mf{a}}\cap N)$ is a subgroup of $M/N$, and $H_{\mf{a}}/(H_{\mf{a}}\cap N)$ is non-abelian since it admits $H_{\mf{a}}/C_{H_{\mf{a}}}(\mf{a}) $ as a quotient. The factor $M/N$ is therefore a non-abelian normal factor of $G$.

Consider $L$ a closed $G$-invariant subgroup of $H$ such that $L \ngeq M$.  There exists $g \in G$ such that $gH_{\mf{a}}g\inv \nleq L$, and since $\mf{a} \in \mf{B}_H^{\min}$, the subgroup $gH_{\mf{a}}g\inv$ is the least normal subgroup of $H$ covering the chief block $g.\mf{a}$. In particular, $L$ avoids $g.\mf{a}$, and consequently $L \le \CC_H(g.\mf{a})$. Seeing that $\CC_H(g.\mf{a}) = g\CC_H(\mf{a})g\inv$, we deduce that $L \le k\CC_H(\mf{a})k\inv$ for all $k \in G$, and thus, $L \le N$.

It now follows that $M/N$ is a non-abelian chief factor of $G$ and that $M$ is the unique smallest closed normal subgroup of $G$ that covers $\mf{b}:=[M/N]$.  That is to say, $M/N$ is the lowermost representative of the chief block $\mf{b}$.  In view of Lemma~\ref{lem:extension:min_covered}, we deduce that $\mf{b}$ is the extension of $\mf{a}$ to $G$, completing the proof.
\end{proof}

The extension of chief blocks from $H$ to $G$ thus produces an equivalence relation on $\mf{B}^{\min}_H$ where two chief blocks are equivalent if they have the same extension to $G$.  It turns out that this relation is completely determined by the structure of $\mf{B}^{\min}_H$ as a poset together with the action of $G$ on this poset.

\begin{defn}
Let $G$ be a topological group and let $H$ be a closed normal subgroup of $G$.  Define a preorder $\preceq_G$ on $\mf{B}^{\min}_H$ by setting $\mf{a} \preceq_G \mf{b}$ if there exists $g \in G$ such that $g.\mf{a} \le \mf{b}$. If $\mf{a} \preceq_G \mf{b}$ and $\mf{b} \preceq_G \mf{a}$, we say $\mf{a}$ and $\mf{b}$ are in the same \defbold{$G$-stacking class}\index{stacking class} and write $\mf{a} \sim_G \mf{b}$. 
\end{defn}

There are two kinds of $G$-stacking class.
\begin{defn}
Let $G$ be a topological group, let $H$ be a closed normal subgroup of $G$, and let $\mc{S}$ be a $G$-stacking class of $\mf{B}^{\min}_H$. We say $\mc{S}$ is an \textbf{antichain orbit}\index{antichain orbit} if all elements of $\mc{S}$ are pairwise incomparable in $\mf{B}^{\min}_H$, and $G$ acts transitively on $\mc{S}$. We say $\mc{S}$ is a \textbf{proper stacking class}\index{proper stacking class} if for all $\mf{a},\mf{b} \in \mc{S}$, there exists $g \in G$ such that $g.\mf{a} < \mf{b}$.
\end{defn}

\begin{lem}\label{lem:stackingclass_types}
Let $G$ be a topological group, let $H$ be a closed normal subgroup of $G$, and let $\mc{S}$ be a $G$-stacking class of $\mf{B}^{\min}_H$. Then exactly one of the following holds:
\begin{enumerate}[(1) ]
\item $\mc{S}$ is an antichain orbit;
\item $\mc{S}$ is a proper stacking class.
\end{enumerate}
\end{lem}

\begin{proof}
It is clear that the two cases are mutually exclusive. 

Suppose that $\mc{S}$ is an antichain and let $\mf{a},\mf{b} \in \mc{S}$.  There exists $g \in G$ such that $g.\mf{a} \le \mf{b}$, and since $\mc{S}$ is an antichain, we have $g.\mf{a}= \mf{b}$.  Therefore, $G$ acts transitively on $\mc{S}$, so $\mc{S}$ is an antichain orbit.

Suppose that $\mc{S}$ is not an antichain; that is, there exist $\mf{a},\mf{b} \in \mc{S}$ with $\mf{a} < \mf{b}$.  Given $\mf{c},\mf{d} \in \mc{S}$, there exist $g_1,g_2 \in G$ such that $g_1.\mf{b} \le \mf{d}$ and $g_2.\mf{c} \le \mf{a}$.  We then have
\[
g_1g_2.\mf{c} \le g_1.\mf{a} < g_1.\mf{b} \le \mf{d},
\]
so there exists $g: = g_1g_2 \in G$ such that $g.\mf{c} < \mf{d}$.  Hence, $\mc{S}$ is a proper stacking class.
\end{proof}

We now show that the $G$-stacking relation determines the structure of $\{\mf{a}^G \mid \mf{a} \in \mf{B}^{\min}_H\}$ as a subset of $\mf{B}^{\min}_G$.

\begin{lem}\label{lem:chief:covering}
Let $G$ be a topological group, let $H$ be a closed normal subgroup of $G$ with $\mf{a} \in \mf{B}_H^{\min}$, and let $M/N$ be the lowermost representative of $\mf{a}^G$. For every $\mf{b} \in \mf{B}_H^{\min}$ covered by $M$ and every closed $L\normal H$ such that $L \nleq N$, there exists $g\in G$ such that $L$ covers $g.\mf{b}$. In particular, $M/N$ has no non-trivial abelian $H$-invariant subgroups.
\end{lem}

\begin{proof}
The existence of $M/N$ is given by Proposition~\ref{prop:induced_block}. Let $\mf{A}$ be the set of $\mf{b} \in \mf{B}_H$ such that $M$ covers $\mf{b}$ and take $\mf{b} \in \mf{A}$.  Since $M$ covers $\mf{b}$, we have $\CC_M(\mf{b}) < M$.  The group $\bigcap_{g \in G}\CC_M(g.\mf{b})$ is thus a proper $G$-invariant closed subgroup of $M$ and thereby is contained in $N$, since $N$ is the unique largest proper closed $G$-invariant subgroup of $M$.  Letting $L \normal H$ such that $L \nleq N$, there exists $g \in G$ such that $L \not\le \CC_M(g.\mf{b})$, hence $L$ covers $g.\mf{b}$.  

Considering the case $\mf{b} = \mf{a}$, we see that $N$ avoids $g.\mf{a}$ via Proposition~\ref{prop:induced_block}, so $\ol{LN}/N$ covers $g.\mf{a}$, ensuring that $\ol{LN}/N$ is non-abelian.  In particular, $M/N$ has no non-trivial abelian $H$-invariant subgroups.
\end{proof}

\begin{thm}\label{thm:induced_chief}
Let $G$ be a topological group with $H$ a closed normal subgroup of $G$ and let $\mf{a}, \mf{b} \in \mf{B}^{\min}_H$.  Then $\mf{a}^G \le \mf{b}^G$ in $\mf{B}^{\min}_G$ if and only if $\mf{a} \preceq_G \mf{b}$.  In particular,
\[
\{\mf{a}^G \mid \mf{a} \in \mf{B}^{\min}_H\} \simeq \mf{B}^{\min}_H/\sim_G
\]
as partially ordered sets.
\end{thm}

\begin{proof}
Fix $\mf{a},\mf{b} \in \mf{B}^{\min}_H$.  Suppose first that $\mf{a}^G \nleq \mf{b}^G$. By Proposition~\ref{prop:block:ordering}, there is a closed normal subgroup $K$ of $G$ that covers $\mf{b}^G$ but does not cover $\mf{a}^G$.  We deduce that $K \cap H$ covers $\mf{b}$, but it does not cover $\mf{a}$. Since $K \cap H$ is normal in $G$, the subgroup $K \cap H$ avoids $g.\mf{a}$ for all $g \in G$. We conclude there does not exist $g \in G$ such that $g.\mf{a} \le \mf{b}$, and thus, $\mf{a} \not\preceq_G \mf{b}$.

Now suppose that $\mf{a}^G \le \mf{b}^G$ and let $M/N$ be the lowermost representative of $\mf{b}^G$. The subgroup $M$ covers $\mf{a}^G$ by Proposition~\ref{prop:block:ordering}, so $M$ covers $\mf{a}$. In view of Proposition~\ref{prop:induced_block}, $H_{\mf{b}}\normal H$ is contained in $M$ and not contained in $N$. Lemma~\ref{lem:chief:covering} thus implies there exists $g \in G$ such that $H_{\mf{b}}$ covers $g.\mf{a}$.  We conclude via Corollary~\ref{cor:ordering:min_cover} that $g.\mf{a} \le \mf{b}$, so $\mf{a} \preceq_G \mf{b}$.

The conclusion about the structure of $\{\mf{a}^G \mid \mf{a} \in \mf{B}^{\min}_H\}$ as a partially ordered set is now clear.
\end{proof}

Every chief block $\mf{b} = \mf{a}^G$ of $G$ that is the extension of a minimally covered chief block $\mf{a}$ of $H$ is formed either from an antichain orbit or from a proper stacking class.  We can distinguish the two cases by considering the structure of the lowermost representative of $\mf{b}$.

\begin{prop}\label{prop:induced_chief:minimal}
Let $G$ be a topological group with $H$ a closed normal subgroup of $G$ and let $\mf{a} \in \mf{B}^{\min}_H$ with $M/N$ the lowermost representative of $\mf{a}^G$.  Then 
the following are equivalent:
\begin{enumerate}[(1) ]
\item The $G$-stacking class of $\mf{a}$ is an antichain orbit;
\item $M/N$ has a minimal closed $H$-invariant subgroup;
\item The set $\mc{M}$ of minimal closed $H$-invariant subgroups of $M/N$ is of the form $\mc{M} = \{(gKg\inv)/N \mid g \in G\}$ where $K/N$ is a representative of $\mf{a}$, and $\mc{M}$ is a quasi-direct factorization of $M/N$.
\end{enumerate}
\end{prop}

\begin{proof}
Suppose that $(2)$ holds and let $K/N$ be a minimal closed $H$-invariant subgroup of $M/N$. Since $M/N$ is a chief factor of $G$, we see that $M/N$ is generated topologically by the set $\mc{S} := \{(gKg\inv)/N \mid g \in G\}$. The minimality of $K/N$ ensures any two distinct $G$-conjugates of $K/N$ have trivial intersection and hence commute; additionally, since $M/N$ is a non-abelian chief factor of $G$, we have $\Z(M/N) = \triv$.  Proposition~\ref{prop:quasiproduct:centralizers} thus implies that $\mc{S}$ is a quasi-direct factorization of $M/N$. It is now easy to see that $\mc{S}=\mc{M}$.
%
%


The factor $K/N$ is a representative of some chief block $\mf{c}$ of $H$. In view of Lemma~\ref{lem:chief:covering}, there is some $gKg^{-1}/N$ that covers $\mf{a}$. Theorem~\ref{thm:unique_associate} ensures we may insert $H$-invariant subgroups $N\leq B<A\leq gKg^{-1}$ so that $A/B$ is associated to $\mf{a}$. Since $gKg^{-1}/N$ is a chief factor of $H$, we deduce that $A/B=gKg^{-1}/N$, hence $g.\mf{c}=\mf{a}$. Claim $(3)$ is now established. For claim $(1)$, taking any $\mf{a}'\in \mf{B}_H^{\min}$ so that $\mf{a}^G$ is the extension of $\mf{a}'$, it follows that $\mf{a}'\in\{g.\mf{a} \mid g \in G\}$.  The set $\{g.\mf{a} \mid g \in G\}$ is thus the class of minimally covered blocks of $H$ that are extendable to $\mf{a}^G$, and Theorem~\ref{thm:induced_chief} implies this set is exactly the $G$-stacking class of $\mf{a}$. We thus deduce that $(2)$ implies both $(1)$ and $(3)$.

Conversely, suppose that $(2)$ does not hold and set $K := \ol{H_{\mf{a}}N}$.  The factor $K/N$ is non-trivial since $M/N$ covers $\mf{a}$, and it is not minimal as a closed $H$-invariant subgroup of $M/N$. There thus exists $N < R < K$ with $R$ closed and normal in $H$.  Appealing to Lemma~\ref{lem:chief:covering}, $R/N$ covers $g.\mf{a}$ for some $g\in G$, and the minimality of $H_{\mf{a}}$ ensures that $R/N$ does not cover $\mf{a}$. We deduce that $g.\mf{a} <\mf{a}$, proving that the $G$-stacking class of $\mf{a}$ is not an antichain. Hence,  $(1)$ implies $(2)$.

That $(3)$ implies $(2)$ is immediate, so the proof is complete.
\end{proof}

In locally compact Polish groups, there is a large collection of chief blocks which are both minimally covered and extendable from any \emph{open} subgroup; this is shown in \cite{RW_LC}. This result, however, depends on the structure of locally compact groups.  

\subsection{Chief block structure of chief factors}\label{ssec:chief_block_structure} A chief factor can be topologically simple, or more generally, it can be a quasi-direct product of copies of a simple group. We now argue there are only two further possibilities for the structure of a chief factor.

\begin{defn}\label{def:A-simple_types} 
Let $G$ be a topologically characteristically simple group. 
\begin{enumerate}[(1) ]
\item The group $G$ is of \defbold{weak type}\index{weak type} if $\mf{B}_G^{\min}=\emptyset$. 
\item The group $G$ is of \defbold{stacking type}\index{stacking type} if $\mf{B}_G^{\min}\neq \emptyset$ and for all $\mf{a}, \mf{b} \in \mf{B}_G^{\min}$, there exists $\psi \in \Aut(G)$ such that $\psi.\mf{a} < \mf{b}$.
\end{enumerate}
\end{defn}


\begin{thm}\label{thm:chief:block_structure}
Suppose that $G$ is an $A$-simple topological group for some $A \le \Aut(G)$. Then exactly one of the following holds:
\begin{enumerate}[(1) ]
\item The group $G$ is of weak type.
\item The group $G$ is of semisimple type, and $A$ acts transitively on $\mf{B}_G$.
\item The group $G$ is of stacking type, and for all $\mf{a}, \mf{b} \in \mf{B}_G^{\min}$, there exists $\psi \in A$ such that $\psi.\mf{a} < \mf{b}$. 
\end{enumerate}
\end{thm}

\begin{proof}
The case that $\mf{B}^{\min}_G$ is empty is covered by $(1)$, so we assume that $\mf{B}^{\min}_G$ is non-empty. Giving $A$ the discrete topology, the group $G\rtimes A$ is a topological group under the product topology. Moreover, since $G$ is $A$-simple, the chief factor of $G \rtimes A$ covering $\mf{a}^G$ for any $\mf{a} \in \mf{B}^{\min}_G$ is $G/\triv$.  We deduce that $\mf{B}^{\min}_G$ forms a single $G \rtimes A$-stacking class via Theorem~\ref{thm:induced_chief}.

Lemma~\ref{lem:stackingclass_types} ensures that either $\mf{B}^{\min}_G$ is an antichain orbit or it is a proper stacking class.  If it is an antichain orbit, then $G$ has a minimal closed normal subgroup by Proposition~\ref{prop:induced_chief:minimal}, and hence $G$ is of semisimple type by Proposition~\ref{prop:semisimple:charsimple}.  Proposition~\ref{prop:semisimple_type:blocks} implies $\mf{B}_G = \mf{B}^{\min}_G$, and $(2)$ now follows. If instead $\mf{B}^{\min}_G$ is a proper stacking class, then $(3)$ follows.

The three possibilities are mutually exclusive: cases $(2)$ and $(3)$ are mutually exclusive by considering the ordering on $\mf{B}^{\min}_G$, and both cases imply that $\mf{B}^{\min}_G$ is non-empty.
\end{proof}

The following is an immediate consequence of Theorem~\ref{thm:chief:block_structure} and the relevant definitions.

\begin{cor}\label{cor:types_factors}If $G$ is a topological group and $K/L$ is a chief factor of $G$, then $K/L$ is either of weak type, semisimple type, or stacking type.
\end{cor}

We also note that for stacking to occur, the automorphism group must contain elements of infinite order.

\begin{cor} 
If $G$ is an $A$-simple topological group with $A$ a torsion group of continuous automorphisms of $G$, then $G$ is either of weak type or of semisimple type.
\end{cor}
\begin{proof}
Suppose that $G$ is not of weak type and suppose for contradiction that $G$ is of stacking type. Taking $\mf{a}\in \mf{B}_G^{\min}$, we may find $\psi\in A$ so that $\psi.\mf{a}<\mf{a}$. It follows that $\psi^n.\mf{a}<\mf{a}$ for all $n\ge 1$, and since $\psi$ has finite order, this is absurd. The group $G$ is thus of semisimple type.
\end{proof}

In Polish groups, the three types are also association invariants.

\begin{prop}\label{prop:type_invariant}
Given a Polish group $G$ and $\mf{a} \in \mf{B}_G$, every representative of $\mf{a}$ is of the same type: weak, semisimple, or stacking.
\end{prop}

\begin{proof}
Let $M/C$ be the uppermost representative of $\mf{a}$ and let $K_1/L_1$ and $K_2/L_2$ be two representatives of $\mf{a}$. Proposition~\ref{prop:upper-rep} supplies normal compressions $K_1/L_1\rightarrow M/C$ and $K_2/L_2\rightarrow M/C$. Theorem~\ref{thm:compression:chief} gives block space isomorphisms $\phi_i:(X_{M/C},\leq)\rightarrow (X_{K_i/L_i},\leq)$ with inverse morphism $\chi_i$ for $i\in \{1,2\}$.  The composition of induced maps $\wt{\phi_2}\circ\wt{\chi_1}:\mf{B}_{K_1/L_1}\rightarrow \mf{B}_{K_2/L_2}$ is an isomorphism of partially ordered sets.  Applying Proposition~\ref{prop:compression_minimally_covered} twice, we see that this map restricts to an isomorphism of partially ordered sets $\wt{\phi_2}\circ\wt{\chi_1}:\mf{B}^{\min}_{K_1/L_1}\rightarrow \mf{B}^{\min}_{K_2/L_2}$, where $\mf{B}^{\min}_{K_1/L_1}$ and $\mf{B}^{\min}_{K_2/L_2}$ have the induced partial order.

The case of weak type is now immediate:
\[
K_1/L_1 \text{ is of weak type } \Leftrightarrow \mf{B}^{\min}_{K_1/L_1} = \emptyset \Leftrightarrow \mf{B}^{\min}_{K_2/L_2} = \emptyset \Leftrightarrow K_2/L_2 \text{ is of weak type}.
\]
If $K_1/L_1$ is of semisimple type, then Proposition~\ref{prop:semisimple_type:blocks} ensures that $\mf{B}^{\min}_{K_1/L_1}$ is an antichain, hence $\mf{B}^{\min}_{K_2/L_2}$ is also an antichain. Since $K_2/L_2$ is either of semisimple or stacking type, the factor $K_2/L_2$ must be of semisimple type.  Reversing the roles of $K_1/L_1$ and $K_2/L_2$ gives the converse implication. We thus deduce that $K_1/L_1$ is of semisimple type if and only if $K_2/L_2$ is of semisimple type. The proposition now follows.
\end{proof}

With Proposition~\ref{prop:type_invariant} in hand, the following definition is sensible: 
\begin{defn} For $G$ a Polish group, a chief block $\mf{a}\in \mf{B}_G$ is of \textbf{weak/\allowbreak semisimple/\allowbreak stacking type} if some (equivalently, all) representative(s) are of weak/\allowbreak semisimple/\allowbreak stacking type.
\end{defn}

Given a minimal closed normal subgroup $M$ of a Polish group $G$, one might hope at this point for a more detailed description of how $M$ is built out of chief factors, or at least of the poset $\mf{B}^{\min}_M$ in the case of stacking type.  See \S\ref{ex:stacking} for an indication of why such a description is likely to be difficult to obtain for general Polish groups.

To conclude the section, let us briefly observe how the type of a chief block can change when it is extended from a normal subgroup.

\begin{prop}
Let $G$ be a Polish group with $H$ a closed normal subgroup of $G$ and $\mf{a} \in \mf{B}^{\min}_H$.  If $\mf{a}^G$ is of semisimple type, then $\mf{a}$ is of the same type.
\end{prop}

\begin{proof}
Let $M/N$ be the lowermost representative of $\mf{a}^G$ and set $K := H_\mf{a}$.  That $K$ is the least normal subgroup covering $\mf{a}$ ensures that $K$ is topologically perfect. Additionally, we have a normal compression $K/K\cap N\rightarrow\ol{KN}/N$.

By Theorem~\ref{thm:norm_sgrps}, there is a subgroup $L/N$ of $\ol{KN}/N$ that is topologically generated by components of $M/N$ such that $\ol{KN}/L$ is abelian. As $\ol{KN}/N$ is topologically perfect, $\ol{KN}=L$, and thus, $\ol{KN}/N$ is of semisimple type. Proposition~\ref{prop:semisimple:compressions} then ensures $K/K\cap N$ is of semisimple type. Since semisimplicity is preserved under quotients, $K/C_K(\mf{a})$ is also of semisimple type, hence $\mf{a}$ is of semsimple type.
\end{proof}

\section{Examples}\label{sec:ex}
\addtocontents{toc}{\protect\setcounter{tocdepth}{1}}
\subsection{The association relation}\label{ex:kleinfour} Take $G$ to be the Klein four-group, $C_2 \times C_2$.  The group $G$ has five subgroups: the trivial group, $G$ itself, and three subgroups $A_1,A_2,A_3$ of order $2$.  There are three distinct chief series for $G$, namely the series of the form
\[
1 < A_i < G
\]
for $1 \le i \le 3$.  This gives rise to a total of six distinct chief factors of $G$: three `lower' chief factors $A_i/1$, and three `upper' chief factors $G/A_i$.

\begin{claim}
Let $\{K/L,M/N\}$ be a pair of distinct chief factors of $G$.  Then the following are equivalent:
\begin{enumerate}[(1) ]
\item $K/L$ is associated to $M/N$;
\item There are distinct $i,j\in \{1,2,3\}$ so that
\[
\{K/L,M/N\} = \{A_i/\triv,G/A_j\}.
\]
\end{enumerate}
\end{claim}

\begin{proof}
Suppose that $\{K/L,M/N\} = \{A_i/\triv,G/A_j\}$ for $i,j$ distinct. We see that $A_iA_j = G\triv$, $A_i \cap \triv A_j = \triv$, and $A_j \cap \triv A_j = A_j$.  Therefore, $K/L$ and $M/N$ are associated, verifying $(1)$.

Conversely, suppose that $K/L$ and $M/N$ are associated. If $L = N$, then $K = KN = ML = M$, contradicting the assumption that $K/L$ and $M/N$ are distinct; hence, $L \neq N$.  We additionally have $K \cap (LN) < K$, so $LN \neq G$.  One of $L$ and $N$ must therefore be trivial; without loss of generality, $L= \triv$.

The subgroup $N$ is not trivial and not equal $G$, so $N = A_j$ for some $j$.  Since $M > N$, we deduce that $M = G$. On the other hand, the factor $K/L$ is chief, hence we have $|K|=2$ and $K = A_i$ for some $i$.  The assumption that $KN = ML$ ensures $i$ and $j$ are distinct. We now conclude $(2)$.
\end{proof}

The graph of associations between chief factors of $G$ is a hexagon, so association of chief factors is not a transitive relation. Indeed, it is possible to have a sequence $F_0,F_1,F_2,F_3$ of distinct chief factors such that $F_i$ is associated to $F_{i+1}$ for $0 \le i \le 2$, and $F_3$ lies in the same chief series as $F_0$ as follows:

\[
F_0 := A_1/\triv, \; F_1 := G/A_2, \; F_2 := A_3/\triv, \; F_3 := G/A_1.
\]
This is the shortest possible sequence of this nature, as shown by Lemma~\ref{chief:associated:ordered}.

\subsection{Non-Archimedean Polish groups as normal factors of a chief factor}\label{ex:stacking}

For the purposes of this subsection, $0 \in \Nb$.

Let $\Tw$ be the regular tree in which every vertex has $\aleph_0$ neighbors; we consider $\Tw$ to be a metric space under the usual path metric. Write $V\Tw$ for the set of vertices of $\Tw$.  Choose an infinite geodesic ray $\xi:=(v_0,v_1,\dots)$ in $\Tw$; \textit{i.e.} a sequence of distinct vertices $v_i$ such that $d(v_i,v_{i+1})=1$ for all $i \in \Nb$.  The \defbold{end} $[\xi]$ defined by $\xi$ is the equivalence class of $\xi$ among infinite geodesic rays, where the rays $\xi = (v_0,v_1,\dots)$ and $\xi' = (v'_0,v'_1,\dots)$ are equivalent if there exist $i,j \in \Nb$ such that $v_{i+n} = v'_{j+n}$ for all $n \in \Nb$.  Note that for each vertex $w$ of $\Tw$, there is exactly one infinite geodesic ray in $[\xi]$ with initial vertex $w$.  Automorphisms of the tree then act on the infinite geodesic rays, and consequently on ends, via $g.(v_0,v_1,\dots) = (g(v_0),g(v_1),\dots)$.

The end $[\xi]$ defines an orientation of the tree: given an edge $e = (u,w)$ there is a geodesic ray $(w_0,w_1,w_2,\dots) \in [\xi]$ such that $\{u,w\} = \{w_0,w_1\}$, but either all such rays have $u=w_0$ and $w=w_1$, in which case we say $e$ points towards $[\xi]$, or else all such rays have $w=w_0$ and $u=w_1$, in which case we say $e$ points away from $[\xi]$.  Conversely, distinct ends of $\Tw$ have distinct sets of edges pointing towards them.

Given the end $[\xi]$ we can also define a \defbold{Busemann function of $[\xi]$}, which is a function $f$ from $V\Tw$ to $\Zb$ with the following property: if $(u,w)$ points towards $[\xi]$, then $f(w) - f(u) = 1$.  Write $B_{[\xi]}$ for the set of Busemann functions of $[\xi]$.  One sees that if two Busemann functions $f,f'$ are such that $f(u)=f'(u)$ for some vertex $u$, then $f=f'$; write $f_0$ for the Busemann function such that $f_0(v_0) = 0$.  Conversely if $f$ is a Busemann function, then so is $w \mapsto f(w) - n$ for all $n \in \Zb$; so in fact $B_{[\xi]} = \{f_n \mid n \in \Zb\}$ where $f_n: w \mapsto f_0(w) - n$.

For $n \in \Zb$, define the horoball $X_n:= \{w \in V\Tw \mid f_0(w) \ge n\}$ and the horosphere $Y_n = X_n \setminus X_{n+1}$; see Figure~$1$.
  
\begin{figure}[h]\label{fig:X_n}
    \centering
    \includegraphics[width=0.8\textwidth]{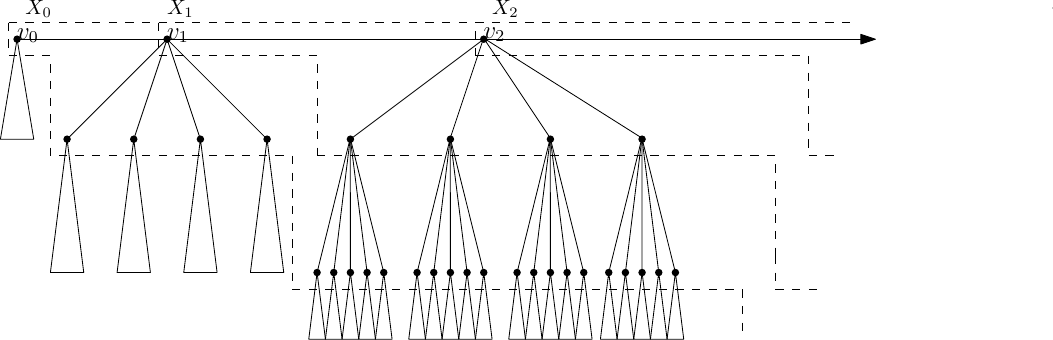}
    \caption{The horoballs $X_n$}
\end{figure}

Set
\[
G := \{g \in \Aut(\Tw) \mid [g.\xi] = [\xi]\}.
\]
We see that $G$ is exactly the stabilizer of the binary relation on $V\Tw$ consisting of the edges $(u,w)$ pointing towards $\xi$.  As such, $G$ is a closed subgroup of $\Sym(V\Tw)$; it is therefore Polish by Lemma~\ref{lem:increasing_Polish}.

We now define a family of subgroups of $G$, all of which will have a non-abelian chief factor.

Choose a function $c: E \rightarrow \Nb$, where $E$ is the set of edges pointing away from $\xi$, with the following properties:
\begin{enumerate}[(1) ]
\item we have $c((v_{i+1},v_i)) = 0$ for all $i \in \Nb$;
\item for all $w \in V\Tw$, and writing $E_w$ for the edges in $E$ with origin $w$, then $c$ restricts to a bijection $c_w$ from $E_w$ to $\Nb$.
\end{enumerate}
Given $g \in G$, there is then an associated permutation of $\Nb$ given by the composition $c_{g(w)} \circ g \circ c\inv_w$.

Let $L$ be a closed transitive subgroup of $\Sym(\Nb)$.  We define $G(L)$ to be the set of $g \in G$ such that $c_{g(w)} \circ g \circ c\inv_w \in L$ for all $w \in V\Tw$.  It is straightforward to check that $G(L)$ is a subgroup of $G$.  Since $L$ is closed in $\Sym(\Nb)$, it follows that $G(L)$ can be prescribed by its orbits on the set of finite tuples of vertices, and hence $G(L)$ is closed in $G$.

There is a natural action of $G$ on $B_{[\xi]}$ given by $g.f(w) = f(g\inv(w))$.  Let $P$ be the kernel of this action and let $P(L) = G(L) \cap P$.  Observe that if $g \in G$ fixes a vertex of $\Tw$, then $g \in P$; that is, $P$ contains all vertex stabilizers of $G$.  In particular, $P$ has nonempty interior and hence is an open subgroup of $G$, so $P(L)$ is an open subgroup of $G(L)$.  The fact that $L$ is transitive ensures that $P(L)$ acts transitively on $Y_n$ for all $n \in \Zb$.

Let $t$ be the automorphism of $\Tw$ such that $c_{t(w)} \circ t \circ c\inv_w = 1$ for all $w \in V\Tw$ and $t(v_i) = v_{i+1}$ for all $i \in \Nb$; here we rely on property (1) of $c$ to ensure the existence of such an automorphism.  Then $t \in G(L)$ for any $L \le \Sym(\Nb)$.  There is an injective homomorphism $\delta: G/P \rightarrow \Zb$ given by the equation $g.f_0 = f_{\delta(g)}$.  In particular, $t.f_0(v_0) = -1$, so $\delta(t) = 1$.  Thus $\delta$ is an isomorphism and we have $G = P \rtimes \grp{t}$, and similarly $G(L) = P(L) \rtimes \grp{t}$.

Our next aim is to describe the closed normal subgroups of $P(L)$.  Unsurprisingly, these depend on the closed normal subgroups of $L$, but some general features can be established.

Define $P_n(L)$ to be the subgroup of $P(L)$ fixing all vertices in $X_n$.  Then for all $n \in \Zb$, it is easy to see that: $P_n(L)$ is a closed normal subgroup of $P(L)$;  $P_n(L) < P_{n+1}(L)$; and $tP_n(L)t^{-1} = P_{n+1}(L)$.  Thus we have a chain
\[
\dots {} < P_{-2}(L) < P_{-1}(L) < P_0(L) < P_1(L) < P_2(L) < {} \dots {}
\]
of closed normal subgroups of $P$ that is shifted by the action of $\grp{t}$.

Consider an element $g \in P(L)$.  We know that $g.\xi \in [\xi]$, so there is some $i,j \in \Nb$ such that $g(v_{i+n}) = v_{j+n}$ for all $n \ge 0$; the $\grp{g}$-invariance of the Busemann functions ensures $i=j$.  In other words, $g.v_n = v_n$ for all $n \ge i$.  One then sees that for $n \ge i$, there is a unique element $g_n \in P_n(L)$ that has the same action as $g$ on the subtree $T_n$ consisting of those vertices $w$ such that $d(w,v_n) < d(w,v_{n+1})$, and fixes all vertices outside of $T_n$.  Since $\bigcup_{n \ge i}T_i = \Tw$, it follows that $g_n \rightarrow g$ as $n \rightarrow +\infty$.  Thus
\[
P(L) = \ol{\bigcup_{n \in \Zb}P_n(L)}.
\]
On the other hand, since $\bigcup_{n \in \Zb} X_n = V\Tw$, we see that
\[
\triv = \bigcap_{n \in \Zb}P_n(L).
\]

For all $n \in \Zb$ there is also a subgroup $Q_n(L)$ of $P(L)$ that consists of those $g \in P(L)$ such that $c_{g(w)} \circ g \circ c\inv_w = 1$ for all $w \in V\Tw \setminus X_{n+1}$.  It is not hard to see that $Q_n(L)$ is closed and that for all $g \in P(L)$, there is a unique $g' \in Q_n(L)$ such that $g'.w = g.w$ for all $w \in X_n$; thus $P(L)$ is a semidirect product $P_n(L) \rtimes Q_n(L)$.

The stabilizer in $Q_n(L)$ of $w \in Y_n$ fixes pointwise the subgraph $\pi\inv_n(w)$, where $\pi\inv_n$ is the closest point projection from $\Tw$ to $X_n$.  As a consequence, the quotient $P(L)/P_{n-1}(L)$ of $P(L)$ takes the form of an unrestricted wreath product $L \; \mathrm{Wr}_{Y_n} \; Q_n(L)$, where the base group of the wreath product, $\prod_{Y_n}L$, is the normal factor $P_n(L)/P_{n-1}(L)$.

We will not attempt a complete classification of the closed $P(L)$-invariant subgroups of $P_n(L)/P_{n-1}(L)$, as it is complicated to describe the structure of the abelian $P(L)$-invariant factors.  There are however two easy observations to be made here.
\begin{enumerate}[(1) ]
\item For any closed normal subgroup $M$ of $L$, we have a corresponding closed $P(L)$-invariant subgroup $\prod_{Y_n}M$ of $P_n(L)/P_{n-1}(L)$.  In particular, for every closed normal factor $M/N$ of $L$, there is therefore a closed normal factor $(\prod_{Y_n}M)/(\prod_{Y_n}N)$ of $P(L)$ that is isomorphic to a direct product of $\aleph_0$ copies of $M/N$.  So depending on $L$, the group $P(L)$ can have many different isomorphism types of closed normal factors.
\item Let $R$ be a closed normal subgroup of $P(L)$ such that $P_{n-1}(L) \le R \le P_n(L)$; consider $R/P_{n-1}(L)$ as a subgroup of $\prod_{Y_n}L$, and let $M$ be the closure of the projection of $R$ onto one of the copies of $L$.  Then $M$ is a closed normal subgroup of $L$, and a standard commutator argument shows that
\[
\prod_{Y_n}\ol{[M,M]} \le R/P_{n-1}(L) \le \prod_{Y_n} M.
\]
So (1) accounts for all closed normal subgroups of $P(L)$ between $P_{n-1}(L)$ and $P(L)$ up to an abelian factor.  In particular: if $L$ is non-abelian and topologically simple (for example, $L = \Sym(\Nb)$), then $P_n(L)/P_{n-1}(L)$ is a chief factor of $P(L)$.
\end{enumerate}

Another way to see a wreath product in the structure of $P(L)$ is as follows.  The $P$-translates of $VT_0$ form a partition of $V\Tw \setminus X_1$, and that any choice of automorphisms of the rooted trees $\{gT_0 \mid g \in P\}$ fixing the root, one for each distinct subgraph of $\Tw$ in the set, can be combined to specify a unique element of $P_0$.  Thus $P_0$ is the direct product of the $P$-conjugates of $H(L)$, where $H(L)$ is the closed subgroup consisting of all elements of $P$ that fix every vertex in $V\Tw \setminus VT_0$ (that is, $H$ fixes all vertices $w$ such that $d(w,v_0) > d(w,v_1)$).  We can then write $P(L)$ as
\[
P(L) = H(L) \; \mathrm{Wr}_{Y_0} \; Q_0(L)
\]
where the base group is $P_0(L)$, or indeed
\[
P(L) = H_n(L)  \; \mathrm{Wr}_{Y_n} \; Q_n(L)
\]
for any $n \in \Zb$, where $H_n(L) = t^nH(L)t^{-n}$ and the base group is now $P_n(L)$.

We can now say something about closed normal subgroups of $P(L)$ in general, not just those between $P_{n-1}(L)$ and $P_n(L)$.  First we appeal to a general fact about wreath products.

\begin{lem}\label{lem:wreath_commutator}
Let $A$ and $B$ be topological groups such that $B$ is equipped with a transitive action on an infinite set $X$.  Form the unrestricted wreath product
\[
C := A \: \mathrm{Wr}_X \: B = \prod_{X} A \rtimes B,
\]
where the base group $\prod_{X} A$ has the product topology.  Then $\prod_{X} A \le \ol{[C,C]}$.
\end{lem}

\begin{proof}
(Our thanks to Yves Cornulier \cite{YCor} for this argument.)  

It is easy to see that $\prod_X \ol{[A,A]} \le \ol{[C,C]}$, so we may assume that $A$ is abelian.

Consider the elements of $\prod_X A$ of the form $\delta_x(a)\delta_y(a)\inv$ for $x,y \in X$ and $a \in A$, where $\delta_x(a)$ is the function from $X$ to $A$ taking the value $a$ at $x$ and $1$ elsewhere.  We see that every such element is a commutator, namely
\[
\delta_x(a)\delta_y(a)\inv = \delta_{b.y}(a)\delta_y(a)\inv = b\delta_y(a)b\inv\delta_y(a)\inv,
\]
where $b \in B$ is such that $b.y = x$.  Thus the subgroup $M$ generated by the elements of the form $\delta_x(a)\delta_y(a)\inv$ is contained in $[C,C]$.  Moreover, we see that $M$ consists exactly of the elements $(a_x)_{x \in X}$ of the direct sum $\bigoplus_X A$ such that $\prod_{x \in X} a_x = 1$.  Since $X$ is infinite, we see that for all $(a_x)_{x \in X}$ and any finite subset $Y$ of $X$, there is $(a'_x)_{x \in X} \in M$ such that $a'_x = a_x$ for all $x \in Y$.  Thus $M$ is dense in $\prod_X A$, so $\prod_X A \le \ol{[C,C]}$.
\end{proof}

\begin{claim}\label{claim:normal}
Let $N$ be a proper non-trivial closed normal subgroup of $P(L)$.  Then there is a maximal $n \in \Zb$ such that $N \nleq P_n(L)$.  Moreover, for this $n$ we have
\[
P_{n-1}(L) \le \ol{[P_n(L),P_n(L)]} \le N \le P_{n+1}(L).
\]
\end{claim}

\begin{proof}
There must be some $n \in \Zb$ such that $N \nleq P_{n}(L)$, since $\bigcap_{n \in \Zb}P_n(L) = \triv$.  We claim that if $N \nleq P_{n}(L)$, then $N \ge P_{n-1}(L)$.  By the symmetry in the argument, it now suffices to consider the case $n=0$, that is, $N \nleq P_0(L)$.  In particular, there is some vertex of $X_0$ that is not fixed by $N$; given the uniqueness of rays in $[\xi]$ starting at a given vertex, it follows that there is a vertex in $Y_0$ that is not fixed by $N$.  Since $P$ acts transitively on $Y_0$ and $N$ is normal in $P$, in fact there must be $a \in N$ such that $av_0 \neq v_0$, and hence $aT_0 \cap T_0 = \emptyset$.  In particular, taking $g,h \in H(L)$, then $ah\inv a\inv$ centralizes $H(L)$ and hence $[g,[a,h]] = [g,h]$.  In particular, $[g,h] \in N$ for all $g,h \in H(L)$; since $N$ is closed, we have $\ol{[H_0(L),H_0(L)]} \le N$.  Consider next the structure of $H_0(L)$: it is itself a wreath product
\[
H_0(L) = H_{-1}(L) \; \mathrm{Wr}_{\Nb} \; L.
\]
By Lemma~\ref{lem:wreath_commutator}, it follows that $H_{-1}(L) \le \ol{[H_0(L),H_0(L)]}$.  Notice that $P_m(L)$ is the direct product of the $P(L)$-conjugates of $H_m(L)$ for all $m \in \Zb$; thus
\[
P_{-1}(L) \le \ol{[P_0(L),P_0(L)]} \le N.
\]

To finish the argument: we have shown that if $N \nleq P_n(L)$, then $N \ge P_{n-1}(L)$.  In particular, if $N \nleq P_n(L)$ for all $n \in \Zb$, then $N \ge P_{n-1}(L)$ for all $n \in \Zb$; since the subgroups $P_n(L)$ have dense union in $P(L)$, this would contradict the hypothesis that $N$ is a proper closed subgroup of $P(L)$.  Thus there is a maximal $n \in \Zb$ such that $N \nleq P_n(L)$, and for this $n$ we have
\[
P_{n-1}(L) \le \ol{[P_n(L),P_n(L)]} \le N \le P_{n+1}(L)
\]
as claimed.
\end{proof}

In particular, Lemma~\ref{lem:wreath_commutator} ensures that if $L$ is topologically perfect, then so are $H_n(L)$ and $P_n(L)$.  In the case that $L$ is non-abelian and topologically simple, we have already seen that $P_n(L)/P_{n-1}(L)$ is a chief factor of $P(L)$.  We thus have the following corollary.

\begin{cor}\label{cor:normal}
Suppose that $L$ is topologically perfect.  Then for all $n \in \Zb$, every closed normal subgroup $N$ of $P(L)$ is comparable to $P_n(L)$ in the inclusion order, that is,
\[
N \nleq P_n(L) \Rightarrow N \ge P_n(L).
\]
If in fact $L$ is topologically simple, then the closed normal subgroups of $P(L)$ are totally ordered as follows:
\[
\triv < \dots < P_{-2}(L) < P_{-1}(L) < P_0(L) < P_1(L) < P_2(L) < \dots < P(L).
\]
\end{cor}

With such a description of the closed normal subgroups of $P(L)$, it is now easy to see that $P(L)$ is a chief factor of $G(L)$.  We also obtain a sufficient condition for $P(L)$ to be of stacking type.

\begin{claim}\
\begin{enumerate}[(1) ]
\item The group $P(L)$ is the minimal non-trivial closed normal subgroup of $G(L)$.  In particular, $P(L)/\triv$ is the lowermost representative of a minimally covered chief block of $G(L)$.
\item Suppose that $L$ is topologically perfect and that $\mf{B}^{\min}_L \neq \emptyset$ (for example, $L = \Sym(\Nb)$).  Then $P(L)$ is a chief factor of $G(L)$ of stacking type.
\end{enumerate}
\end{claim}

\begin{proof}
Let $N$ be a non-trivial closed normal subgroup of $G(L)$; we aim to show $P(L) \le N$.  We see that if $N \nleq P(L)$, then $N$ does not centralize $P(L)$, and hence $N \cap P(L) \ge [N,P(L)] > \triv$, so $N \cap P(L)$ is a non-trivial closed normal subgroup of $G(L)$; thus we may assume $N \le P(L)$.  Then by Claim~\ref{claim:normal}, we have $P_n(L) \le N$ for some $n \in \Zb$.  By conjugating by $t$ we see that $P_n(L) \le N$ for all $n \in \Zb$.  Since $P(L) = \ol{\bigcup_{n \in \Zb}P_n(L)}$, we conclude that $N = P(L)$, proving (1).

For (2), suppose that $L$ is topologically perfect and that there exists a minimally covered chief block $\mf{a}$ of $L$; let $M/N$ be the lowermost representative of $\mf{a}$.  Then $P(L)$ has a corresponding closed normal factor $(\prod_{Y_1}M)/(\prod_{Y_1}N)$, regarded as a subgroup of $P_1(L)/P_0(L)$.  Let $M^*$ be the preimage of $\prod_{Y_1}M$ in $P(L)$ and let $N^*$ be the preimage of $\prod_{Y_1}N$ in $P(L)$; we claim that $M^*/N^*$ is the lowermost representative of a minimally covered chief block of $P(L)$.  To prove this, it is enough to show, given a closed $P(L)$-invariant subgroup $R$ of $M^*$ such that $R \nleq N^*$, then $R = M^*$.

We see that $R \nleq P_0(L)$; since $L$ is topologically perfect, Corollary~\ref{cor:normal} implies that $R \ge P_0(L)$, so we can consider $R/P_0(L)$ as a closed $P(L)$-invariant subgroup of $P_1(L)/P_0(L) \cong \prod_{Y_1}L$.  Since $R \nleq N^*$ but $R \le M^*$, the projection $R'$ of $R/P_0(L)$ onto one of the copies of $L$ satisfies $R' \nleq N$, while $\ol{R'}$ is a closed $L$-invariant subgroup of $M$.  Since $M/N$ is the lowermost representative of $\mf{a}$, the only possibility is that $\ol{R'} = M$, from which it follows, as observed earlier with closed normal subgroups of wreath products, that
\[
\prod_{Y_1} \ol{[M,M]} \le R/P_0(L).
\]
Now $M$ is topologically perfect (since $\ol{[M,M]}$ covers $M/N$, while $M$ is already minimal among closed normal subgroups of $L$ covering $M/N$), so in fact
\[
M^*/P_0(L) = \prod_{Y_1} M = \prod_{Y_1} \ol{[M,M]} \le R/P_0(L),
\]
and we conclude that $R = M^*$.  This argument proves that $[M^*/N^*] \in \mf{B}^{\min}_{P(L)}$.  We have already seen that $P(L)$ is a chief factor of $G(L)$, and the fact that $\mf{B}^{\min}_{P(L)} \neq \emptyset$ ensures that $P(L)$ is of stacking or semisimple type.  The closed normal subgroup structure of $P(L)$ makes it clear that $P(L)$ is not of semisimple type, so $P(L)$ is of stacking type.
\end{proof}

\begin{rmk}
We defined the function $c$ from $E$ to $\Nb$, and took $L$ to be a subgroup of $\Sym(\Nb)$, but in practice ``$\Nb$'' here is just a generic example of a countably infinite set with a distinguished element $0$; the same construction will work where $L$ is any closed transitive permutation group on a countably infinite set.  In particular, the construction of $G(L)$ can easily be iterated: $G(L)$ acts faithfully and transitively on $V\Tw$.  In turn, the topologically perfect Polish group $L' = G(L) \; \mathrm{Wr}_{\Nb} \; \Sym(\Nb)$ is a closed transitive permutation group on the countably infinite set $V\Tw \times \Nb$, and $L'$ has a minimal closed normal subgroup $\prod_{\Nb} P(L)$; in particular, $\prod_{\Nb} P(L)$ is a minimally covered chief factor of $L'$.  One can then build the group $G(L')$ with stacking type chief factor $P(L')$, and so on.

As for the group $L \le \Sym(\Nb)$ we start with, it is probably hopeless to try to describe the possibilities for $\mf{B}^{\min}_L$ as a poset, if all we know is that $L$ is transitive, closed and topologically perfect.  The construction of $G(L)$ shows that if $M/N$ is a chief factor of stacking type, there is no detailed structure theory of the poset $\mf{B}^{\min}_{M/N}$ in general.  Perhaps $\mf{B}^{\min}_{M/N}$ can be described as a ``stack'' of copies of $\mf{B}^{\min}_L$ for some other Polish group $L$, but that does not imply any sort of homogeneity of the poset $\mf{B}^{\min}_L$.
\end{rmk}

\begin{ack}
	We would like to thank Fran\c{c}ois Le Ma\^{i}tre for contributing to this work and for his many helpful remarks. Portions of this work were also developed during a stay of both authors at the \textit{Mathematisches Forschungsinstitut Oberwolfach}; we thank the MFO for its hospitality.  We also thank the anonymous referee for their diligent reading of the article and suggestions for improvement.
\end{ack}

\appendix

\section{Examples of normal compressions}

In this appendix, we provide various examples of normal compressions $\psi:G\rightarrow H$ for $H$ and $G$ non-locally compact Polish groups. In these examples, $\psi$ will be the inclusion map. 

\subsection{Basic examples}

The most basic example of a normal compression is probably $\Qb\injects \Rb$, where $\Qb$ is equipped with the discrete topology. Another example is provided by the inclusion $\LL^1(X,\mu,\Rb)\injects \LL^0(X,\mu,\Rb)$, where $\LL^0(X,\mu,\Rb)$ is the group of measurable functions on a standard probability space $(X,\mu)$ equipped with the topology of convergence in measure, and $\LL^1(X,\mu,\Rb)$ is the subgroup of integrable functions equipped with the $\LL^1$ topology. 
Let us also note the following non-example of normal compression: $\LL^\infty(X,\mu,\Rb)$ is not a Polishable subgroup of $\LL^0(X,\mu,\Rb)$  (see \cite[Lemma 9.3.3]{MR2455198}), but it is still dense and normal in $\LL^0(X,\mu,\Rb)$. 

Moving away from abelian examples, where the $H$-action by conjugacy on $G$ is trivial, we note an interesting topologically simple example: the Polish group $\mathfrak S_\infty$ of all permutations of the set $\Nb$ of integers. Its topology is the one induced by the product topology on $\Nb^{\Nb}$ seeing $\Nb$ as a discrete space.  The group $\mathfrak S_{\infty}$ contains the countable group $\mathfrak S_{(\infty)}$ of permutations with finite support as a dense normal subgroup, so the inclusion $\mathfrak S_{(\infty)}\injects \mathfrak S_\infty$ provides a first non-abelian example of a normal compression. We point out additionally that $\mathfrak S_{(\infty)}$ is \textit{not} simple, but the commutator subgroup is simple. This shows Theorem~\ref{thm:compression_simple} is sharp.

\subsection{Unitary groups}
For $\mathcal H$ a separable Hilbert space, the group $\mathcal U(\mathcal H)$ of unitaries of $\mathcal H$ is Polish for the strong topology, defined to be the weakest topology making the maps $u\in\mathcal U(\mathcal H)\mapsto u(\xi)$ continuous for each $\xi\in\mathcal H$. A bounded operator $x$ on $\mathcal H$ is called \textbf{compact} if $x(B)$ is compact for the norm on $\mathcal H$ where $B$ denotes the unit ball in $\mathcal H$. It is easy to see that the compact operators form a norm-closed ideal in $\mathcal B(\mathcal H)$, which is denoted by $\mathcal K(\mathcal H)$.(The ideal of compact operators is the norm closure of the ideal of operators with finite rank.)

Since the circle group $\mathbb S^1\cdot\mathrm{id}_{\mathcal H}$ is a compact subspace of $\mathcal B(\mathcal H)$, its sum with the ideal $\mathcal K(\mathcal H)$ is also norm closed. In particular, the unitary group
$$\mathcal U_K(\mathcal H):=\mathcal U(\mathcal H)\cap (\mathbb S^1\cdot\mathrm{id}_\mathcal H+\mathcal K(\mathcal H))$$
of compact perturbations of the circle group is closed in the complete metric space $(\mathcal U(\mathcal H),\norm\cdot)$. The operator norm thus induces a complete metric on $\mathcal U_K(\mathcal H)$. 

\begin{lem}$(\mathcal K(\mathcal H),\norm\cdot)$ is separable.
\end{lem}
\begin{proof}
Let $(\xi_i)_{i\in\Nb}$ be a dense subset of $\mathcal H$. Then it is an easy corollary of the singular value decomposition theorem for compact operators that the countable set of finite sums of elements of the form $\la\xi_i,\cdot\ra \xi_j$ is dense in $\mathcal K(\mathcal H)$. We deduce that $\mathcal{K}(\mathcal{H})$ is separable.
\end{proof}
\noindent The group $\mathcal U_K(\mathcal H)$ is thus Polish when equipped with the norm topology.

We now argue $\mathcal{U}_K(\mathcal H)$ is dense in $\mathcal{U}(\mathcal H)$. Let us see the finite dimensional unitary groups $\mathcal U(n)$ as subgroups of $\mathcal U(\mathcal H)$ via the embeddings 
\[
u\in\mathcal U(n)\mapsto\left(\begin{smallmatrix} u&0\\ 0& 1 \end{smallmatrix}\right).
\] 
Every $\mathcal U(n)$ is then a subgroup of $\mathcal U_K(\mathcal H)$, and since their union is dense in $\mathcal U(\mathcal H)$, the subgroup $\mathcal U_K(\mathcal H)$ is dense in $\mathcal U(\mathcal H)$. We conclude that $\mathcal U_K(\mathcal H)\rightarrow \mathcal U(\mathcal H)$ is a normal compression.\\ 

Here are two further examples; for details see \cite{Ando:2011lr}. One can consider the smaller ideal $\mathcal S_p(\mathcal H)$ of Schatten  operators of class $p$, which is Polish for the associated Schatten $p$-norm. The group $\mathcal U_p(\mathcal H)$ of unitaries $u$ such that $1-u\in \mathcal S_p(\mathcal H)$ is Polish for the natural $p$-norm and is dense in $\mathcal U(\mathcal H)$, since it contains $\mathcal U(n)$ for every $n\in\Nb$. Thus, $\mathcal U_p(\mathcal H)\rightarrow \mathcal U(\mathcal H)$ is a normal compression.

For the second example, let $(M,T)$ be a type II$_\infty$ separable factor. The ideal $\mathcal I:=\{x\in M: T(x^*x)<+\infty\}$ is Polish for the topology induced by the $2$-norm defined by $\norm{x}_2:=\sqrt{T(x^*x)}$. The unitary group $\mathcal U_{I}(M):=\mathcal U(M)\cap\{1-x: x\in \mathcal I\}$ is thus a Polish group, and its inclusion in $\mathcal U(M)$ is normal. That $\mathcal{U}_{I}(M)$ is also dense follows  from a result of M. Broise \cite{MR0223900} stating that symmetries generate $\mathcal U(M)$ together with the fact that every symmetry onto a projection of infinite trace can be approximated in the strong operator topology by symmetries onto projections of finite trace.

\subsection{Full groups}

Let $(X,\mu)$ be a standard probability space. A Borel bijection $T$ of $(X,\mu)$ is called \textbf{non-singular} if the pushforward measure $T_*\mu$ is equivalent to $\mu$; that is, if for all Borel $A\subseteq X$, we have $\mu(A)=0$ if and only if $\mu(T\inv(A))=0$. Denote by $\Aut^*(X,\mu)$ the group of non-singular Borel bijections of $(X,\mu)$ where two such bijections are identified if they coincide up to measure zero. 

The group $\Aut^*(X,\mu)$ can be endowed with a metrizable topology induced by the uniform metric $d_u$ given as follows: for all $S,T\in\Aut^*(X,\mu)$,
$$d_u(S,T):=\mu(\{x\in X: S(x)\neq T(x)\}).$$

A key fact about $d_u$ is that the associated compatible ambidextrous metric $\tilde d_u$ on $\Aut^*(X,\mu)$ defined by 
$$\tilde d_u(S,T):=d_u(S,T)+d_u(S\inv,T\inv)$$
is complete. The uniform topology is however not separable. The action of the circle group on itself by translation allows us to see the circle group as a discrete subgroup of $\Aut^*(X,\mu)$ for the uniform topology.

\begin{rmk}The group $\Aut^*(X,\mu)$ nevertheless admits a natural Polish group topology called the weak topology, see \textit{e.g.} \cite{danilenko2011ergodic}. 
\end{rmk}

Nevertheless, there are interesting separable $d_u$-closed subgroups of $\Aut^*(X,\mu)$. For a countable Borel equivalence relation $\mathcal R$ on a standard probability space $(X,\mu)$, the full group $[\mathcal R]$ consists of all $T\in\Aut^*(X,\mu)$ such that for almost all $x\in X$, one has $(x,T(x))\in\mathcal R$. The group $[\mathcal R]$ is closed for the uniform topology. Indeed, if $T$ does not belong to $[\mathcal R]$, let  $\epsilon:=\mu(\{x\in X: (x,T(x))\not\in\mathcal R\})>0$. For $S$ so that $\tilde{d}_u(S,T)<\frac{\epsilon}{2}$, the map $S$ will coincide with $T$ on a positive measure subset of $\{x\in X: (x,T(x))\not\in\mathcal R\}$, hence $S$ does not belong to $[\mathcal R]$. The group $[\mathcal R]$ is thus closed in the uniform topology, and since $\tilde d_u$ is complete, we conclude that $([\mathcal R],d_u)$ is a complete metric space. 

We now show that group $([\mathcal R],d_u)$ is also separable. First define a measure $M$ on the Borel equivalence relation $\mathcal R$ by integrating the counting measure on fibers: for all Borel $A\subseteq \mathcal R$, we put
$$M(A):=\int_X\abs{A_x}d\mu(x),$$
where $A_x:=\{y\in X: (x,y)\in A\}$. By the Lusin-Novikov theorem, \cite[(18.10)]{K95}, $\mathcal R$ can be written as the union of graphs of functions $X\to X$. Since each of these graphs has measure $1$, the measure $M$ is $\sigma$-finite. 

Define the measure algebra of $(\mathcal R, M)$ to be the algebra of finite measure Borel subsets of $\mathcal R$ with two such sets being identified if they coincide up to measure zero. Since $(\mathcal R,M)$ is a standard $\sigma$-finite space, its measure algebra is separable when equipped with the metric $d_M(A,B):=M(A\bigtriangleup B)$. Further, the map which associates to $T\in[\mathcal R]$ its graph seen as an element of the measure algebra of $(\mathcal R,M)$ multiplies distances by two. It now follows that $[\mathcal R]$ is separable. 

Full groups arise naturally as follows: let a countable group $\Gamma$ act in a non-singular manner on $(X,\mu)$ and consider the countable Borel equivalence relation 
\[
\mathcal R_\Gamma:=\{(x,\gamma\cdot x): x\in X\text{ and } \gamma\in\Gamma\}.
\]
Since the action is non-singular, $\Gamma$ is a subgroup of $[\mathcal R_\Gamma]$. Assuming the action is also ergodic, meaning that every Borel $\Gamma$-invariant subset of $X$ has measure either zero or one, one has the following trichotomy:
\begin{enumerate}[$\bullet$ ]
\item (type II$_1$) $\Gamma$ preserves a  Borel probability measure $\nu$ equivalent to $\mu$, and so all the elements of $[\mathcal R_\Gamma]$ also preserve $\nu$. 
\item (type II$_\infty$) $\Gamma$ preserves a $\sigma$-finite infinite measure $\nu$ equivalent to $\mu$, and so all the elements of $[\mathcal R_{\Gamma}]$ also preserve $\nu$.
\item (type III) $\Gamma$ preserves no $\sigma$-finite measure equivalent to $\mu$.
\end{enumerate}

In the type II$_1$ and the type III case, a result of Eigen \cite{MR654590} states that the full group $[\mathcal R_\Gamma]$ is simple. However, in the type II$_\infty$ case, letting $\nu$ be the $\sigma$-finite measure which is preserved, the group $[\mathcal R_\Gamma]_f$ of elements of $[\mathcal R_\Gamma]$ whose support has finite $\nu$-measure is a Borel normal subgroup of $[\mathcal R_\Gamma]$. One can then equip $[\mathcal R_\Gamma]_f$ with the metric $d_\nu$ defined by $d_\nu(S,T):=\nu(\{x\in X: S(x)\neq T(x)\})$. The inclusion $[\mathcal R_\Gamma]_f\injects [\mathcal R_\Gamma]$ is a normal compression as shown by the following proposition.

\begin{prop}
Let $\Gamma$ be a countable group acting on the $\sigma$-finite infinite measure space $(X,\nu)$ in a measure-preserving manner. Then $([\mathcal R_\Gamma]_f,d_\nu)$ is a Polish dense subgroup of $([\mathcal R_\Gamma],d_u)$. 
\end{prop}
\begin{proof}
We first show that $[\mathcal R_\Gamma]_f$ is dense. By Rohlin's lemma (see \cite[Theorem 7.7]{MR0435350}), we only need to show that every periodic element of $[\mathcal R_\Gamma]$ can be approximated by elements of $[\mathcal R_\Gamma]_f$. For every such element $T$, there exists an increasing exhausting sequence of finite measure $T$-invariant sets $X_n$ such that $\bigcup_{n\in\Nb} X_n=X$, which yields the result by considering $T_n$ defined by 
\begin{equation*}T_n(x):=\left\{\begin{array}{cc}T(x) & \text{ if }x\in X_n \\x & \text{otherwise}.\end{array}\right.\end{equation*}

Let us now show that $[\mathcal R_\Gamma]_f$ is Polish. To this end, we first prove that the metric $d_\nu$ is complete. Let $(T_n)_{n\in \Nb}$ be a $d_\nu$-Cauchy sequence of elements of $[\mathcal R_\Gamma]_f$. Up to taking a subsequence, we may assume that for all $n\in\Nb$, $d_{\nu}(T_n,T_{n+1})<\frac 1{2^n}$. Since $\sum_{n\in\Nb}\frac 1{2^n}<+\infty$, the Borel--Cantelli lemma  implies that for almost all $x\in X$, there is $N(x)\in\Nb$ such that for all $n\ge N(x)$ we have $T_n(x)=T_{N(x)}(x)$. For all such $x$, we set $T(x):=T_{N(x)}(x)$, and it is easily checked that $T\in [\mathcal R_\Gamma]_f$ and that $d_\nu(T_n,T)\to 0$. 

The proof of the separability of $d_\nu$ follows the same lines as that for $([\mathcal R_\Gamma],d_u)$. We equip $\mathcal R_\Gamma$ with a $\sigma$-finite measure $M_\nu$ defined by $M_\nu(A):=\int_X\abs{A_x}d\nu(x)$. The group $[\mathcal R_\Gamma]_f$ then naturally embeds in to the separable measure algebra of $(\mathcal R_\Gamma,M_\nu)$, and thus,  it is separable. We conclude that $[\mathcal R_\Gamma]_f$ is Polish. 
\end{proof}
	
Our last example was explored in depth by Kechris in his monograph \cite{MR2583950}; we direct the reader there for details. Let $\mathcal R$ be a type II$_1$ ergodic equivalence relation on $(X,\mu)$ so that the preserved measure is actually $\mu$. Form the metric $d_u$ on $[\mathcal R]$ as discussed above and let $N(\mathcal R)$ be the automorphism group of $[\mathcal R]$.

By Dye's reconstruction theorem, every automorphism of the full group $[\mathcal R]$ is the conjugation by some $T\in\Aut(X,\mu)$. In particular, every automorphism of $[\mathcal R]$ is an isometry, so $N(\mathcal R)$ is a subgroup of the Polish group of isometries of $([\mathcal R], d_u)$. Since being a group automorphism is a closed condition, we see that $N(\mathcal R)$ is a closed subgroup of $\mathrm{Isom}([\mathcal R])$, hence it is a Polish group. However, the fact that $[\mathcal R]\normal N(\mathcal R)$ is a dense normal inclusion depends very much on $\mathcal R$. The following two phenomena can occur:

\begin{enumerate}[(a) ] 
\item The group $[\mathcal R]$ is a closed subgroup of $N(\mathcal R)$. In this case, $N(\mathcal R)/[\mathcal R]$ is countable, and $[\mathcal R]$ is the connected component of the identity in $N(\mathcal R)$. It can even happen that $[\mathcal R]=N(\mathcal R)$.

\item The group $[\mathcal R]$ is not closed in $N(\mathcal R)$. If $\mathcal R$ is also hyperfinite, then $[\mathcal R]$ is dense in $N(\mathcal R)$, and we have an example of a normal compression $[\mathcal R]\rightarrow N(\mathcal R)$. We remark that we do not know any other equivalence relations for which this occurs.
\end{enumerate}



\bibliographystyle{bibgen}
\bibliography{biblio}

\end{document}